\documentclass[11pt]{amsart}

\usepackage{mathtools}
\usepackage[english]{babel}
\usepackage{cancel}
\usepackage{xcolor}
\usepackage[most]{tcolorbox}
\usepackage{float}
\usepackage[a4paper,top=35mm,bottom=30mm,left=30mm,right=30mm]{geometry}

\usepackage{latexsym,amsmath,amsthm,amssymb,amsfonts,mathrsfs,enumerate}
\usepackage{xcolor}
\usepackage{color}
\usepackage[normalem]{ulem}
\usepackage{cancel} 
\usepackage{graphicx}
\usepackage[colorlinks=true, allcolors=blue]{hyperref}

\usepackage{tikz}
\usepackage{pgfplots}
\pgfplotsset{compat=1.18}

\newtcolorbox{greybox}{colframe=black,colback=gray!10,coltext=black,boxrule=0.5mm,arc=2mm,width=\linewidth,left=6pt,right=6pt,top=6pt,bottom=6pt,breakable,enhanced
}

\definecolor{caca}{rgb}{0.137,0.3,0.541}
\newtheorem*{thmA}{Theorem A}
\newcommand{\refThmA}{A\space}
\newcommand{\refThmAnospace}{A}
\newtheorem*{thmB}{Theorem B}
\newcommand{\refThmB}{B\space}
\newcommand{\refThmBnospace}{B}

\newtheorem*{thmC}{Theorem C}

\newcommand{\refThmCnospace}{C}
\newtheorem*{thmD}{Theorem D}
\newcommand{\refThmD}{D\space}
\newcommand{\refThmDnospace}{D}

\newtheorem*{thmE}{Theorem E}
\newcommand{\refThmE}{E\space}
\newtheorem*{thmF}{Theorem F}
\newcommand{\refThmF}{F\space}
\newcommand{\refThmFnospace}{F}

\newtheorem*{thmD'}{Theorem D'}

\newtheorem{thm}{Theorem}[section]

\newtheorem{prop}[thm]{Proposition}
\newtheorem{cor}[thm]{Corollary}
\newtheorem*{corstar}{Corollary}

\newtheorem{lem}[thm]{Lemma}
\newtheorem{Claim}[thm]{Claim}

\newtheorem{dfn}[thm]{Definition}
\newtheorem{rem}[thm]{Remark}
\newtheorem{remark}[thm]{Remark}

\newcommand{\R}{\mathbb{R}}
\newcommand{\RR}{\mathbb{R}}
\newcommand{\area}{\operatorname{Area}}
\newcommand{\BB}{\mathbf{}{B}}
\newcommand{\ee}{\mathbf{e}}
\newcommand{\oomega}{\boldsymbol{\omega}}
\newcommand{\vv}{\mathbf{v}}
\newcommand{\Hh}{\mathbf{\mathscr{H}}}

\newcommand{\DriftL}{\Delta + \ee_3\cdot\nabla}
\newcommand{\Wedge}{\mathsf{V}^+}

\DeclareMathOperator{\width}{width}
\DeclareMathOperator{\genus}{genus}
\DeclareMathOperator{\Hess}{Hess}
\DeclareMathOperator{\dist}{dist} 
\DeclareMathOperator{\diverg}{div} 
\DeclareMathOperator{\osc}{osc}
\DeclareMathOperator{\Graph}{Graph}

\title[Uniqueness of tangent planes and  (non-)removable singularities]{Uniqueness of tangent planes\\ and (non-)removable singularities \\at infinity for collapsed translators}

\date{\today}

\author[E.S. Gama]{Eddygledson S. Gama}
\address[Gama]{
  Departamento de Matem\'atica,
   Centro de Ci\^encias Exatas e da Natureza, Universidade Federal Pernambuco, Recife-PE, 50670-901, Brazil.
}
\email{eddygledson.gama@ufpe.br}

\author[F. Mart\'in]{Francisco Mart\'\i{}n}
\address[Mart\'in]{
  Departamento de Geometr\'\i{}a y Topolog\'\i{}a, 
  Instituto de Matem\'aticas IMAG,
  Universidad de Granada,
  18071 Granada, Spain.
}
\email{fmartin@ugr.es}

\author[N.M. M{\o}ller]{Niels M. M{\o}ller}
\address[M{\o}ller]{
  Copenhagen Centre for Geometry and Topology (GeoTop), Department of Mathematical Sciences,
  University of Copenhagen, 
  DK-2100 Copenhagen, Denmark.
}
\email{nmoller@math.ku.dk}


\begin{document}

\newgeometry{top=20mm, left=30mm, right=30mm, bottom=30mm}

\begin{abstract}
We show that mean curvature flow translators may exhibit non-removable
singularities at infinity, due to jump discontinuities in their asymptotic profiles, and that
oscillation can persist so as to yield a continuum of subsequential limit tangent planes. Nonetheless, we prove that as time $t\to\pm \infty$, any finite entropy, finite genus, embedded, collapsed translating soliton (not necessarily mean convex) in $\mathbb{R}^3$ converges to a uniquely determined collection of planes.

This requires global analysis of quasilinear soliton equations with non-perturbative drifts, which we analyze via sharp non-standard elliptic decay estimates for the drift Laplacian, implying improvements on the Evans-Spruck and Ecker-Huisken estimates in the soliton setting, and exploiting a link from potential theory of the Yukawa equation to heat flows with $L^\infty$-data on non-compact slice curves of these solitons.

The structure theorem follows: such solitons decompose at infinity into standard regions asymptotic to planes or grim reaper cylinders. As one application, we classify collapsed translators of entropy two with empty limits as $t\to +\infty$.

\end{abstract}

\maketitle

\setcounter{tocdepth}{1}

\tableofcontents

\section{Introduction}
  Uniqueness of tangent flows is a key issue in the study of geometric evolution equations. This means that whether the limiting behavior of the flow, for example near a singularity, is independent of the subsequences chosen. Such properties are notoriously difficult to establish in general, in particular in the case of translating solitons, used to model Type II singularities in the mean curvature flow, which can exhibit complicated asymptotic behavior. Recall that a surface \( \Sigma^2 \subset \mathbb{R}^{3} \) is a translating soliton in the direction $\mathbf{v}\in\mathbb{R}^3$ if the $1$-parameter family of translates of the surface is a mean curvature flow. This means that, for $\mathbf{H}$ the mean curvature vector,
\[
\R \ni t \mapsto\Sigma_t := \Sigma + t\mathbf{v}\quad \text{solves}\quad (\partial_tx)^\perp = \mathbf{H}.
\]

In this paper, we resolve the problem of limits at infinity for collapsed translating solitons, i.e. those contained in slabs of finite width, proving that convergence is independent of the subsequences chosen, as well as establishing an asymptotic structure theorem. We emphasize that our results do not rely on curvature assumptions such as mean convexity.

Our first goal is to show that the flow uniquely converges at infinite times to a fixed finite union of vertical planes (see Definition \ref{def:entropy} for the definition of the entropy):

\begin{thmA}[Existence of limits as time $t\to\pm\infty$]  
     Let \(\Sigma^2 \subset \mathbb{R}^3\) be a translating soliton which is complete, collapsed, embedded, of finite genus and finite entropy. Then, as \( t \to \mp\infty \), the backwards and forwards flows \( \Sigma + t \ee_3 \) converge locally smoothly to a finite union of vertical planes, with multiplicities determined by the large-scale structure of \(\Sigma\):

\begin{itemize}
    \item[(1)]  As \( t \to -\infty \), \( \Sigma + t \ee_3 \) converges to a finite collection of vertical planes, whose number, counted with multiplicity, equals the entropy \( \lambda(\Sigma) \).
    \item[(2)] As \( t \to +\infty \), \( \Sigma + t \ee_3 \) converges to another finite (possibly empty) collection of vertical planes, whose number, counted with multiplicity, corresponds to a structural invariant of \(\Sigma\) related to its asymptotic behavior.
\end{itemize}
\end{thmA}

While Theorem \refThmA describes all possible asymptotics with respect to time, viewing instead $\Sigma$ as a static object, there are many other directions to escape to infinities along the surface, which we study in our second main theorem (see Figure \ref{fig:Xi1}).

\begin{thmB}[Removable singularities at $\infty$]
Let $\Sigma$ be as in Theorem \refThmAnospace.
Limits along all rays exist in directions $\mathbb{S}^1:=\{(0,x_2,x_3) \in \R^3  \: : \: x_2^2+x_3^2=1\}$,
\begin{equation}
\forall \oomega\in\mathbb{S}^1:\quad
\exists\lim_{r \nearrow +\infty}\left(\Sigma - r\oomega+\digamma_{\oomega}(r) \ee_3\right):=\Sigma_\infty(\oomega),
\end{equation}
where $\Sigma_\infty(\oomega)$ is a (possibly empty) finite multiset of planes and (possibly tilted) grim reaper cylinders with multiplicity and $\digamma_{\oomega}(r)$ a smooth remainder with $\displaystyle\lim_{r\to+\infty}\digamma'_{\oomega}(r)=0$. Hence the limit map \[\mathbb{S}^1\ni\oomega\mapsto \Sigma_\infty(\oomega)\] is a well-defined multiset-valued map.

There exists a (possibly empty) finite subset $\mathscr{R}(\Sigma)\subset\mathbb{S}^1\setminus\{\pm\ee_3\}$ such that:
\begin{itemize}
    \item[i)] $\oomega\:\mapsto\:\Sigma_\infty(\oomega)$ is locally constant for $\oomega\in\mathbb{S}^1\setminus\left(\mathscr{R}\:\cup\:\right\{-\ee_3\})$.
    \item[ii)] $\digamma_{\oomega}\equiv0$ , $\forall\oomega\in\mathbb{S}^1\setminus\mathscr{R}(\Sigma).$
\end{itemize}
\end{thmB}

Our third main result provides a geometrically sharp structure theorem for the solitons, demonstrating that all asymptotic regions are standard model pieces. This justifies, in a rigorous sense, the intuition that every such example can be viewed as assembled from planes and (possibly tilted) grim reaper cylinders.

\begin{thmC}[Asymptotic structure theorem]
In Theorem B, the following asymptotic properties hold:
\begin{itemize}
        \item[(i)] $\forall\oomega\in \mathbb{S}^1\setminus\mathscr{R}(\Sigma):\:$ the limit $\Sigma_\infty(\oomega)$ is either empty or consists of planes.
        
    \item[(ii)] $\forall\oomega\in\mathscr{R}(\Sigma):\:$ the limit $\Sigma_\infty(\oomega)$ consists of at least one (possibly tilted) grim reaper cylinders tangent to $\oomega$ and a (possibly empty) collection of planes. 
\end{itemize}
In both situations, the limits (planes and grim reapers) may appear with some multiplicity.
\end{thmC}

The quasilinear equation involved has a drift term which does not decay at infinity, and admits solutions which exhibit a very complicated asymptotic behavior. We show:

\begin{thmD}[Non-removable singularities]
The removable singularities property in Theorem B is sharp: The map $\omega\mapsto\Sigma_\infty(\omega)$ is discontinuous at all points of $\mathscr{R}(\Sigma)$ and may exhibit jump discontinuities at $\omega_0 = -\ee_3$, with respect to Hausdorff distance. The latter behavior occurs e.g. when $\Sigma$ is one of the so-called ``pitchfork'' solitons.

Allowing $\partial\Sigma\neq\emptyset$, more complicated non-removability of singularities can occur:\\ There exist collapsed, complete with boundary $\Sigma$, which exhibit persistent oscillations at infinity, so that as time $t\nearrow +\infty$, subsequential limits of $\Sigma + t\ee_3$ exhaust a whole continuum of limiting tangent planes. 
\end{thmD}

\begin{figure}
     \centering     \includegraphics[trim=0 150 0 0, clip, width=0.50\linewidth]{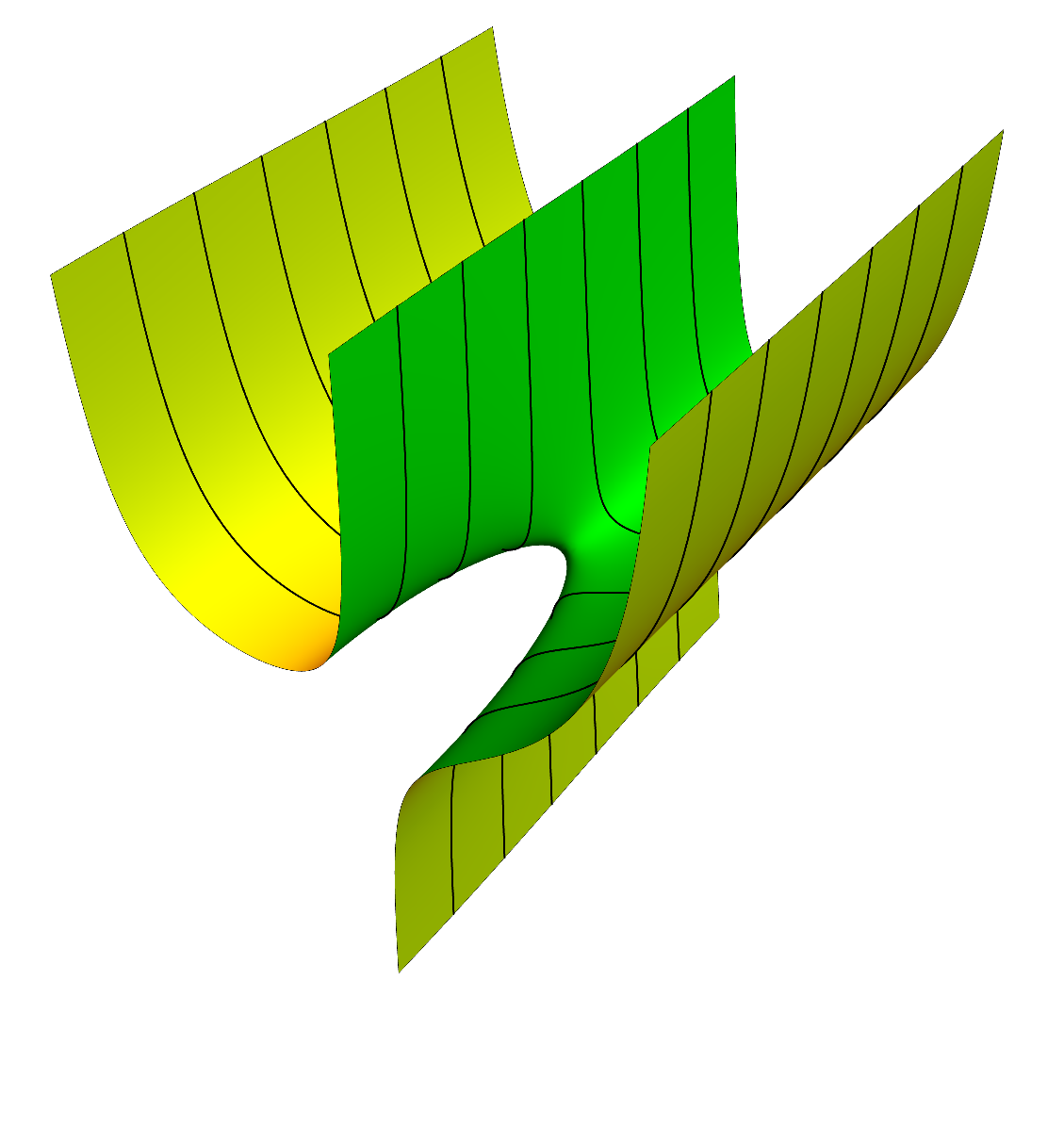}
     \caption{\small Pitchfork mean curvature translating soliton in $\mathbb{R}^3$. (See Theorem \refThmBnospace.)}
     \label{fig:pitchfork}
 \end{figure}

\subsection{Proof ideas}\label{subsec:PDE_tech}

The above theorems will not follow from application of a standard framework, due to the large drift quasilinear nature of the PDE involved, and a major innovation of our approach is in the derivation and application of refined geometric estimates in that setting. These new estimates will turn out to be strong enough to even allow bypassing the development of a larger functional-analytic setup and tools, such as weighted Sobolev spaces or special weighted (cone) H\"{o}lder spaces, that 
one might have expected to play a fundamental role in proving results of this flavor. Indeed such setups have been essential to the previous studies of self-shrinkers, which are asymptotical to cones with a smooth compact link \cite{Otis-Felix-Choi-Mantou, Otis-Felix, KKM, KM, NielsThesis}.

Instead, we first leverage ideas from classical elliptic PDE techniques reminiscent of removable singularity problems in exterior domains, previously known for perturbations at infinity of the Laplacian, as in \cite{Meyers-Serrin}, including in the case of decaying drift terms, which we then extend to the more difficult case at hand of a drift Laplacian with non-perturbative drift term. We do this via non-standard anisotropic elliptic estimates which are supported on directed sausage-shaped regions, rather than on round balls, obtaining sharp decay rates which will differ in strength depending on the global geometry and orientation of the domains.

We then push these results to the quasilinear translating soliton equation, obtaining stronger decay than in the classical gradient and higher derivative estimates of Evans-Spruck and Ecker-Huisken, allowing us to control the soliton's geometry at large scales.

Finally, we then combine the information gained from these steps with refined calculations in potential theory for the Yukawa equation (also known as the modified Helmholtz equation, or the screened Poisson equation), from which we find and exploit an interesting link to the global behaviors of certain 1-dimensional heat flows, with $L^\infty$ data on non-compact boundary curves, which we initiate from slice curves of the original soliton.

These new ideas allow us to prove uniqueness of the asymptotic configuration at infinity for a translating soliton in a fairly direct and elementary way. In particular, we avoid tools such as \L{}ojasiewicz inequalities, Almgren-type frequency monotonicity, and three-annulus theorems \cite{Choi,CM-fre,Lotay,Simon83}, which have played a central role in the analysis of self-shrinkers \cite{Colding-Minicozzi, CM25}. See also the more recent PDE–ODI methods developed for ancient cylindrical flows \cite{Bamler-Lai1,Bamler-Lai2}, and earlier work in the noncollapsed setting for mean-convex ancient solutions \cite{CDDHS}. These techniques seem harder to adapt to (not necessarily mean-convex) collapsed translators: the relevant quasi-linear equation has different symmetry and homogeneity features, and its solutions can display markedly different asymptotic geometries that require a separate analysis.

Thus, while the self-shrinkers' asymptotic cones have a compact link and are endowed with Laplacians with inwards drift terms, for the translating solitons considered in this paper we instead need to deal with two kinds of ways to escape to infinity along wedge-shaped and half-planar domains involving now additionally two distinct types of drift terms, pointing upwards and downwards, respectively.

\subsection{Geometric applications.}

One of the key first consequences of our main theorems is the classification of collapsed translators in half-slabs and collapsed translators with finite topology and exactly two wings of grim-reaper type.  

\begin{thmE}
  The unique complete, simply connected translating solitons of finite width and finite entropy in $\R^3$ satisfying the multiplicity-one property at infinity and lying in a half-slab are the (possibly tilted) grim reaper surfaces, and the $\Delta$-wings.
\end{thmE}

This result can be viewed as a variant of the half-space theorem for minimal surfaces proved by Hoffman and Meeks \cite{Hoffman-Meeks}. Unlike the classical half-space theorem, our proof crucially relies on the topological assumption of simply connectivity. Without this constraint, we cannot a priori rule out that a counterexample could potentially be constructed.

An interesting consequence of Theorem \refThmE is given by the following result:

\begin{corstar}
    The unique complete translating graphs (in arbitrary direction) of finite width in $\R^3$ lying in a half-slab are the (possibly tilted) grim reaper surfaces, and $\Delta$-wings.
\end{corstar}

Our final application involves the classification of collapsed translators with entropy two. A natural first step in classifying collapsed translators with finite genus and finite topology is to focus on those with the lowest possible entropy. Since the only translators with entropy equal to unity are trivially the vertical planes, the first non-trivial case to consider is the classification of translators with entropy equal to 2. In \cite{Chini}, Chini showed that, under the assumption of being simply connected, the only such examples are (possibly tilted) grim reapers and $\Delta$-wings. Finally, we are able to obtain a characterization of collapsed translators of entropy 2 and finite genus. To be more precise, we prove: 
\begin{thmF}\label{thmCC}
    Let $\Sigma$ be  a complete, embedded translator with finite width, finite genus and $\lambda(\Sigma)=2$. Assume that  the limit, as $t \to +\infty$, of $\Sigma+t \ee_3$ is the empty set. Then $\Sigma$ is either a (possibly tilted) grim reaper or a $\Delta$-wing. In particular, if the width of $\Sigma$ is $\pi$, then $\Sigma$ is the standard grim reaper cylinder.
\end{thmF}

This theorem's final statement, concerning rigidity of the grim reaper cylinder in the width $\pi$ case, was obtained under the stronger condition of $\Sigma\subseteq \{x_3 \geq a\}$ in \cite{IRM}, using quite different methods.


Theorem \refThmF implies that it is not possible to add ``wormhole'' handles to either a (possibly tilted) grim reaper or to a $\Delta$-wing.

Until recently, it was an open question whether nontrivial examples could exist with just planar wings (in \cite{HMW-SpruckVol} conjectured not to). If such existed, they would be expected to look geometrically like two perturbed vertical planes glued by small necks. From the following small observation/calculation, the existence was unlikely:

\begin{prop}[Doubling obstruction]\label{prop:Nicos}
Vertical planes as translators fail to satisfy the doubling construction condition of Kapouleas \cite{NicosSurvey}. Namely, at any point $p\in\R^3$, evaluating with respect to the Ilmanen metric, denoting $M =(\R^3,e^{x_3}\delta_{ij})$, for $\Sigma$ any vertical plane:
\begin{equation}\label{NicosSaysNo}
\left(\operatorname{Ric}_M(\nu_\Sigma,\nu_\Sigma) + |A_\Sigma|^2\right)(p) < 0.
\end{equation}
\end{prop}

The tools and methods developed in the present work as expected also form a suitable new framework in which to study such problems rigorously, as a special case. Recently Lynch and Tinaglia in \cite{Lynch-Tinaglia}, using an estimate similar to \eqref{grad_down}, which follows by interpolation from \eqref{Hess_down_improved} (see \cite{GMM-FollowUp}), proved planarity under the assumption of  strongly planar ends, which holds in the special case of finite total curvature, by Khan's result \cite{Ilyas}. In \cite{GMM-FollowUp}, among other things, we show how planarity follows in the larger class of collapsed, finite genus, finite entropy translating solitons having only planar wings, from our estimates.

\subsection{Overview of translating solitons in mean curvature flow}

Translators are especially significant in the formation of Type II singularities in mean curvature flow. Compared to Type I singularities, which are modeled by self-shrinkers and have been extensively studied and in some cases classified (e.g., generalized round cylinders as the only entropy-stable i.e. mean convex shrinkers \cite{CM12,Hu93, Whi03}), for Type II singularities a lot less is presently known and the study of them presents substantially different analytical challenges. Translators provide the primary models for Type II singularities, yet understanding in detail which surfaces this class contains remains a largely open problem, despite significant progress in recent years.

For surfaces, the basics of the class of translators is by now somewhat more well understood. The classification of graphical translators, with the bowl soliton dating back to Altschuler and Wu \cite{AW}, was recently completed by Hoffman, Ilmanen, Mart\'in, and White \cite{HIMW}, building upon the foundational work of Wang \cite{Wang}, and Spruck and Xiao \cite{Spruck-Xiao}, showing that: every such translator is either a bowl soliton, a (possibly tilted) grim reaper cylinder, or a member of the one-parameter family of \( \Delta \)-wings.

Beyond graphical examples, more intricate classes of translators have been constructed \cite{HMW-2,HMW-1,Ng,sun-wang}, including those with finite topology and entropy, often constrained within vertical slabs of finite width - these are known as collapsed translators. This class is particularly significant due to its connection with the classification of convex hulls of singularities \cite{Chini-Moller2, Chini-Moller}, making collapsed translators central objects of study in mean curvature flow.

In \cite{entropy}, Gama, Martín, and Møller therefore conducted a detailed study of the class of translating solitons for the mean curvature flow which possess finite entropy and finite genus, and are collapsed, i.e., translators which are contained in vertical slabs. A key notion introduced in this previous analysis is the concept of wings, defined by classifying the components of level sets of certain geometrically motivated Morse functions on the solitons, which describe, in topological terms, the specific asymptotic regions of these translators.

In a collapsed translator, the wings play the role of asymptotic building blocks, meaning that the solitons can be decomposed into a central region (contained in a vertical cylinder) and a collection of non-compact wings. The wings play a crucial role in understanding the global geometry and behavior of collapsed translators, as they determine how the surface extends and interacts with the surrounding space at infinity.

The wings of a collapsed translator are categorized into two distinct types based on their structures and behaviors at infinity (see Figure \ref{fig:wings}):
\begin{itemize}
\item \underline{Planar-type wings}: These wings are graphical over vertical half-planes that are parallel to the boundary planes of the slab.
\vspace{4pt}
\item \underline{Grim reaper-type wings}: Unlike planar wings, these wings are {\it bi-graphs} over certain subsets of half-planes, as happens with {\em (possibly tilted) grim reaper cylinders}.
\end{itemize}

By distinguishing between these two types, the authors provided in \cite{entropy} the initial steps toward understanding the asymptotic structure of general collapsed translators $\Sigma$, revealing how their global shapes are influenced by the interactions between their entropy invariants $\lambda(\Sigma)\in\mathbb{N}$, genera, and the geometric constraints imposed by the vertical slab condition (i.e. collapsedness).

Missing crucially from the analysis in \cite{entropy}, which emphasized softer techniques of topological nature, were thus the truly geometric aspects of the asymptotics, governed by the quasilinear PDE of the self-translating solitons. Taking this next step is what will lead to the main results in the present paper, Theorems \refThmAnospace{}-\refThmFnospace.

\subsection{Outline of the paper}

In Section \ref{preliminaries}, we begin by describing the geometric setup in details, including the decomposition of collapsed solitons into domains of special shapes, which will be needed for the proofs of the main theorems. This is a non-trivial step, since a crucial global aspect of our proofs will lie in the assumption that the surfaces are {\it complete and without boundary}, which cannot be relaxed. Indeed, as shown in Section \ref{sec:counterexamples}, in fact, there do exist collapsed translators with boundary for which the limit \(\Sigma + t \ee_3\) fails to exist as \( t \to +\infty \). This highlights the necessity of the completeness condition in ensuring the asymptotic behavior described in Theorems \refThmAnospace--\refThmCnospace.

In Section \ref{sec:PDEs_drift}, we begin the analysis of the partial differential equation in question, on the special domains from Section \ref{preliminaries}, to show a priori derivative decay estimates for this quasilinear PDE with a drift term which is not decaying at infinity, as already explained in Section \ref{subsec:PDE_tech}.

Section \ref{sec:Rem-Sin} is devoted to proving the new removable singularities theorem at infinity which gives this paper its title. We obtain an interesting dichotomy: as \( t \to -\infty \), the removable singularities theorem exhibits a more local nature in that knowing $L^\infty$ on just an upward-directed wedge is in itself enough for existence of the true limit (Proposition \ref{limit_line_up}), which is in that case even exponentially fast (Theorem \ref{thm:exponential}). However, to establish for \( t \to +\infty \) a similar result (Proposition \ref{inhomog_down}), a global approach is required, making use of the existence of upwards/sideways limits (Proposition \ref{sideways_limits}, Proposition \ref{cor:remov_sing}). 

Section \ref{sec:counterexamples} contains the counterexamples which prove that the global approach to the above proofs (e.g. Proposition \ref{inhomog_down}) is in fact strictly necessary.

Theorem \refThmA is established in Section \ref{sec:Uniqueness} (Theorem \ref{thm-A}) as a culmination of the key results obtained in Sections \ref{sec:PDEs_drift} and \ref{sec:Rem-Sin}.

Sections \ref{sec:asym-wings} through \ref{sec:two-clas} contain the paper's core geometric applications of the preceding analytical results. In Section \ref{sec:asym-wings}
, we turn our attention to the asymptotic behavior of the translating 
soliton along directions other than the vertical ones, that is, distinct from the canonical translation direction $\mathbf{e}_3$ and its negative. We carry out a detailed directional analysis and establish that, away from a finite exceptional set of directions, the translator converges to a union of vertical planes. In these exceptional directions $\mathscr{R}(\Sigma)\subset \mathbb{S}^1\setminus\{\pm\ee_3\}$, as in Theorem \refThmBnospace, we observe the emergence of (possibly tilted) grim reaper cylinders. This refined understanding of the geometry at infinity culminates in the proofs of Theorems B, C, and D, which together provide a comprehensive description of the large-scale profiles of collapsed translators.

In Section \ref{sec:collapsed}, we situate our family of collapsed translators within a broader and recently introduced framework: the natural class of finite-type translators studied by Hoffman, Martín, and White in \cite{HMW-A}. This connection is not merely formal but carries substantial geometric significance. In particular, it enables us to invoke the Morse–Radó theory for translators, developed in \cite{morserado}, which provides powerful topological and variational tools to study the global structure of translating solitons, as in \cite{entropy}. These techniques again play a key role in the classification results and rigidity theorems that we present in Section \ref{sec:half-slab} (Theorem E) and Section \ref{sec:two-clas} (Theorem F).

\subsection*{Acknowledgments}
The authors would like to express their gratitude to Rafe Mazzeo for many engaging discussions during intermediate stages of writing this paper, which provided an opportunity to explore several different ideas which in the end helped solidify the strategies that we follow. The authors are also profoundly thankful to David Hoffman and Brian White for their thoughtful suggestions and corrections. Finally, we would also like to thank Alberto Enciso for his comments and suggestions regarding this work.

E.S. Gama was partially supported by CNPq/Brazil Grant: 313412/2025-1. F. Mart\'in was partially supported by the grant PID2024-156031NB-I00 and by the  IMAG–Maria de Maeztu grant CEX2020-001105-M, both funded by MICIU/AEI/10.13039/501100011033. This material is based upon work supported by the National Science Foundation Grant No. DMS-1928930, while F. Martín was in residence at the Simons Laufer Mathematical Sciences Institute (formerly MSRI) in Berkeley, CA, during the Fall 2024 semester. N.M. M\o{}ller thanks U Hamburg for its hospitality. N.M. M\o{}ller was partially supported by DFF Sapere Aude 7027-00110B, by Carlsberg Foundation Semper Ardens CF21-0680 and by CPH-GEOTOP-DNRF151.

\section{Geometric setup and preliminaries} 
\label{preliminaries}
This section is devoted to describing the background and notation that we will be using in the rest of the paper. The decomposition of collapsed translators in special domains that we are going to introduce in this section will be crucial in the next sections.

\subsection{Translating solitons of finite entropy} \label{sector}
A surface \( \Sigma^2 \subset \mathbb{R}^{3} \) is a translating soliton if and only if its mean curvature vector \(\vec{H}\) everywhere satisfies the equation
\begin{equation}\label{TSE-geo}
\vec{H} = \mathbf{v}^\perp,
\end{equation}
for some non-zero vector \( \mathbf{v} \in \mathbb{R}^{3} \), which we fix as \( \mathbf{v} = \ee_{3} \). 

The concept of entropy of a hypersurface in Euclidean space was introduced by Magni-Mantegazza and Colding-Minicozzi \cite{CM12,MM08}.
Let $\Sigma^2\subset\R^3$ be a surface. Given $x_0\in\R^3$ and $s_0>0$, we define 
\[
F_{x_0,s_0}[\Sigma]=\frac{1}{4\pi s_0}\int_\Sigma e^{-\frac{|x-x_0|^2}{4s_0}}{\rm d}\Sigma.
\]

\begin{dfn} \label{def:entropy}
The entropy of $\Sigma$ is defined as: 
\begin{equation}
    \displaystyle \lambda(\Sigma)=\sup_{x_0\in\R^3,s_0>0}F_{x_0,s_0}[\Sigma]\in (0,\infty].
\end{equation}
\end{dfn}
Sun and Wang \cite{sun-wang} recently established that any translating soliton $\Sigma \subset \R^3$ with finite entropy admits a quantized entropy value: either $\lambda(\Sigma) \in \mathbb{Z}^+$, or $\lambda(\Sigma) \in \sqrt{\tfrac{2\pi}{e}} \cdot \mathbb{Z}^+$. Moreover, in the latter case, $\Sigma$ is necessarily non-collapsed. The former case was already established for collapsed translators by Gama, Mart\'{\i}n and M\o{}ller in \cite{entropy} (see Theorem \ref{lambda-formula} below.)

Let $\Sigma$ be a complete, collapsed, embedded translator with $\lambda(\Sigma)<\infty$ and ${\rm genus}(\Sigma)<\infty.$ Up to rotations around the $x_3$-axis and horizontal translations (all of them preserve the velocity vector of the flow), we can thus assume that $\Sigma$ is contained in the vertical slab given by
\[
\mathcal{S}_w:= \{(x_1,x_2,x_3) \in \R^3 \; :\; |x_1|<w\}, \quad \mbox{for some $w>0$.}
\]
\begin{dfn}[Width of collapsed translators]
If $\Sigma^2\subseteq\mathbb{R}^3$ is a complete, connected collapsed translator contained in a slab, then we define the width of $\Sigma$ as the infimum of the numbers $2w$, where $w$ is given as in the previous paragraph. 
Then we write $\width (\Sigma)=2 w.$

\end{dfn}
In \cite{entropy} we proved that, for such a translating soliton $\Sigma$, the portion of $\Sigma$ lying outside a vertical cylinder (which contains its topology) decomposes into a finite disjoint union of non-compact, simply connected translators with boundary, which we refer to as the wings of $\Sigma$. Moreover, we have demonstrated in \cite{entropy} that there are two kinds of wings: the ones that are graphs over the $(x_2,x_3)$-plane and the ones that are bi-graphs over the $(x_2,x_3)$-plane. Wings of the first kind are called {\em planar wings} and wings of the second kind are called {\em wings of grim reaper type,} see Figure \ref{fig:wings}.

\begin{figure}
     \centering
     \includegraphics[width=0.3\linewidth]{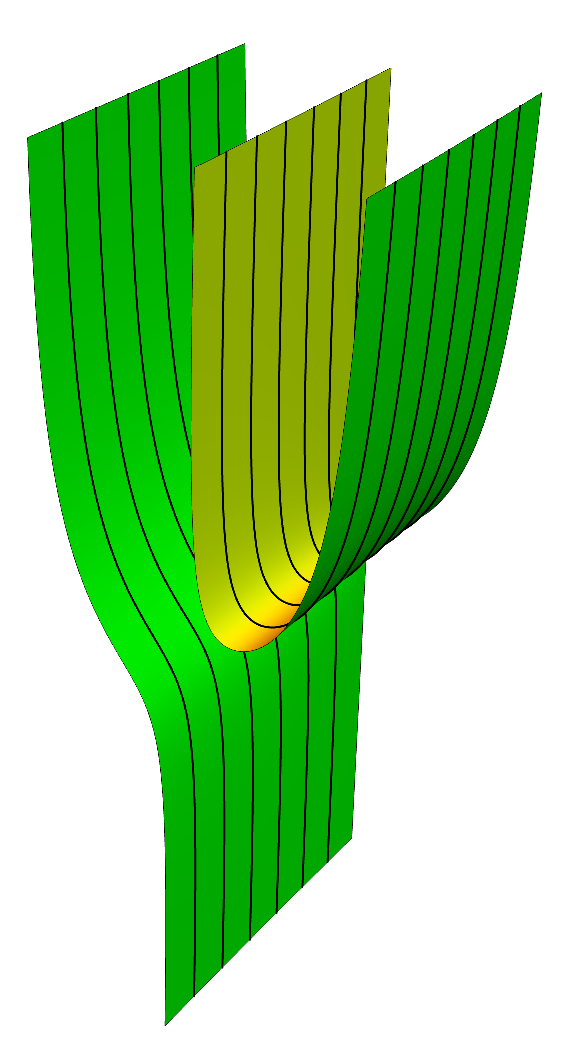}
     \caption{\small Two wings of a translator with finite entropy contained in a vertical slab. One of them is a planar wing (left) and the other one is a grim reaper-type wing (right.) }
     \label{fig:wings}
 \end{figure}
 
\begin{thm}[\cite{entropy}]
    \label{lambda-formula}
Let $\Sigma^2\subseteq\mathbb{R}^3$ be a complete, connected, collapsed embedded translator so that the entropy and the genus are finite. If $\omega_P(\Sigma)$ represents the number of planar wings and $\omega_G(\Sigma)$ represents the number of wings of grim reaper type, then $\omega_P(\Sigma)+2 \, \omega_G(\Sigma)$ is even and 
 \begin{equation}
     \lambda(\Sigma)= \frac 12 \left( \omega_P(\Sigma)+2 \, \omega_G(\Sigma)\right).
 \end{equation}
\end{thm}

Hence, we consider $\Sigma$ a properly, connected embedded translator of width $2 w$ (i.e., contained in the slab $\{|x_1| < w\}$) with finite genus and finite entropy. So, the set
$$\Sigma_1:= \Sigma \cap \{|x_2| \geq t \}, \quad \mbox{for $t>>0$},$$
is the union of $\omega_P(\Sigma)$ planar wings and $\omega_G(\Sigma)$ grim reaper wings.

\subsection{Wings of grim reaper type}\label{sec:reaperwings}
Let's recall \cite[Section 7]{entropy} that if $W$ is a right wing of grim reaper type, then  we define the function
\begin{equation}\label{def:f_W}
(t,+\infty) \ni x_2 \longmapsto f_W(x_2):=\min\{x_3 \; : \; (x_1,x_2,x_3) \in W \} 
\end{equation}
and the (not necessarily straight) curve of minima of $W$ given by
\[
\mathcal{M}_W:=\{(x_1,x_2,x_3)\in W\;:\;x_3=f_W(x_2)\}.
\]
Concerning the set $\mathcal{M}_W$, we trivially have:

\begin{lem}\label{minimal-set}
$\mathcal{M}_W\subseteq\{\langle \nu,\ee_1\rangle=0\}.$
\end{lem}

We also have the following structural information:

\begin{prop}
\textnormal{(\cite[Proposition 7.7]{entropy})}
\label{bigraph}
    The set $W \setminus \mathcal{M}_W$ is the union of two graphs over the domain $\Omega_W=\{ (0,x_2,x_3)  \, : \, x_3 > f_W(x_2)\}$. 

    Moreover, we know that there exists a constant $C_W>0$ such that
    \begin{equation}
        \limsup_{x_2 \to \pm \infty} |f_W'(x_2)| \leq C_W.
    \end{equation}
\end{prop}
\begin{figure}[htpb]
    \centering
    \includegraphics[width=0.5\linewidth]{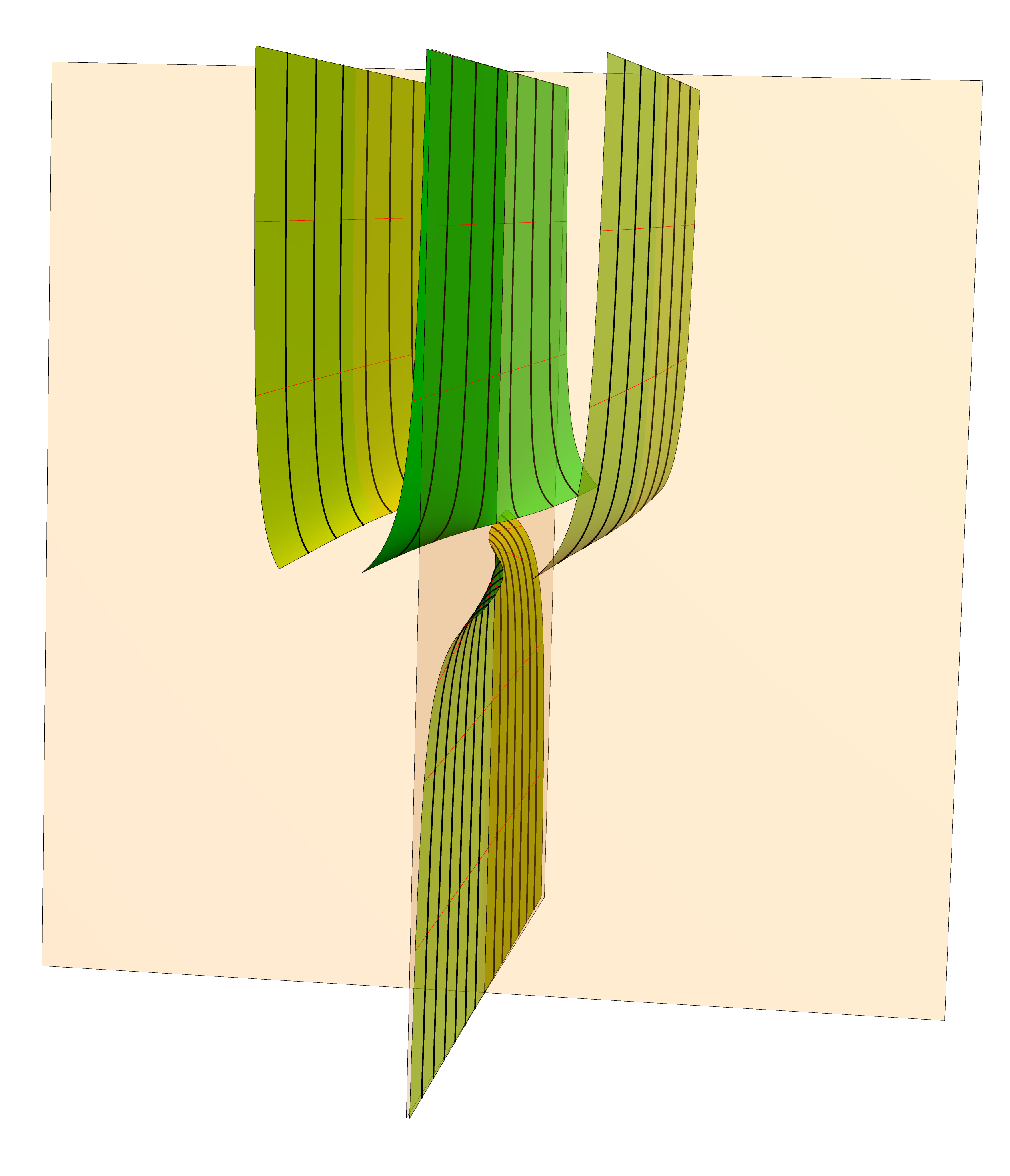}
    \caption{\small The decomposition $\widetilde \Sigma$ for a pitchfork. You can appreciate three $\ee_1$-graphs going up and one $\ee_1$-graph going down. For translators with non-trivial genus, it is necessary to remove also a solid box containing the genus of $\Sigma.$}
    \label{fig:decomp-pitchfork}
\end{figure}

If $W_1^G, \ldots, W^G_{\omega_G(\Sigma)}$ are the wings of grim reaper type, 
then let us define $$ \mathcal M:= \bigcup_{j=1}^{\omega_G(\Sigma)} \mathcal{M}_{W_j^G}.$$

\subsection{Planar wings}
 For a planar wing $W$, we proved in \cite[
Proposition 7.4]{entropy} that $W$ is the $\ee_1$-graph of a function defined over the region $\{(0,x_2,x_3) \: : \; x_2>t\}$, for a right wing, or $\{(0,x_2,x_3) \: : \; x_2<-t\}$, for a left planar wing. Then, for a planar wing $W$, we define:
 \[\mathcal{N}_{W}:= \{(x_1,x_2,x_3) \in W \; : \; x_3=0\}=W\cap \{x_3=0\}.\]

 If $W_1^P,\ldots ,W_{\omega_P(\Sigma)}^P$ are the planar wings of $\Sigma$, then we define:

\[
\mathcal{N}:= \bigcup_{j=1}^{\omega_P(\Sigma)} \mathcal{N}_{W^P_j}
\]

\begin{prop}\label{eq:decomposition-1}
     There exist $s,t>0$ large enough so that
     \[\widetilde \Sigma:= \Sigma \setminus \left(\mathcal{M} \cup \mathcal{N}\cup\{|x_2|<t,|x_3|<s\}\right) \] has a decomposition into $\ee_1$-graphs:
     \begin{equation}
    \widetilde \Sigma = \left(\Sigma_1^{\rm up} \cup \cdots \cup \Sigma_{\lambda(\Sigma)}^{\rm up}\right) \cup \left( \Sigma_1^{\rm down} \cup \cdots \cup \Sigma_{\omega_P(\Sigma)/2}^{\rm down}\right). \end{equation}

 \end{prop}
\begin{proof}
Firstly, outside the set $\{|x_2|<t\}$ each component of the $\Sigma$ is graphical, by the preceding discussion, for the same $t>0$.

Suppose, therefore, for the sake of contradiction, that no such $s>0$ exists. We may assume that it's the upper bound that fails (the other case being similar), then there would exist a sequence of points $p_j\in\Sigma\cap\{|x_2|<t\}$ with $x_3(p_j)\nearrow +\infty$ and unit normals $N_\Sigma(p_j) \parallel \ee_1$. By a small modification of the proof of \cite[Proposition  6.1]{entropy}, the sequence of solitons $\Sigma - p_j$ would $C^\infty$ subconverge on compact subsets of $\mathbb{R}^3$ to some $\Sigma_\infty$ consisting of a finite number (in fact the integer $\lambda(\Sigma)$ counted with multiplicity, by \cite[Corollary 8.5]{entropy}) of planes parallel to the slab. But we would also have $N_{\Sigma_\infty}(0) \parallel \ee_1$, a contradiction.

\end{proof}

\subsection{Graphs $\Sigma_j^{up}$: wedges going up.}

  $\Sigma_j^{\rm up}$ is a  graph over a domain in the $(x_2,x_3)$-plane, $U_j^{\rm up}$, whose boundary is a piecewise smooth curve (see Figure \ref{fig:U_i}.) Thus
\begin{multline*}
    \partial (U_j^{\rm up})=\{(x_2,s) \: : -t \leq x_2 \leq t \} \bigcup \\ \{(t, x_3) \: : \; \varphi_j(t)\leq x_3 \leq  s\} \bigcup \{(-t, x_3) \: : \; \varphi_j(-t)\leq x_3 \leq  s\} \\ \bigcup \{x_3=\varphi_j(x_2)\: : \; |x_2|>t\}.
\end{multline*}

\begin{figure}[htpb]
    \centering
    \includegraphics[width=0.42\linewidth]{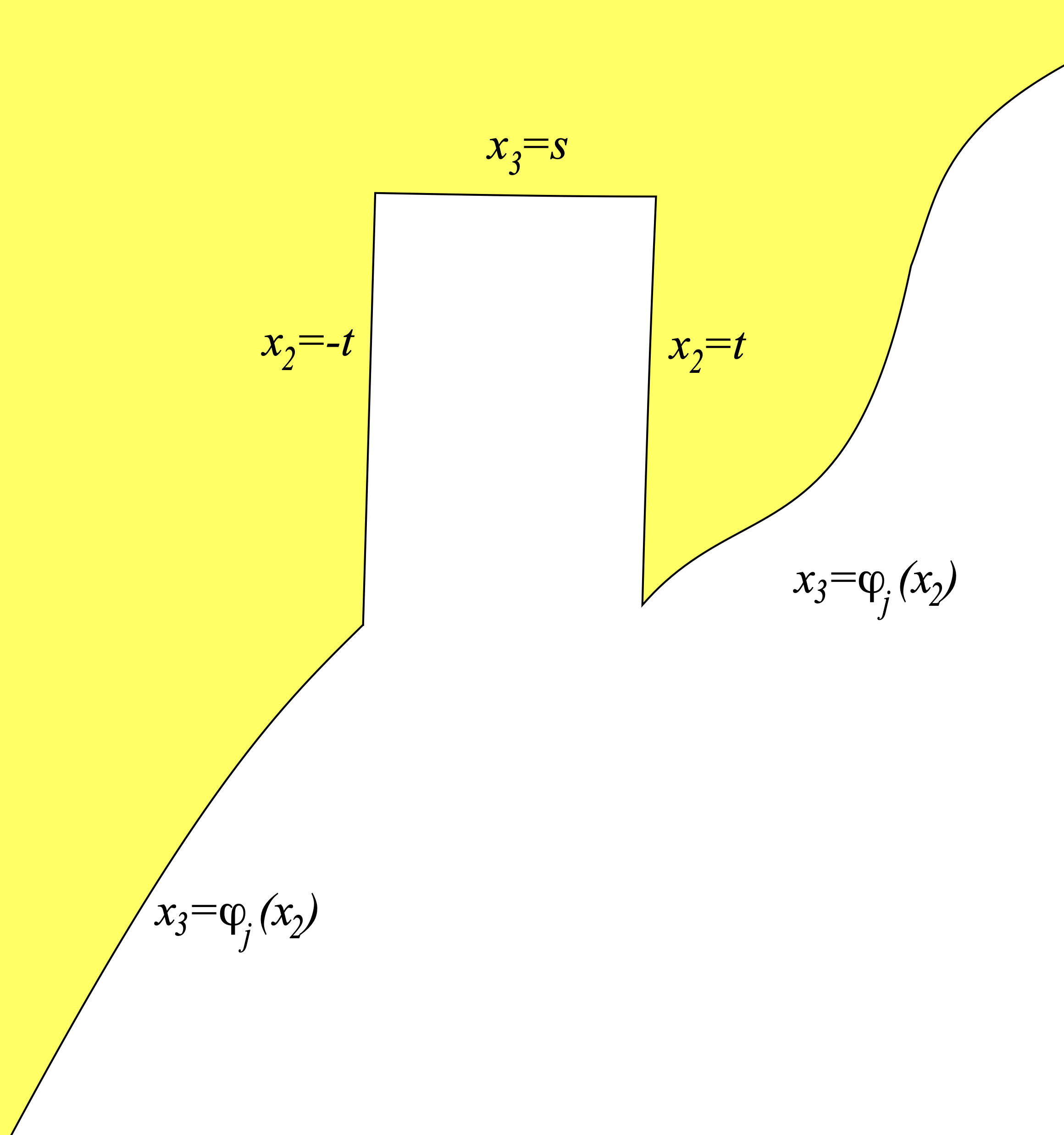}
    \includegraphics[width=0.4\linewidth]{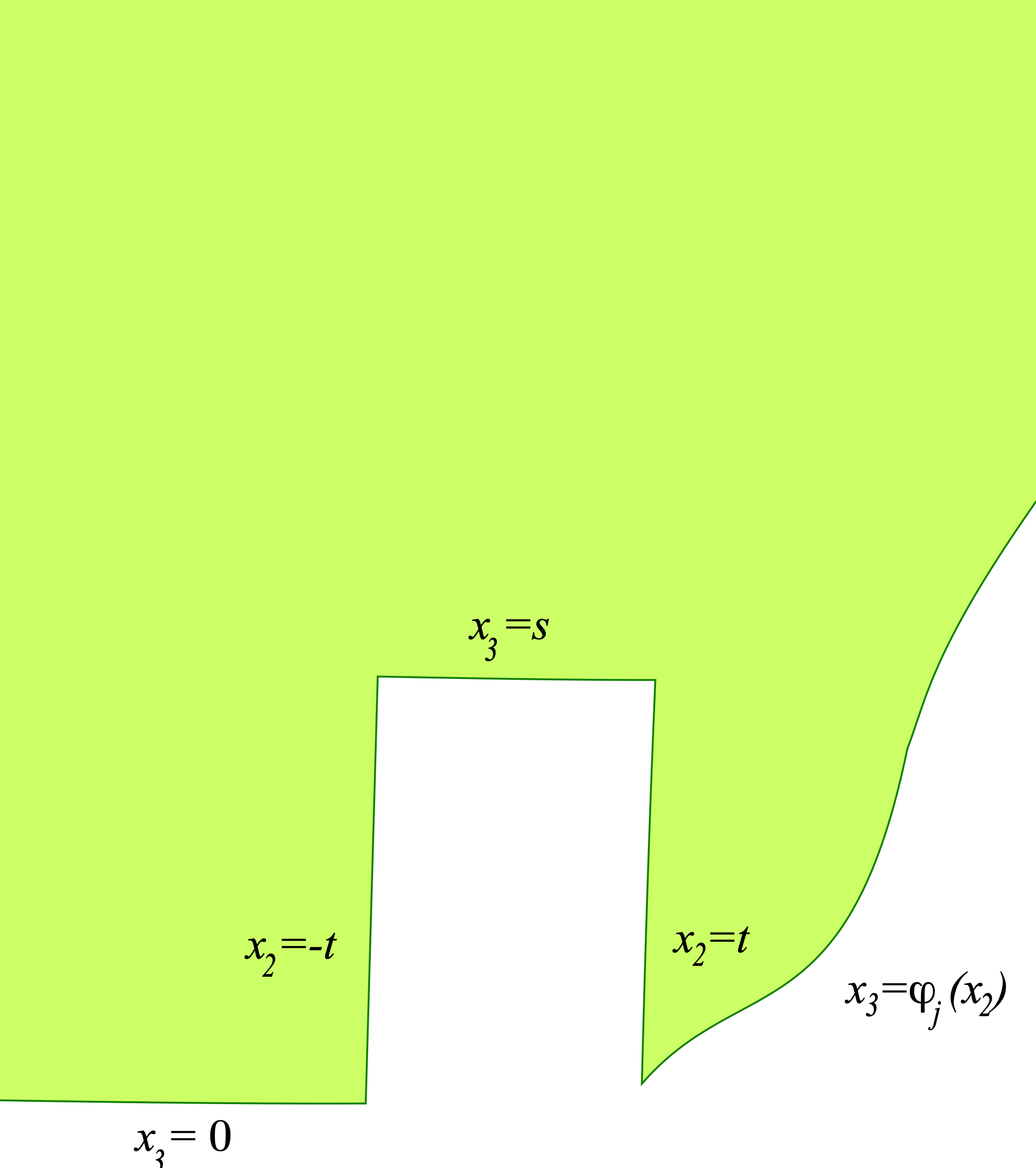}
    \caption{\small Two different types of domains $U_j^{\rm up}$. The third type of domain, that we have not included in the picture, has a boundary whose non-polygonal part consists of the graph of a function $\varphi_j$ defined on the half-line $(-\infty,-t)$ and the horizontal half-line $\{x_3=0, x_2>t\}.$}
    \label{fig:U_i}
\end{figure}

Each  piece of $\varphi_j$ has either the form  $x_3=f_{W_j^G}(x_2)$, for some grim reaper wing $W_j^G$, or the form $x_3=0$, for some planar wing $W_j^P$. Then, by Proposition \ref{bigraph},
  \[ \limsup_{x_2 \to \pm \infty} |\varphi_j'(x_2)| < C, \quad \mbox{for some constant $C>0$.}\]
  
 So, for any $d_j >C$, we have that 
  \begin{eqnarray}
      \limsup_{x \to+\infty} \;(d_j \, x_2-\varphi_j(x_2)) & = & \infty, \\
      \limsup_{x \to-\infty} \;(-d_j \, x_2-\varphi_j(x_2)) & = & \infty,
  \end{eqnarray}

Hence, we easily conclude the following:
\begin{prop}\label{cl:Up}
For any point $p=(x_2,x_3)\in U_j^{\rm up}$, there exists an angle $\theta$ so that the wedge
    \begin{equation} \label{eq:def-wedge}
\Wedge(p,\theta):=\{ q \in \R^2 \; : \; \sphericalangle(q-p,\ee_3) \in (-\theta,\theta) \}
    \end{equation}
    is contained in $U_j^{\rm up}.$ Furthermore,  the angle $\theta$ can be chosen uniformly on $$U_j^{\rm up} \cap \{x_3>s\}.$$
\end{prop}

We are going to denote \begin{equation} \label{def:xip}
    \xi_j={\xi_j}_{(p,\theta)}: \Wedge(p,\theta) \longrightarrow \R,
\end{equation} the function which is a solution of the translating graph equation \eqref{TSE} and
$$ E_j^{\rm up}(p,\theta):= {\rm Graph}(\xi_j)=\{ (\xi_j(x_2,x_3), x_2,x_3) \; : \; (x_2,x_3) \in  \Wedge(p,\theta) \} \subset \Sigma_j^{\rm up}. $$

\subsection{Graphs $ \Sigma_i^{\rm down}$: unbounded $U$-shaped domains. } In the case of a graph $\Sigma_i^{\rm down}$, we know that $\Sigma_i^{\rm down} \cap \{ |x_2| >t\} $ must be contained in 2 planar wings (one to the left and one to the right), because for a grim reaper type wing, $W$, we have that $W+t \ee_3 \to \varnothing$, as $t \to +\infty.$ Let us denote $W_{i_0}^P$ and $W_{i_1}^P$ these two planar wings.

So, if we call $U_i^{\rm down}$ the domain over which $ \Sigma_i^{\rm down}$ is a graph, then the shape of $U_i^{\rm down}$ is like the one indicated in Figure \ref{fig:domain-2}.
\begin{figure}
    \centering
    \includegraphics[width=0.5\linewidth]{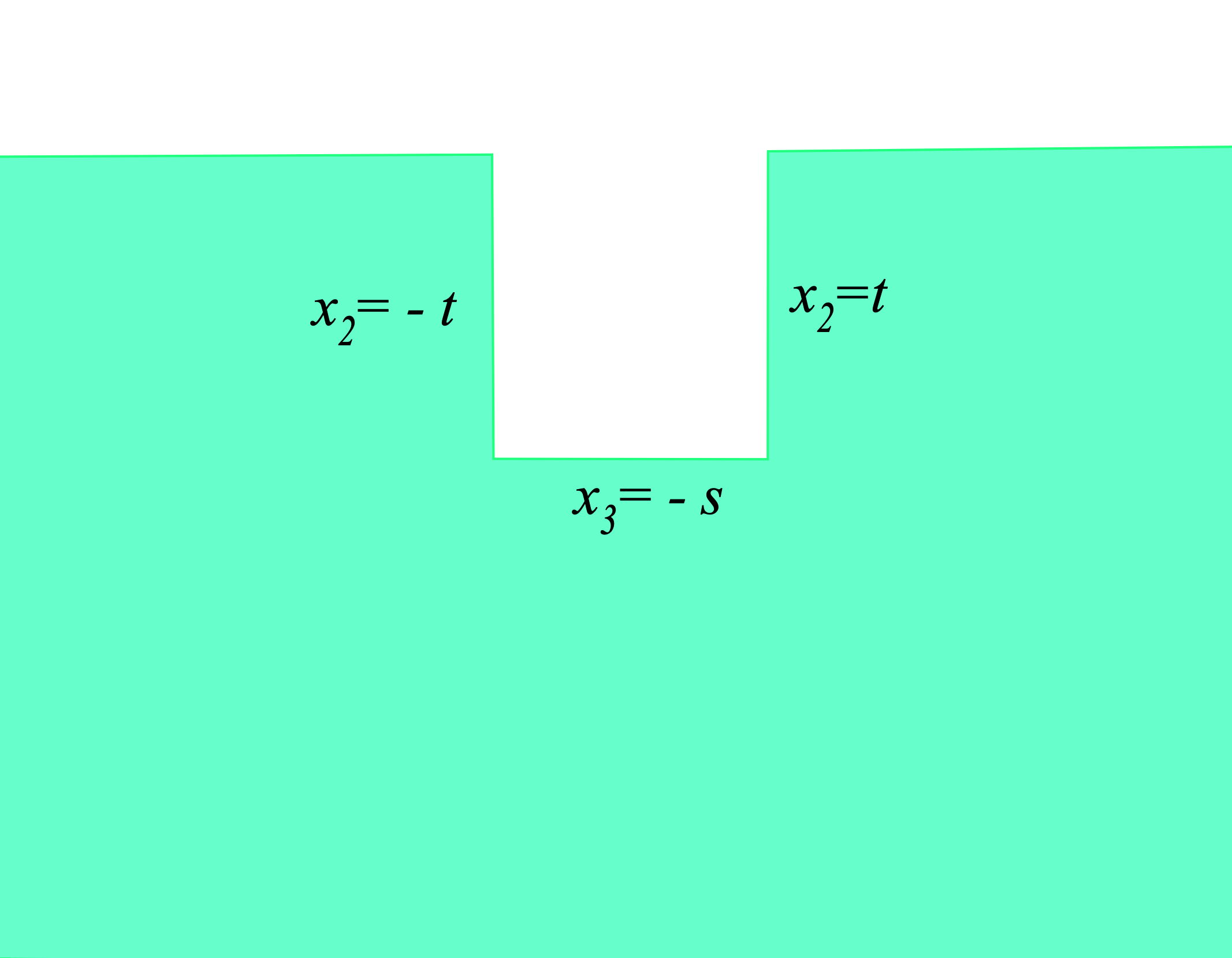}
    \caption{The domain $U_i^{\rm down}$.}
    \label{fig:domain-2}
\end{figure}
In other words, \begin{equation*}
    U_i^{\rm down}=(-\infty,-t)\times (-\infty,0) \cup  [-t,t] \times (-\infty,-s) \cup  (t,\infty) \times (-\infty,0).\end{equation*}
As $W_{i_0}^P$ and $W_{i_1}^P$ are $\ee_1$-graphs over the regions \[(-\infty,-t) \times \R, \quad \mbox{and}\quad (t,\infty) \times \R,\]
respectively. 
Thus, $ \Sigma_i^{\rm down}$ can be included in a bigger region of $\Sigma$ which is also an $\ee_1$-graph: $$ \widetilde {\Sigma_i}:=\Sigma_i^{\rm down} \cup W_{i_0}^P \cup W_{i_1}^P.$$
This region (see Figure \ref{fig:Xi1}, left) is an $\ee_1$-graph over the domain
\begin{equation}\label{def:H-domain}
H:=(-\infty,t)\times \R \cup [-t,t] \times (-\infty,-s) \cup (t,\infty) \times \R.
\end{equation}
This means that there is a function $\kappa_i: H \longrightarrow \R$ whose graph is a solution of \eqref{TSE-geo} (the function itself solving \eqref{TSE} below) and such that
$$\widetilde{\Sigma_i}= \{(\kappa_i(x_2,x_3),x_2,x_3) \; : \; (x_2,x_3) \in H\}.$$
\begin{rem}
Notice that $W^P_{i_0} \cap \{x_3>0\}$ is part of some upward graph $\Sigma_{j_0}^{\rm up}.$ Similarly, $W^P_{i_1} \cap \{x_3>0\}$ is part of some upward graph $\Sigma_{j_1}^{\rm up}.$
\end{rem}

The following theorem encapsulates and summarizes all the key results and findings discussed throughout this section (illustrated in Figure \ref{fig:Xi1}):

\begin{thm}\label{structure-thm}
    Let $\Sigma$ a complete, embedded translator so that its width, its entropy and its genus are finite. Consider $\Sigma^{\rm up}_j$, $j=1, \ldots \lambda(\Sigma)$, and  $\Sigma^{\rm down}_i$, $i=1, \ldots, \omega_P(\Sigma)/2,$ the regions of $\Sigma$ given by Proposition \ref{eq:decomposition-1}. Then one has the following:
    \begin{enumerate}
        \item For any $p \in \Sigma_j^{\rm up}$ there exists a wedge $\Wedge(p,\theta)$ and a smooth function, given by \eqref{def:xip},\[\xi_j:\Wedge(p,\theta) \rightarrow \R, \] such that the graph
        \[{\rm Graph}(\xi_j)=\{(\xi_j(x_2,x_3),x_2,x_3) \; : \; (0,x_2,x_3) \in \Wedge(p,\theta)\}\]
        is contained in $\Sigma_j^{\rm up}.$
        \item For any $\Sigma_i^{\rm down}$, there exists  a function $\kappa_i : H \rightarrow \R$ such that:
        \[ \Sigma_i^{\rm down} \subset {\rm Graph}(\kappa_i)=\{(\kappa_i(x_2,x_3),x_2,x_3) \; : \; (x_2,x_3) \in H\} \subset \Sigma.\]
        Moreover, there exists $a'>0$ and upward graphs $\Sigma_{j_0}^{\rm up}$ and $\Sigma_{j_1}^{\rm up}$ such that if we consider the regions in $\partial ({\rm Graph}(\kappa_i))$ given by \[\sigma_+ := \{ (\kappa_j(t,x_3),t, x_3) \; : \; x_3 > a'\} \quad \sigma_- := \{ (\kappa_j(-t,x_3),-t, x_3) \; : \; x_3 > a'\},\]
        then $\sigma_- \subset \Sigma_{i_0}^{\rm up}$ and $\sigma_+ \subset \Sigma_{i_1}^{\rm up}$. In particular, the curves $\sigma_+$ and $\sigma_-$ can be included in graphs over upward wedges like in item {\it (1)} (see Figure \ref{fig:Xi1}-Right.)
    \end{enumerate}
\end{thm}

\begin{figure}[htbp]
\begin{center}
\includegraphics[width=.4\textwidth]{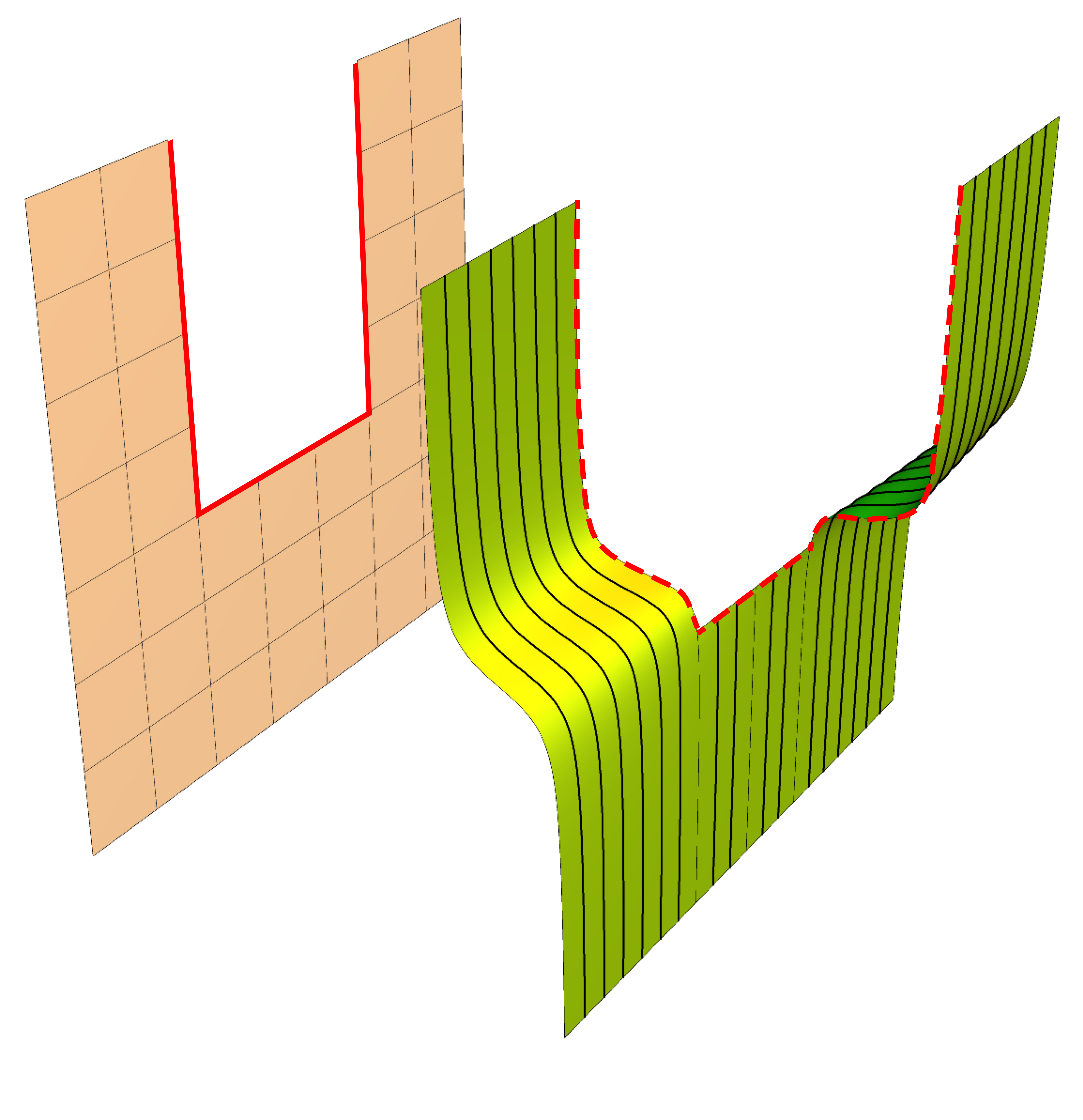}
\includegraphics[width=.4\textwidth]{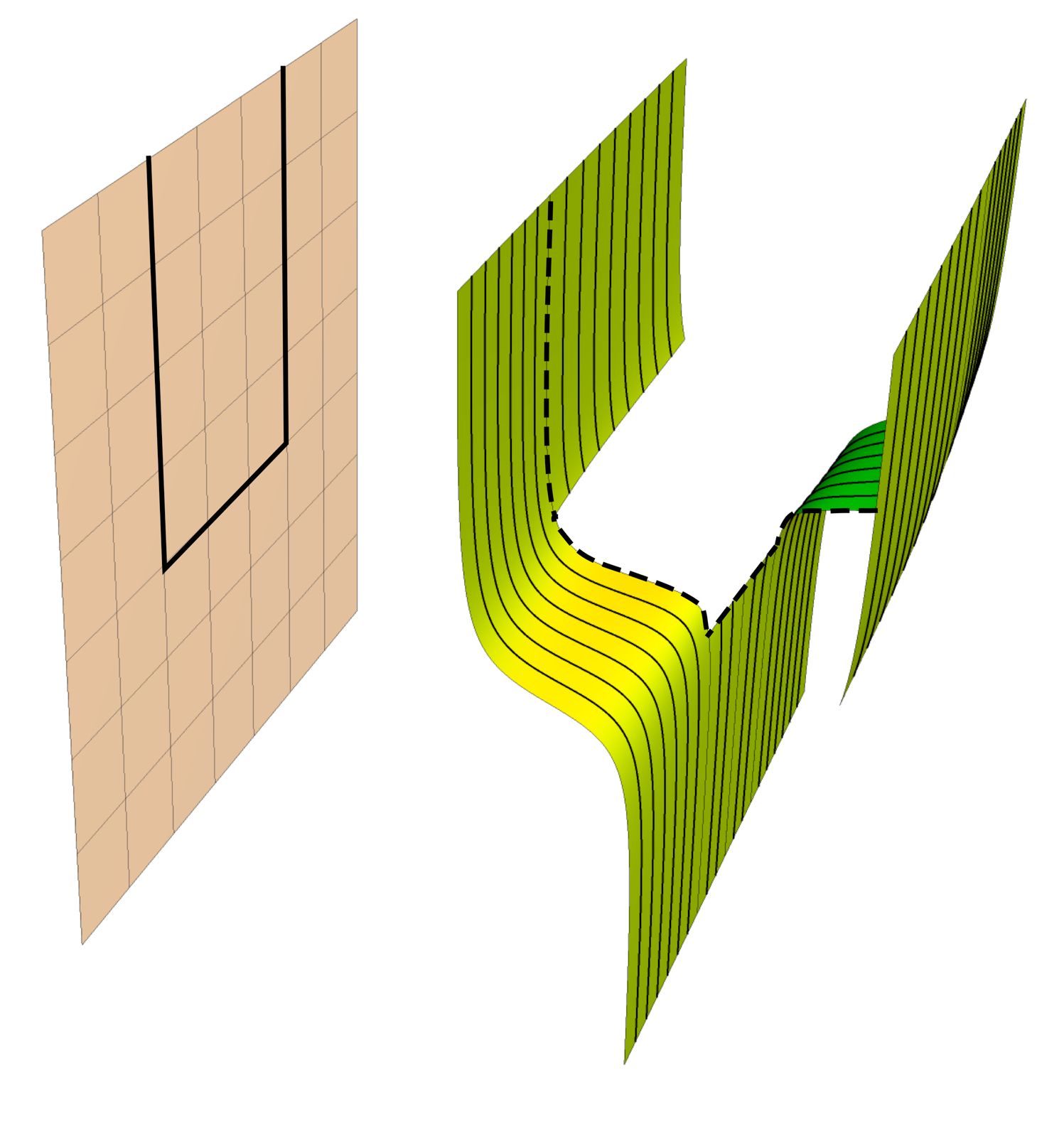}
\end{center}
\caption{Left: The graph of a $\kappa_j$ function in Theorem \ref{structure-thm}. Right: The graph of a $\kappa_j$ function with an extension of it so that the projection is the whole plane.} \label{fig:Xi1}
\end{figure}

\section{A priori derivative decay estimates for PDEs with drift}\label{sec:PDEs_drift}

In this section we will derive the needed $C^k$ a priori derivative decay estimates for the study of our PDEs with drift terms, including first the case of $L$-harmonics, when $L$ is the drift Laplacian, and then for the quasilinear translating soliton equation.

\subsection{The translating horizontal graphs equation and its linearization $L$}

As we saw from the geometric setup in Section \ref{sector}, the question of uniqueness of the limit of the sequences \(\Sigma\pm t_k\ee_3\), for different choices of sequence $\{t_k\}$, is the question of the existence of limits (for fixed $x_2$) of functions $u:\Omega\to\R$ defining $\ee_1$-graphs as in Theorem \ref{structure-thm}:
\[
``\mbox{\emph{Do the limits}} \hspace{-1pt}\lim_{x_3\to \pm\infty} u(x_2,x_3)\:\; \mbox{\emph{exist?}}".
\]

To investigate the existence or non-existence of true limits, i.e. uniqueness of limits independently of subsequence, in our global geometric setting, we will need to first carefully analyze the different types of graphical pieces, and then later combine this information via taking into account the overlaps of the constituent graphical pieces.

Note firstly that if we consider a function $u:\Omega\to\R$, considered as a graph over a vertical plane, then the self-translating soliton equation \eqref{TSE-geo} means that $u$ satisfies the following quasilinear equation
\begin{equation}\label{TSE}
    \diverg\left(\frac{\nabla u}{\sqrt{1+|\nabla u|^2}}\right)=-\frac{\partial _{3} u}{\sqrt{1+|\nabla u|^2}},
\end{equation}
where $\nabla u(x_2,x_3) = (\partial_2 u, \partial_3 u)$. We will refer to Equation \eqref{TSE}, that this paper is largely devoted to the study of, as the translating soliton equation (or translating horizontal graphs equation), whose solutions include the vertical planes $u(x_2,x_3) = ax_2 +b, a,b\in\R$, i.e. the planes that contain the translation direction. (See also Remark \ref{GoHorizontal}.)

Equation \eqref{TSE} is equivalent to the following equation
\begin{equation}\label{drift-Laplacian}
Lu=P(\nabla u, \Hess u),    
\end{equation}
where the linearization of \eqref{TSE} is the drift Laplacian
\begin{equation}\label{def:L}
L=\Delta + e_3\cdot \nabla,
\end{equation}
while the higher order term reads
\begin{equation}\label{RHS}
P(\nabla u, \Hess u) := \frac{\Hess u(\nabla u, \nabla u)}{1+|\nabla u|^2} = \frac{\sum_{i,j\in{2,3}}\partial^2_{ij}u\,\partial_iu\,\partial_j u}{1 + (\partial_2 u)^2 + (\partial_3 u)^2},
\end{equation}
denoting the Hessian by $(\Hess u)_{ij} = \partial_{ij}^2 u$.

\begin{rem}\label{GoHorizontal}
Equations \eqref{TSE}-\eqref{RHS} should be contrasted with the vertical translating graphs equation, for graphs into the direction of translation, which is likely more well-known to the reader, and has been studied extensively \cite{AW,CHH2,CHH1,HIMW,HIMW-2,Spruck-Xiao,Wang},
\begin{equation}\label{OldGraphs}
\diverg\left(\tfrac{\nabla v}{\sqrt{1+|\nabla v|^2}}\right)=\tfrac{1}{\sqrt{1+|\nabla v|^2}}.
\end{equation}
While locally of course also describing patches of translators, equation \eqref{OldGraphs} cannot describe surfaces that are (nearly) vertical on large scales, in a way suitable for the analytical details that we need. The drift term being exposed so clearly from the viewpoint of the horizontal graphs Equations \eqref{drift-Laplacian}-\eqref{def:L}, in the operator $L$, an operator which will be very crucial to our analysis in this section, is less apparent in \eqref{OldGraphs}.

Note also that the geometries that the two equations are used to describe are also fundamentally different, as the surfaces we consider via \eqref{TSE} are not naturally viewed as cylindrical over compact links (such as e.g. a bowl translator would be).
\end{rem}

\subsection{Green's function for $L$ via the Yukawa equation}
The drift Laplacian $L = \DriftL$ is conjugate via multiplication by appropriate exponentials, to the operator $\hat{L}$:
\begin{equation}\label{KonjugTrick}
\hat{L}v = \Delta v - \tfrac{1}{4}v = e^{\frac{x_3}{2}} L\: (e^{\frac{-x_3}{2}}v).
\end{equation}
Thus, up to these exponential multiplicative factors, which certainly have large global effects on growth/decay, what we study here at the linear level is in fact the Yukawa equation (also known in the literature as the screened Poisson equation, or the modified Helmholtz equation), which is famous for its appearance in particle physics.

This also means that, using the classical expression for the Green's function (or fundamental solution) for $\hat{L}$, namely $G_{\hat{L}}(r) = \frac{1}{2\pi}K_0(\frac{r}{2})$, where $K_0$ denotes the modified Bessel function of the second kind, we easily derive that
\begin{equation}\label{GreensL}
G_L(x,x') = \frac{1}{2\pi}K_0\left(\frac{|x-x'|}{2}\right)e^{\frac{x_3' - x_3}{2}}
\end{equation}
is a Green's function for the drift Laplacian $L$.

This Green's function is, of course, not symmetric under $x' \leftrightarrow x$, as $L$ is not a formally self-adjoint operator. It leads to a Poisson-Duffin formula for half-spaces, Equation \eqref{L_Poisson-Duffin}, which will be used below to provide a finely detailed analysis of the behavior of the translating soliton equation at infinity (opposite the translation direction), namely for the proofs of Theorem \ref{thm-A} (proof of uniqueness of tangent planes at infinite time $t\to +\infty$) and Theorem \ref{WobbleEksempel} (counterexamples to uniqueness of tangent planes at infinite time $t\to +\infty$).

Letting in \eqref{GreensL} the variable $x' = 0$, and scaling by a constant, we also get a useful $L$-harmonic function, $Lu_K = 0$, for
\begin{equation}\label{u_K}
 u_K = e^{\frac{- x_3}{2}}K_0\left(\tfrac{1}{2}|x|\right),
\end{equation}
which below will be used to inform the analysis (sharpness of the gradient estimates in Lemma \ref{varme_gradient_inhomogent}), as well as its anti-barrier cousin,
\begin{equation}\label{u_I}
 u_I = e^{\frac{- x_3}{2}}I_0\left(\tfrac{1}{2}|x|\right),
\end{equation}
where $I_0$ denotes the modified Bessel function of the first kind, and satisfying likewise $Lu_I = 0$. $I_0$ is used to prove maximum principles on noncompact domains for the drift Laplacian in Section \ref{subsec:MPL}.

\begin{rem}
The function $u_K$ in Equation \eqref{u_K} also showed up within the analysis of \cite{Ilyas}, or rather the more slowly decaying $\sqrt{u_K}$, (see Remark \ref{rem:sq_root_barrier} for a longer discussion) but not, it appears, $u_I$ from \eqref{u_I}, or $G_L$ in \eqref{GreensL}. This explicit connection to the Green's function of $L$ in \eqref{GreensL}, and the relation to Yukawa's equation via \eqref{KonjugTrick}, which we have explained here, will also turn out to be central in our further analysis.
\end{rem}

\subsection{Sausage-shaped subdomains}
In the below proofs, we will be making use of subdomains of a special form, which we define and introduce notation for here. By an upward-directed $\rho$-sausage in space, $S^+_\rho(p)\subseteq\mathbb{R}^{n+1}$, we mean domains of the form, for $p\in\mathbb{R}^{n+1}$:
\begin{equation}\label{def:poelse}
S^+_\rho(p) := \cup_{t\in [0,\rho^2)} \left(B_\rho(p) + t\ee_{n+1}\right),
\end{equation}
and refer to $\rho > 0$ as the width of the sausage. In our setup of this paper, we will always take $n+1 = 2$, although the following of course also applies with appropriate modifications for arbitrary dimensions $n$.

\begin{dfn}[Domains with the upward sausages property]\label{dfn:sausage_property}
Suppose that $\Omega\subseteq \mathbb{R}^n$ is an open domain. We say that $\Omega$ has the upward sausages property if either $\partial \Omega = \emptyset$ or, when $\partial \Omega \neq\emptyset$,
\begin{equation}\label{eq:upward_sausages_cond}
\forall p\in\Omega:\quad S^+_\rho(p) \subseteq \Omega \quad\mathrm{for}\quad\rho := \dist(p,\partial\Omega).
\end{equation}
\end{dfn}

\begin{rem}\label{sausage_sampler}
For later reference, we note some explicit examples of domains which (do not) satisfy the upward sausages property:
\begin{itemize}
\item  Examples: all of space, wedges, any slanted upper half-plane (incl. left and right half-planes), all satisfy the upward sausages condition.

\item  Non-examples: punctured planes, exterior domains $\Omega = \mathbb{R}^{n+1}\setminus K$ (for $K\neq\empty$ compact), U-shaped domains (plane minus upward half-infinite rectangle), slanted lower half-planes (apart from left and right half-planes) do not satisfy the conditions of the upward sausages property.
\end{itemize}
\end{rem}

\subsection{Gradient and Hessian estimates for the translating soliton equation}
In order to make use of the properties of solutions to Poisson problems $Lu=f$ for the drift Laplacian $L$, in the study of the translating soliton equation \eqref{TSE}, we need to ensure that we have suitably strong decay estimates for the right-hand side (non-homogeneous) term $P$. This section will begin by recording the weaker standard estimates for quasilinear problems holding in larger generality. In Section \ref{sec:sharp_linear}, we will show the crucial sharp linear estimates for $L$, and then finally in Section \ref{sec:sharp_nonlinear} combine this sharpened linear theory with information about the non-homogeneous term, to get the new upgraded quasilinear estimates, which we will then be relying on throughout the rest of the paper.

We start by recalling that Gama et. al \cite[Proposition  3.5]{GHLM} (or see Evans-Spruck, \cite[Corollary 5.3]{Evans-Spruck}) implies that for the quasilinear Equation \eqref{TSE}, holds:

\begin{prop}[Classical quasilinear gradient estimates \cite{Evans-Spruck,GHLM}]\label{grad_bounded}
Let $u: \Omega \rightarrow \R$ solve the translating soliton equation \eqref{TSE}. Then $u$ satisfies the following a priori gradient estimate:
    \begin{equation}\label{grad_bound_basic}
    |\nabla u| \leq K\quad {\rm on}\quad \Omega',
    \end{equation}
    for any $\Omega'\subsetneq\Omega$ such that $\dist(\partial\Omega',\partial\Omega)\geq 1$,
    where $K = K(\|u\|_{\infty})>0$ is a universal constant depending only on the supremum norm of $u$.
\end{prop}

For $\Omega$ taken to be a domain as in the geometric structural result, Theorem \ref{structure-thm}, note that the type of domain is not changed if we shrink it into a unity-width tubular neighborhood of its boundary. Thus, we will in the below always be able to easily arrange that \eqref{grad_bound_basic} holds for the domains we consider. Next, we record the classical curvature estimates by Ecker-Huisken:

\begin{prop}[Classical quasilinear Hessian estimates \cite{Ecker-Huisken}]\label{sec-decay}
 For $u:\Omega\to\R$ any bounded solution of the translating soliton equation \eqref{TSE},
 \begin{equation}\label{Hess_down}
|\Hess u\,|(p)\leq \frac{K}{\dist(p,\partial \Omega)^{\frac{1}{2}}},\quad\mathrm{on}\quad \Omega',
\end{equation}
for any $\Omega'\subsetneq\Omega$ such that $\dist(\partial\Omega',\partial\Omega)\geq 1$, where $K = K(\|u\|_{\infty})>0$ is a universal constant depending only on the supremum norm of $u$.

In the case of the domains satisfying the upward sausages condition, this can be strengthened to:
 \begin{equation}\label{Hess_up}
|\Hess u\,|(p)\leq \frac{K}{\dist(p,\partial \Omega)},\quad\mathrm{on}\quad \Omega'.
 \end{equation}
\end{prop}
In other words, if $\Omega$ is one of the domains in Theorem \ref{structure-thm}, then in case of the downward pointing domains \eqref{Hess_down} holds, while for the upwards pointing \eqref{Hess_up} holds.

\begin{proof}
We consider a solution of $L u=P,$ where $P= P(\nabla u,\Hess u)$. Then, we define
\begin{equation}\label{eq:def-v}
    v(x_2,x_3,t):= u(x_2,x_3-t), 
\end{equation}
which is a solution to the mean curvature flow equation
\begin{equation}
    \partial_t v -\Delta v=-\partial_{3} u -\Delta u=-P.
\end{equation}

We consider $p_0=(x_2^0,x_3^0)$ in $\Omega'$ and let $\rho= \sqrt{\dist (p_0,\partial \Omega) - 1}> 0$ (respectively $\rho= \dist (p_0,\partial \Omega) - 1$ in the case of the upwards pointing domains).

From Proposition  \ref{grad_bounded}, we have that $\| \nabla v\| \leq K$
on $B_\rho(p_0) \times (- \rho^2, 0]$, by the definition of the domains $S^+_\rho(p_0)$ in \eqref{def:poelse}. Then we can apply \cite[Proposition  3.21]{Ecker} to deduce that there exists a constant $K$ (depending only on the old constant $K$, which in turn only depends on $\|u\|_{\infty}$), so that
\begin{equation}
    \label{eq:2ff} |A(p_0,0)|^2 \leq \frac{K}{\rho^2},
\end{equation}
where $A(\cdot, t)$ represents the second fundamental form of the $e_1$-directed graph of $v(\cdot,t))$.
Thus:
\begin{equation}
    \label{eq:2ff-1} |A(p_0)|^2 \leq \frac{K}{\rho^2},
\end{equation}
where $A(p_0)$ is the second fundamental form of ${\rm graph}(u).$ This provides the desired estimate for the Hessian of $u$ (for possibly a different $K$):
\begin{equation} \label{eq:Hess}
    \forall p_0\in\Omega':\quad |\Hess u|(p_0) \leq \frac{K}{\dist(p_0 ,\partial \Omega)}.
\end{equation}
\end{proof}

These standard estimates in Proposition \ref{grad_bound_basic} and Proposition \ref{sec-decay} unfortunately do not give strong enough decay for our purposes, which include to show a removable singularity theorem at the various types of infinities, for the self-translating soliton equation, and hence later be able prove the existence of true limits at infinity.

Therefore, in the next section, we will be upgrading the estimates to the needed strength, by detailed analysis of the properties of solutions to the Poisson problem for the drift Laplacian.

\subsection{Sharp elliptic-parabolic gradient estimates, for the drift Laplace-Poisson equation}\label{sec:sharp_linear}

We will need the following crucial non-standard elliptic gradient estimates with explicit dependencies of the constants on the domain sizes and directions, which we note have an ``elliptic-parabolic'' nature, in that they superficially may look like classical elliptic estimates for the operator $L$. They are therefore in fact not statements about round balls in space, rather instead about directed elongated domains in space the form of the $\rho$-sausages from \eqref{def:poelse}.

These estimates can thus be thought of as stemming from the parabolic origins of the elliptic operator, and we will need them with sharp domain-dependency of the constants reflecting this fact and depending on the geometry, incl. direction, of the domains in question:

\begin{lem}[Gradient estimates for the drift Laplacian on $\rho$-sausages]\label{varme_gradient_inhomogent}
There exists a numerical constant $K$ such that for any $u:S^+_\rho(p)\to \mathbb{R}$, where $S^+_\rho$ is an upward sausage as in \eqref{def:poelse} of arbitrary width $\rho > 0$, the following gradient estimate holds for $L=\Delta + e_3\cdot \nabla$ the drift Laplacian:
\begin{equation}\label{drift_poelse}
 |\nabla u|(p)\leq K\left(\frac{1}{\rho}\left \|u\right\|_{L^\infty(S^+_\rho(p))} + \rho\left\|Lu\right\|_{L^\infty(S^+_\rho(p))}\right).
\end{equation}

\noindent{}For $p\in\Omega$ a general domain, with $\rho := \dist(p,\partial\Omega)$:
\begin{equation}\label{sharp_down_est}
|\nabla u|(p)\leq K\left(\rho^{-\frac{1}{2}}\left \|u\right\|_{L^\infty(B_\rho(p))} + \rho^\frac{1}{2}\left\|Lu\right\|_{L^\infty(B_\rho(p))}\right).
\end{equation}

\noindent{}For $p\in\Omega$ a domain satisfying the upward sausages condition \eqref{eq:upward_sausages_cond}, we have, with $\rho := \dist(p,\partial\Omega)$ the stronger estimate:
\begin{equation}
|\nabla u|(p)\leq K\left(\frac{1}{\rho}\left \|u\right\|_{L^\infty(B_\rho(p))} + \rho\left\|Lu\right\|_{L^\infty(B_\rho(p))}\right).
\end{equation}

\noindent{}In particular, bounded $L$-harmonic functions $u$ in general satisfy $|\partial_{2} u|(p) \leq K \rho^{-\tfrac{1}{2}}\|u\|_\infty$, where the exponent ``$-\frac{1}{2}$'' on $\rho$ is sharp.
\end{lem}

\begin{proof}

We assume WLOG that $p=(0,0)$. Let $f: = Lu$, and furthermore define $\varphi(x_2,x_3,t) : = -f(x_2,x_3 - t)$ as well as $v(x_2,x_3,t) := u(x_2,x_3 - t)$. The equality \eqref{drift_poelse} will now be a direct consequence of the following gradient estimate for solutions to the inhomogeneous heat equation $\partial_tv - \Delta v = \varphi$ on $C_\rho := B_\rho(0,0) \times (-\rho^2,0]$:
\begin{equation}\label{varme_gradient}
|\partial_{i} v(0,0,0)| \leq C\left(\frac{1}{\rho}\left \|v\right\|_{L^\infty(C_\rho)} + \rho\left\|\varphi\right\|_{L^\infty(C_\rho)}\right).
\end{equation}
The estimate \eqref{varme_gradient} is surely standard folklore, although as we were not able to pinpoint a self-contained reference to this precise form, we will briefly explain its proof here.

We first note that by parabolic re-scaling around $(0,0,0)$ via $v_\rho (x,t) := v(\rho x, \rho^2 t)$,
it is enough to prove the estimate \eqref{varme_gradient} for $\rho = 1$.

The case $\varphi\equiv 0$ now firstly follows from the standard cut-off argument written out in e.g. in \cite{Evans}[p. 59-63]. Namely, the standard interior gradient estimates for the homogeneous heat equation show that for $(x^0_2,x^0_3)\in V_{\alpha}$,
\begin{align}
&|\partial_{i}v(x^0_2,x^0_3, 0)|\leq \|\partial_{i}v\|_{L^\infty(C_{\rho/2}((x^0_2,x^0_3)))}\leq\frac{C}{\rho}\|v\|_{L^\infty(C_\rho(x^0_2,x^0_3))}\\
&|\partial_{ij}v(x^0_2,x^0_3, 0)|\leq \|\partial_{ij}v\|_{L^\infty(C_{\rho/2}((x^0_2,x^0_3)))}\leq\frac{C}{\rho^2}\|v\|_{L^\infty(C_\rho(x^0_2,x^0_3))},
\end{align}
provided $\rho>0$ is chosen small enough that all points where $u$ is evaluated remain within $V_{\alpha}$. For this, note that by definition
\begin{align*}
\|v\|_{L^\infty(C_\rho(x^0_2,x^0_3))} & = \sup\{u(x_2,x_3 - t)\mid (x_2,x_3)\in B_\rho(x^0_2,x^0_3),\: -\rho^2 < t\leq 0\}\\
& = \sup\{u(x_2,x_3)\mid (x_2,x_3) \in S^+_\rho(x^0_2,x^0_3)\}\\
&=\|u\|_{L^\infty(S^+_\rho(x^0_2,x^0_3))},
\end{align*}
where the set over which the supremum is taken is the natural sausage-shaped region $S^+_\rho(x^0_2,x^0_3)$ in \eqref{def:poelse}.

Next, the introduction of a non-zero source term $\varphi$ gives an extra contribution from the additional term $\tilde{\varphi} : = \zeta \varphi$ (where $\zeta$ is a cut-off function to a smaller parabolic cylinder still with apex at (0,0,0)), which to the gradient of the solution adds, from Duhamel's formula and the parabolic maximum principle (to establish the representation formula, and which is where the ``$-1$'' time-shift comes in, to ensure that the solutions agree at the bottom of the parabolic cylinder $C_1$) a term given by the following, for $-1 < t \leq 0$, where here $x = (x_2,x_3)$ and $y = (y_2,y_3)$:
\begin{equation}
I = \partial_{i} \left[ \int_{0}^{t+1} \int_{\mathbb{R}^2} \Phi(x-y, t + 1 - s)\tilde{\varphi}(y, s - 1) dy_2dy_3 ds\right],
\end{equation}
where $\Phi(x-y,t+1-s) = \frac{1}{4\pi(t+1-s)}\exp\left({-\frac{|x-y|^2}{4(t+1-s)}}\right)$ is the standard heat kernel. Noting $\partial_{i}\Phi =\frac{y_i - x_i}{8\pi(t+1-s)^2}\exp\left({-\frac{|x-y|^2}{4(t+1-s)}}\right)$, and using the substitution $\eta := (x-y)/(2\sqrt{t+1-s})$, we estimate:
\begin{equation}
\begin{split}
|I| &\leq \frac{1}{\pi} \int_0^{t+1} (t+1-s)^{-\frac{1}{2}}\int_{\mathbb{R}^2} |\eta_i| e^{-|\eta|^2}\tilde{\varphi}(x - 2\sqrt{t+1-s},s)d\eta ds\\
& \leq C_1\|\tilde{\varphi}\|_\infty \int_0^{t+1} (t+1-s)^{-\frac{1}{2}} ds = C_2\sqrt{t+1} \|\tilde{\varphi}\|_\infty\leq C_2\|\varphi\|_\infty,
\end{split}
\end{equation}
when $-1 < t \leq 0$, which also justifies having taken the derivative inside the integral. This finishes the proof of \eqref{varme_gradient}, and hence of \eqref{drift_poelse}.

Sharpness of the exponent for general domains \eqref{sharp_down_est} follows as claimed by inserting the $L$-harmonic function $u_K = e^{-x_3/2}K_0(|(x_2,x_3)|/2)$, defined on $\Omega\setminus\{0\}$, where $K_0$ denotes the modified Bessel function of the second kind,  which is essentially the Green's function for the drift Laplacian $L$ (see \eqref{GreensL}-\eqref{u_K}).

\end{proof}
As one simple corollary to Lemma \ref{varme_gradient_inhomogent} we conclude the following:
\begin{cor}[Liouville property for the drift Laplacian]\label{cor:liouville}
Let $u:\mathbb{R}^{n+1} \to \mathbb{R}$ be a bounded solution to $Lu = 0$, where $L = \Delta + \ee_3\cdot \nabla$. Then $u = \mathrm{const}$.
\end{cor}

We will however not be making use of Corollary \ref{cor:liouville} this paper, since our graphs will usually not be entire.

\subsection{Improved quasilinear gradient and Hessian estimates}\label{sec:sharp_nonlinear}

We can now use Lemma \ref{varme_gradient_inhomogent} to improve on the classical Evans-Spruck style gradient \cite{Evans-Spruck} and Ecker-Huisken style Hessian estimates \cite{Ecker-Huisken}, recorded in the previous section's Propositions \ref{grad_bounded} and \ref{sec-decay}, to gain additional orders of decay over these classical estimates:

\begin{prop}[Improved gradient decay estimates for the soliton equation] \label{grad_estimate_quasilinear_all}

For $u:\Omega\to\R$ any bounded solution of the translating soliton equation \eqref{TSE},
 \begin{equation}\label{grad_down}
|\nabla u|(p)\leq \frac{K}{\dist(p,\partial \Omega)^{\frac{1}{2}}},\quad\mathrm{on}\quad \Omega',
\end{equation}
for any $\Omega'\subsetneq\Omega$ such that $\dist(\partial\Omega',\partial\Omega)\geq 1$, where $K = K(\|u\|_{\infty})>0$ is a universal constant depending only on the supremum norm of $u$.

In the case of the domains satisfying the upward sausages condition \eqref{eq:upward_sausages_cond}, this can be strengthened to:
 \begin{equation}\label{grad_up}
|\nabla u|(p)\leq \frac{K}{\dist(p,\partial \Omega)},\quad\mathrm{on}\quad \Omega'.
 \end{equation}
\end{prop}

\begin{proof}[Proof of Proposition \ref{grad_estimate_quasilinear_all}]
From \eqref{drift-Laplacian}, we have $Lu=P(\nabla u,\Hess u)$, so using Propositions \ref{grad_bounded} and \ref{sec-decay}, one gets
\[
|P(\nabla u,\Hess u)|_{L^\infty(B_{\rho^{\frac{1}{2}}}(p))}\leq \frac{K}{\rho^{\frac{1}{2}}},
\]
where $\rho=\dist(p,\partial \Omega).$ At this moment, we bootstrap on the order of decay, by using first the estimate \eqref{sharp_down_est} in Lemma \ref{varme_gradient_inhomogent} when inserting a smaller sausage, of width $\rho^{\frac{1}{4}}$, into the width $\frac{1}{2}$ sausage, which can be done when $\rho >1$ (and otherwise the estimate is trivial), to infer that
\begin{equation}\label{quarter_decay}
|\nabla u|(p)\leq \frac{K}{\rho^{\frac{1}{4}}}.
\end{equation}

Now, in order to bootstrap this decay estimate to the power $\frac{1}{2}$, we first use \eqref{quarter_decay} together with Proposition \ref{sec-decay} to improve the previously shown decay of the nonlinear term $P$ to
\[
|P(\nabla u,\Hess u)|_{L^\infty(B_\rho(p))}\leq \frac{K}{\rho}.
\]
But then, invoking \eqref{sharp_down_est} in Lemma \ref{varme_gradient_inhomogent} again gives us the desired estimate
\begin{equation}
|\nabla u|(p)\leq \frac{K}{\rho^{\frac{1}{2}}}.
\end{equation}
This finishes the proof of the estimate for general domains \eqref{grad_down}, and the proof of \eqref{grad_up} is an analogous bootstrap on decay, by using the appropriate improved estimates throughout.
\end{proof}

\begin{prop}[Improved Hessian decay estimates for the soliton equation] \label{ImprovedHessian}
For $u:\Omega\to\R$ any bounded solution of the translating soliton equation \eqref{TSE},
\begin{equation}\label{Hess_down_improved}
|\Hess u|(p)\leq \frac{K}{\dist(p,\partial \Omega)},\quad\mathrm{on}\quad \Omega',
\end{equation}
for any $\Omega'\subsetneq\Omega$ such that $\dist(\partial\Omega',\partial\Omega)\geq 1$, where $K = K(\|u\|_{\infty})>0$ is a universal constant depending only on the supremum norm of $u$.

In the case of the domains satisfying the upward sausages condition \eqref{eq:upward_sausages_cond}, this can be strengthened to:
 \begin{equation}\label{Hess_up_improved}
|\Hess u|(p)\leq \frac{K}{(\dist(p,\partial \Omega))^2},\quad\mathrm{on}\quad \Omega'.
 \end{equation}
\end{prop}

\begin{proof}
Differentiating \eqref{RHS}, we get:
\begin{equation}
L(\partial_k u) = \partial_k (L u) = \partial_kP = Q_k(\nabla u, \Hess u, \partial^3_{ijk} u),
\end{equation}
where
\begin{equation*}
Q_k = -\frac{2}{\left(1 + |\nabla u|^2\right)^{2}}\left[\sum_{i,j,l\in{2,3}}\partial^2_{ij}u\,\partial_iu\,\partial_ju\,\partial_ku\,\partial_lu + \left(1 + |\nabla u|^2\right)\sum_{i,j\in{2,3}}\partial_k\hspace{-2pt}\left(\partial^2_{ij}u\,\partial_iu\,\partial_ju\right) \right].
\end{equation*}
From the standard general quasilinear estimates, see Ecker-Huisken \cite[Proposition  3.22]{Ecker}, \cite{Ecker-Huisken}, we get that $|\partial^3_{ijk}u|\leq C\rho^{-2}$ on any $\rho$-sausage. Taking $\rho = \dist(\cdot,\partial\Omega)$ on domains with the upward sausages condition (resp. $\rho = \sqrt{\dist(\cdot,\partial\Omega)}$ on general domains), we see that $|\partial^3_{ijk}u|\leq C\dist(\cdot,\partial\Omega)^{-2}$ (resp. same with power "-1" for down).

Thus, $|L(\partial_k u)| = |Q_K| \leq C\dist(\cdot,\partial\Omega)^{-3}$ (resp
 power "-3/2"). Together with \eqref{drift_poelse} from Lemma \ref{varme_gradient_inhomogent} applied to $\partial_k u$, using that $|\partial_k u|\leq C\dist(\cdot,\partial\Omega)^{-1}$ (resp. power "-1/2"), and with the same choices of $\rho$, respectively, we get from \eqref{grad_down}-\eqref{grad_up} the Hessian bounds in Equations \eqref{Hess_down_improved}-\eqref{Hess_up_improved}.
\end{proof}

Pulling all this together we have, via \eqref{Hess_down} and \eqref{grad_down}, shown the following improved decay for the nonlinear term $P$:

\begin{prop}[Higher order term decay estimates for the soliton equation]\label{P_decay_estimates}
 Let $u:\Omega\to\R$ be any bounded solution of the translating soliton equation \eqref{TSE}. Then
 \begin{equation}\label{down_RHS}
|P(\nabla u,\Hess u)|(p)\leq \frac{K}{\left(\dist(p,\partial \Omega)\right)^{2}},\quad p\in\Omega',
 \end{equation}
 for any $\Omega'\subsetneq\Omega$ such that $\dist(\partial\Omega',\partial\Omega)\geq 1$,
    where $K = K(\|u\|_{\infty})>0$ is a universal constant depending only on the supremum norm of $u$.

In the case of the domains satisfying the upward sausages condition \eqref{eq:upward_sausages_cond}, this can be strengthened to:
\begin{equation}\label{up_RHS}
|P(\nabla u,\Hess u)|(p)\leq \frac{K}{\left(\dist(p,\partial \Omega)\right)^{4}}, \quad\mathrm{on}\quad \Omega'.
 \end{equation}
\end{prop}

\begin{rem}For the upward-directed domains, one can now show even further improvements of the estimates in \eqref{grad_up}, \eqref{Hess_up_improved} and \eqref{up_RHS}, by continuing to bootstrap on the order of decay, as in the proofs of Proposition \ref{grad_estimate_quasilinear_all} and Proposition \ref{ImprovedHessian}, to get arbitrarily high powers of decay. However will instead in Theorem \ref{upwards-limits-exp} below upgrade this all the way to exponential decay, for the upwards wedge domains. Already from the estimate \eqref{Hess_up}, it follows that over upwards wedges, bounded graphical pieces of the translating soliton are automatically of finite total curvature, meaning finiteness of the integral $\int_{\Omega'}|A|^2 < \infty$.

For the general case on the other hand, which applies to downward pointing domains in Theorem \ref{structure-thm}, note that \eqref{down_RHS} does not a priori allow for any further bootstrapping on the order of decay. The weaker estimate in \eqref{Hess_down_improved} for downward pointing domains also does not allow one to draw any finiteness conclusions about the total curvature of such pieces. See also the counterexample below in Theorem \ref{WobbleEksempel}.
\end{rem}

\section{Removable singularity theorems at infinity for PDEs with drift}\label{sec:Rem-Sin}

In this section we will prove the announced removable singularity theorems at infinity for our PDEs with drift terms, including in the Dirichlet problems for the drift Laplacian and the translating soliton equation \eqref{TSE}.

\subsection{Maximum principles for the drift Laplacian}\label{subsec:MPL}

Since the drift Laplacian that we work with is $L = \DriftL$, without any constant term, the usual strong maximum of course holds on all bounded smooth domains, a fact which we will often be making use of in the below.

In the non-compact cases which we will be dealing extensively with in the below, there's no reason to expect maximum principles which are strong enough to imply uniqueness in the Dirichlet problem for $L$ on upward-pointing (along the translating soliton's direction of motion) domains. To wit, in contrast to the non-drifted Laplacian $\Delta$, we note:

\begin{remark}[Failure of $L^\infty$ global maximum principle for $L$]\label{rem:MP_fails}
The non-trivial bounded solution to $Lu = 0$ on $\{x_3\geq 0\}$ given by $u(x_2,x_3) := 1 - e^{-x_3}$ shows that uniqueness in the Dirichlet problem may in general fail for the drift operator $L$ on noncompact domains, even in the class of $L^\infty$ functions. (See also a version of the Liouville theorem, in Corollary \ref{cor:liouville}.)
\end{remark}

The lack of such maximum principles, as shown in Remark \ref {rem:MP_fails}, is an important complication for the proof of the removable singularity theorem in $L^\infty$ on upward wedges, which is, however, remedied by the sharp elliptic estimates in Lemma \ref{varme_gradient_inhomogent}, which will turn out to exhibit fast enough decay to be used to side-step the issue of an unknown Dirichlet kernel, in order to directly prove limits at infinity on such domains, in Proposition \ref{limit_line_up}. Only after this convergence result may we then make use of the fact that if prescribing the asymptotics at infinity, then useful maximum principles do in fact hold, also in the upward direction. See the proof of Theorem \ref{thm:exponential} below.

As the following maximum principle lemma (which we will only be needing in the case of the lower half-plane) shows, in the downwards direction, one does however have uniqueness in the Dirichlet problem for the drift Laplacian $L$, even in the merely $L^\infty$ class, without requiring any decay or convergence at infinity:

\begin{lem}[Phragm\'en-Lindel\"o{}f in $L^\infty$ for drift Laplacian on downward domains]\label{Phragmen_Lindelof_L_infty}
\phantom{hest}
\noindent{}Consider $\Omega$ a smooth domain whose complement contains an upwards wedge, e.g. $\Omega$ being the lower half-plane. Then, for the drift Laplacian $L = \Delta + e_3\cdot \nabla$ the following hold:

\begin{enumerate}[(i)]

\item[]

\item $u\in C^2(\Omega)\cap L^\infty(\Omega): \quad Lu \geq 0, \quad\mathrm{and}\quad u_{\mid \partial \Omega} \leq 0 \quad\Rightarrow\quad u \leq 0.$

\item[]

\item Uniqueness holds for the $L$-harmonic Dirichlet problem on $\Omega$ in the class of bounded functions:
\begin{equation*}
u_1,u_2\in C^2(\Omega)\cap L^\infty(\Omega):\quad Lu_1 = Lu_2,\quad\mathrm{and}\quad u_{1\mid \partial \Omega} = u_{2\mid \partial \Omega}\quad\Rightarrow\quad u_1 = u_2.
\end{equation*}
\end{enumerate}
\end{lem}
\begin{proof}
The first part follows by using the barrier derived from the anti-Green's function in \eqref{u_I}, i.e.
$\iota(x_2,x_3): = e^{-x_3}I_0(|(x_2 - a,x_3 - b)|/2)$, where $I_0$ denotes the modified Bessel function of the first kind, as an $L$-harmonic anti-barrier with $a,b$ chosen so that $(a,b)$ lies outside of $\Omega$.

Namely, on domains that avoid an upwards wedge, the $\iota$ function tends exponentially to infinity as $|(x_2,x_3)|\to \infty$, while the (sub)harmonic considered are all in $L^\infty$. Fixing first arbitrary $\varepsilon >0$ and applying the standard maximum principle for the drift Laplacian $L$, which has no constant term, on an exhaustion by compact domains to $u - \varepsilon u_b$, shows that $u - \varepsilon u_b\leq 0$ on $\Omega$. But then $\varepsilon \searrow 0$ gives that $u\leq 0$ as claimed.

The second statement is equivalent to proving that the Dirichlet problem with zero boundary value only admits the trivial solution, which follows directly from Part (i), by considering also the function $-u$.
\end{proof}

\subsection{Limits for $Lu = P$ solutions on upward wedges}
We employ polar coordinates $(x_2,x_3)=(r\cos\theta, r\sin\theta)$, and consider open unbounded wedges, which are symmetric around $e_3$ and spanning an angle $ 0 < 2 \, \alpha < \pi$:
\begin{equation} \label{def:VA}
 V_\alpha:=\Wedge((0,0),\alpha)=\left\{(r\cos\theta, r\sin\theta))|\: r>0,\: -\alpha < \theta - \tfrac{\pi}{2} < \alpha\right\}\subseteq \mathbb{R}^2.   
\end{equation}

\begin{prop}\label{limit_line_up}
Suppose that $u: V_\alpha \to \mathbb{R}$ is a bounded solution to $Lu = P$. Then the upward limit at infinity exists along each vertical line, i.e. $\exists K_{\infty}\in\mathbb{R}$ s.t. for fixed $x_2$,
\begin{equation}\label{Lu_f_graense}
K_{\infty} = \lim_{x_3\to+\infty} u(x_2,x_3).
\end{equation}
Furthermore, $K_{\infty}$ is independent of $x_2$.
\end{prop}

\begin{proof}
From Lemma \ref{varme_gradient_inhomogent}, together with the bounds on $P$ in \eqref{up_RHS}, we have the following uniform decay estimates for the derivatives of the solution to $Lu = P(\nabla u, \Hess u)$ on upwards wedges:
\begin{align}
&\partial_{i}u = O\left(|x_3|^{-1}\right),\quad x_3\to +\infty\label{1ste_afl}\\
&\partial_{ij}u = O\left(|x_3|^{-2}\right),\quad x_3\to +\infty\label{2nden_afl}
\end{align}

We can therefore write $\partial_3 u = -\Delta u =: g(x_2,x_3)$, where we by \eqref{2nden_afl} know that $g(x_2,x_3)=O\left(|x_3|^{-2}\right)$, as $x_3\to +\infty$. Fixing $x_2$, we can now simply integrate this first order ODE $u' = \partial_3 u = g$ w.rt. $x_3$ along vertical lines. This shows that $\exists K_{\infty}(x_2)\in\mathbb{R}$ such that:
\begin{equation}
u(x_2,x_3) = K_{\infty}(x_2) + I(x_2,x_3),
\end{equation}
the last term being the integral expression $I(x_2,x_3):= -\int^\infty_{x_3} g(x_2,s)ds$. Using the decay of $g$, this integral is of order $O(\frac{1}{|x_3|})$ as $x_3\to +\infty$, establishing the limit in \eqref{Lu_f_graense}. It is now also not hard to see that $K_\infty(x_2)$ is independent of $x_2$. E.g., if two such values were different, by the mean value theorem, this would violate the gradient bounds in \eqref{1ste_afl}.
\end{proof}

\subsection{Upwards exponential limits}\label{upwards-limits-exp}

In this section, we prove that the limit in Proposition \ref{limit_line_up} is in fact exponentially fast. Namely:

\begin{thm}[Exponential convergence on upwards wedges]\label{thm:exponential}
Consider translating horizontal graphs as in Proposition \ref{limit_line_up}. Then there exists a wedge $V_{\alpha'}\subseteq V_\alpha$ of possibly steeper angle and a number $\beta>0$ depending only on $\alpha'$, and a constant $C$ depending only on $\|u\|_\infty$ and $K_\infty$, such that the following bound holds, for $p_0$ the vertex of $V_{\alpha'}$:
\begin{equation}\label{up_exp_bound}
 \left|u(p) - K_\infty\right|\leq Ce^{-\beta |p-p_0|},\quad p\in V_{\alpha'}.
\end{equation}
\end{thm}

\begin{rem}
 The dependence of $\beta>0$ on the wedge angle in Theorem \ref{thm:exponential} is necessary, and $\beta>0$ could in fact be arbitrarily small, as seen by considering the tilted reaper solitons.
\end{rem}

We will do this by upgrading the decay we already proved in the previous sections, by using a technical barrier argument. To understand where the barrier below comes from geometrically, consider the family of tilted grim reaper surfaces. These would serve ideally as nonlinear barriers, except with respect to the boundary conditions needed. I.e. we could proceed if only we could combine two such in a linear combination - which is of course not allowed in general for such a nonlinear equation. In a sense, the $\Delta$-wing solutions are such nonlinear combinations of two tilted grim reaper cylinders, but these solutions are then not explicit in their construction.

This issue can, however, be solved by using the following geometric ideas: Note that the Killing fields generated by a family of tilted grim reapers cylinders, gotten by appropriately translating one fixed surface in such a way that there is a fixed straight line boundary, are given via functions of the form
\begin{equation}
 w^\pm_{\alpha,\beta} = e^{-\beta (x_3 \pm \alpha x_2)},
\end{equation}
which are thus $L$-harmonic $L w^\pm_\beta = 0$, provided we match the power as $\beta := \frac{1}{1 + \alpha^2}$, corresponding of course also to the respective tilted reapers' geometry. Taking now a linear combination of these solution to the linear equation, we can state and prove the following generally useful fact about these functions as barriers for the original nonlinear problem in \eqref{TSE}:

\begin{lem}\label{lem:superbarrier}
Let $\alpha\geq 0$ and define for $V_\alpha :=\{x_3 + \alpha x_2\geq 0\}\cap \{x_3 - \alpha x_2 \geq 0\}$:
\begin{equation}\label{superbarrier}
 w_\alpha(x_2,x_3) := e^{-\frac{x_3 + \alpha x_2}{1+\alpha^2}} + e^{-\frac{x_3 - \alpha x_2}{1+\alpha^2}},\quad \mathrm{on} \quad V_\alpha
\end{equation}

Then $w_\alpha > 0$ is $L$-harmonic $Lw_\alpha = 0$ and satisfies $w_\alpha > 1$ on the boundary $\partial V_\alpha$.

Furthermore the following inequality holds, for any $\alpha\geq 0, \forall \alpha, A\geq 0, K\in\mathbb{R}$:
\begin{equation}\label{HeldigKartoffel}
\diverg\left(\frac{\nabla (Aw_\alpha + K)}{\sqrt{1+|\nabla (Aw_\alpha + K)|^2}}\right) + \frac{\partial _{x_3} (Aw_\alpha + K)}{\sqrt{1+|\nabla (Aw_\alpha + K)|^2}} \leq 0,\quad \mathrm{on}\quad V_\alpha,
\end{equation}
i.e. functions of the form $Aw_\alpha + K$ are supersolutions of the full quasilinear equation on the wedge $V_\alpha$ (see \eqref{def:VA} for the definition.)
\end{lem}

\begin{proof}[Proof of Lemma \ref{lem:superbarrier}]
Computing directly shows that $L(Aw_\alpha + K) = 0$. Thus, plugging the function $Aw_\alpha + K$ into Equation \eqref{RHS}, the claim is equivalent to the inequality $-P(\nabla (Aw_\alpha + K),\Hess (Aw_\alpha + K)) \leq 0$, which in turn follows from the convexity of the function $Aw_\alpha$, when $A\geq 0$.
\end{proof}

\begin{rem}
Note that the maximum principle on upward-directed $V_\alpha$ for the class of $L^\infty$ functions fails, in the following sense: The constant solution $u=1$ on $V_\alpha$ obeys $0\leq u \leq w_\alpha$ on $\partial V_\alpha$, but of course $u$ does not stay on one side of the decaying barrier $w_\alpha$ throughout $V_\alpha$. In the course of the proof below, we show that a maximum principle however does hold, in the narrower class of functions with prescribed asymptotics at infinity.
\end{rem}

\begin{proof}[Proof of Theorem \ref{thm:exponential}]
From the proof of Proposition \ref{limit_line_up}, we have the uniform asymptotics
\begin{equation}\label{known_decay}
 \left|u(p) - K_\infty\right|\leq \frac{C}{|p|},\quad p\in V_{\hat{\alpha}},
\end{equation}
which can arranged by possibly taking a steeper parameter $\hat{\alpha} < \alpha$.

Let now $A:= \sup_{V_{\hat{\alpha}}} u - K_\infty\geq 0$.
The properties above entail:
\begin{equation}\label{u_w_arranged}
\begin{split}
    &u(p)\leq Aw_{\hat{\alpha}}(p) + K_\infty ,\quad p\in \partial V_{\hat{\alpha}},\\
    &Aw_{\hat{\alpha}}(p) + K_\infty \to K_\infty,\quad\mathrm{when}\quad |p|\to\infty\quad \mathrm{in}\quad V_{\hat{\alpha}}.
\end{split}
\end{equation}

Consider now the compact (with boundary) graphs obtained by intersecting $V_{\hat{\alpha}}$ with balls $B_R(p_0)$ for arbitrary large radius $R>0$, based at $p_0$ the tip point of $V_{\hat{\alpha}}$. I.e. we now take the restrictions $u_{\mid V_{\hat{\alpha}}\cap B_R(p_0)}$.

Let an arbitrary $\varepsilon > 0$ be given.
The properties \eqref{known_decay}-\eqref{u_w_arranged} imply that there exists $R_0>0$ large enough, depending on $C$ and $\varepsilon$, so that for all $R\geq R_0$:
\begin{equation}
u_{\mid \partial(V_{\hat{\alpha}}\cap B_R(p_0))} \leq Aw_{\hat{\alpha}\mid \partial(V_{\hat{\alpha}}\cap B_R(p_0))} + K_\infty + \varepsilon.
\end{equation}

Invoking now the strong maximum principle for divergence form quasilinear partial differential inequalities, in the form of e.g. \cite[Theoorem 1]{Serrin70}, applied to the quasilinear solution $u$ and the barrier $Aw_{\hat{\alpha}} + K_\infty + \varepsilon$, which by Lemma \ref{lem:superbarrier} is superharmonic for the quasilinear equation, i.e. \eqref{HeldigKartoffel}, it follows that in fact
\begin{equation}
    u(p) \leq Aw_{\hat{\alpha}} + K_\infty + \varepsilon, \quad p\in V_{\hat{\alpha}}\cap B_R(p_0).
\end{equation}
Namely, otherwise the compact graph with boundary $u_{\mid V_{\hat{\alpha}}\cap B_R(p_0)}$ could be translated in the negative $e_1$-direction for a last interior point of touching with the barrier function $Aw_{\hat{\alpha}} + K_\infty + \varepsilon$ (whose boundary values would be shifting strictly away in this step). Letting $R\nearrow\infty$ now shows that 
\begin{equation}
u(p) \leq Aw_{\hat{\alpha}} + K_\infty + \varepsilon, \quad p\in V_{\hat{\alpha}},
\end{equation}
and finally by $\varepsilon \searrow 0$, and repeating with the symmetric barrier gotten from the function $-w_{\hat{\alpha}}$:
\begin{equation}
K_\infty - Aw_{\hat{\alpha}} \leq u \leq K_\infty + Aw_{\hat{\alpha}}, \quad \mathrm{on}\quad V_{\hat{\alpha}}.
\end{equation}
By the behavior as $|p|\to +\infty$ of the function $w_{\hat{\alpha}}$ in \eqref{superbarrier}, in the domain $V_{\alpha'}$, for any steeper parameter $\alpha' < \hat{\alpha} < \alpha$, the conclusion \eqref{up_exp_bound} now follows.
\end{proof}

\subsection{Sideways limits}\label{sideways-limits-sec}
For the general case of downward limits, another property that we need to verify in order to apply the study of the $L$-operator is the existence of sideways limits. In this section, we give an analytic proof of this fact. See also Proposition \ref{prop:osc_persists} below.

\begin{prop}[Sideways limits for the translating soliton equation]\label{sideways_limits}
 Suppose that $\Omega$ is a downward domain of the type considered in Theorem \ref{structure-thm}, i.e. the complement of an upward infinite rectangle, whose sides are also contained in upward wedge domains in $\Sigma$. Then sideways limits exist, are independent of the height $x_3$, and agree with the respective upward limits $c_{\pm}$. That is:
\begin{equation}
\forall x_3\in\mathbb{R}:\quad \lim_{x_2\to \pm\infty} u(x_2,x_3) = c_{\pm}.
\end{equation}
\end{prop}

The proof will follow from the following stronger statement:

\begin{prop}[Removable singularities at infinity for translators]\label{cor:remov_sing}
As in Proposition \ref{sideways_limits} the limits exists along all rays in $\Omega$, with the exception of the straight downward direction, and coincide with the corresponding upward limits $c_{\pm}$.

In fact, given any tilted upper half-space of finite negative slope contained in $\Omega$, true limits hold with respect to approach to infinity from within this half-space. In this sense a (restricted) removable singularities theorem holds true at infinity.
\end{prop}

\begin{proof}[Proof of Proposition  \ref{cor:remov_sing} and \ref{sideways_limits}]
   Note that we may WLOG assume, considering by symmetry right-side limits $x_2\to\infty$, by translation and by possibly restricting to a subdomain, that the domain has the special form $\Omega := \{\alpha x_2 + x_3 \geq 0\}$, where $\alpha >0$, which in particular avoids any corners, therefore simplifying the geometry and calculus involved in the proof.

    Using the non-standard Hessian bounds from Proposition \ref{ImprovedHessian}, in Equation \eqref{Hess_up_improved} for domains with the upward sausages property, together with the quasilinear equation \eqref{TSE} and the definition of the drift Laplacian \eqref{def:L}, we derive estimates of the form
    \begin{equation}\label{d_3_decay}
    |\partial_{3} u|\leq C\left[\dist(\cdot, \partial\Omega)\right]^{-2},\quad \dist(\cdot,\partial \Omega) \geq 1, 
    \end{equation}
    where the constant $C$ only depends on the domain and on $\|u\|_\infty$. The simplifying assumptions on $\Omega$ arranged above ensure that throughout $\Omega$ the elementary formula $\dist(\cdot, \partial\Omega) = \frac{1}{\sqrt{\alpha^2 + 1}}|\alpha x_2 + x_3|$ holds.

    Now, since the domain $\Omega$ contains upward wedge domains, we indeed have limits $c_{+}$ in the upward direction, independently of $x_2$, and we may therefore from \eqref{d_3_decay} write a representation formula simply by integrating back down from infinity,
    \begin{equation}\label{rep_formula}
    u(x_2,x_3) = c_{\pm} - \int^{\infty}_{x_3}\partial_{3}u(x_2,x_3')dx_3',
    \end{equation}
    which is well-defined by \eqref{d_3_decay}, for the solution $u$ to $Lu = P$ in \eqref{TSE}.

    The expression \eqref{rep_formula} can now also easily be estimated, when $\alpha x_2 + x_3\geq \sqrt{\alpha^2 + 1}$:
    \begin{equation}\label{est:down_to_zero}
    \left|u(x_2,x_3) - c_{\pm}\right|\leq \int_{x_3}^\infty \frac{C}{\alpha^2 + 1}\frac{1}{\left(\alpha x_2 + x_3'\right)^{2}}dx_3' = \frac{C'}{\alpha x_2 + x_3},
    \end{equation}
    where $C'$ again depends only on $\Omega$, $\alpha$ and $\|u\|_\infty$.
    
    Taking finally, for each fixed $x_3$ value, $x_2\to\infty$ in \eqref{est:down_to_zero}, this completes the proof that sideways limits exist and agree with the respective upward limits:
    \begin{equation}
    \forall x_3\in\mathbb{R}:\quad \lim_{x_2\to \pm\infty}u(x_2,x_3) = :c_{\pm}.
    \end{equation}

\end{proof}

\begin{remark}
Note in the proof of Proposition \eqref{sideways_limits} the important subtlety that the immediately natural quantity to consider, $\partial_2 u$, in itself does not a priori decay strongly enough to be integrated back in from infinity, with respect to $x_2$.
Namely, we know at most that $|\partial_2 u| \leq C\left[\dist(\cdot, \partial\Omega)\right]^{-1}$, as proven in \eqref{grad_up} from Proposition \ref{grad_estimate_quasilinear_all}, which deteriorates to an order worse than the estimate in \eqref{d_3_decay}, which is nonintegrable in $x_2$ on half-lines.

Thus, the above proof in a crucial way makes use of the fact that the upward-pointing domains we consider are not only left or right half-planes, but that they in fact have the special shape as in Figures \ref{fig:decomp-pitchfork}--\ref{fig:Xi1}, so that they contain downward slanted (of finite slope) half-planes.
\end{remark}

\begin{remark}\label{rem:down_discontinuity}
Note also that the straight downward limits do however not have to coincide with the upward and ray limits in Proposition \ref{sideways_limits} and Proposition \ref{cor:remov_sing}, and the limit value along rays may in general exhibit a jump discontinuity at the exact downward direction $-\ee_3$. In fact, this happens in the pitchfork example. See Figures \ref{fig:wings}-\ref{fig:decomp-pitchfork}, as well as Proposition \ref{prop:planar} and the further discussions there.
\end{remark}

\subsection{Downward limits for the translating soliton equation}

\subsubsection{Limits for $L-$harmonic functions on lower half-planes}

We first show downward limits in the $L$-harmonic case.

\begin{prop}[$L$-harmonic limits]\label{homog_down}
Consider $u: \Omega_b\to\mathbb{R}$ a bounded $L$-harmonic function on a lower half-plane $\Omega_b:=\{x_3 < b\}$, such that sideways limits exist,
i.e. $\lim_{x_2\to\pm\infty} u(x_2,b) =: c_{\pm}$.

Then for every $x_2$ the limit $\displaystyle \lim_{x_3\to-\infty} u(x_2,x_3)$ exists and is independent of $x_2$.
\end{prop}

For the proof we will use the following result of independent interest, which links the global spatial behavior in the 2-dimensional bounded $L$-harmonic problem to the time-relaxation of a 1-dimensional heat problem with $L^\infty$ data.

Note that for the applications below, we will need good global control of the constant which the traces of the $L$-harmonic function relaxes to at negative infinity, which \eqref{MagiskFormel} provides.

\begin{prop}[Global relaxation of $L$-harmonic $L^\infty$ functions to 1d heat solutions]\label{GlobalStab}
\noindent{}\\
Let $\Omega_b:=\{x_3 < b\}$. There exists a numerical constant $c_0$ (e.g. $c_0 = 5$ works) s.t. for every $u\in C^2(\bar{\Omega}_b)\cap L^\infty(\Omega_b)$ with $Lu=0$, where $L = \DriftL$ is the drift Laplacian, the following estimates hold:
\begin{equation}\label{MagiskFormel}
\left|u(x_2,x_3) - \frac{1}{\sqrt{-4\pi x_3}}\int_{-\infty}^\infty e^{\frac{(x_2'-x_2)^2}{4x_3}}u(x_2',0)dx_2'\right| \leq \frac{c_0\|u\|_\infty}{\sqrt{-x_3}},
\end{equation}
for all $x_2\in\mathbb{R}$ and $x_3\leq -1$.
\end{prop}

\begin{rem}
While simpler estimates would suffice to show convergence in the cases we need for the uniqueness of limits,
we will also be needing these more concise asymptotics in Proposition \ref{GlobalStab} for the construction of the counterexample, Theorem \ref{WobbleEksempel}, including in the case where no limits exist.
\end{rem}

Clearly, the estimate from the proposition implies, in particular:
\begin{equation}
\begin{split}
 \liminf_{x_3\to -\infty} u(x_2, x_3) &= \liminf_{x_3\to -\infty} \frac{1}{\sqrt{-4\pi x_3}}\int_{-\infty}^\infty e^{\frac{(x_2'-x_2)^2}{4x_3}}u(x_2',0)dx_2',\\
 \limsup_{x_3\to -\infty} u(x_2, x_3) &= \limsup_{x_3\to -\infty} \frac{1}{\sqrt{-4\pi x_3}}\int_{-\infty}^\infty e^{\frac{(x_2'-x_2)^2}{4x_3}}u(x_2',0)dx_2'.
\end{split}
\end{equation}

Thus, we see that the (non-)existence of limits of the 2-dimensional $L$-harmonic function as $x_3\searrow -\infty$ is identical for the $L$-harmonic $u$ and for its associated 1-dimensional heat problem with $L^\infty$ data on the real line equal to the boundary values of $u$. Thus, we conclude the following corollary:

\begin{cor}
 The downwards limits of a bounded $L$-harmonic function $u$ on the lower half-plane exist if and only if the 1-dimensional heat problem on $\mathbb{R}$, with $L^\infty$ initial data at time zero given by the boundary values of $u$ along the line $\mathbb{R}\times\{0\}$, eventually relaxes to a constant in the forward time direction.
\end{cor}

Consulting now the classical literature (e.g., Denisov-Zhikov \cite{DZ}) on relaxation of the Euclidean heat equation (with not necessarily decaying data, but e.g. $L^\infty$), we then conclude the following sharp criterion for existence of limits:

\begin{cor}\label{down_limit}
 Given a bounded $L$-harmonic function $u$ on the lower half-plane.
 
 Then downward limits $\displaystyle \lim_{x_3\to-\infty} u(x_2,x_3)$ exist if and only if the interval averages of the boundary data, defined by
 \begin{equation}\label{interval_averages}
   I_a := \frac{1}{2a}\int_{-a}^a u(x_2',0)dx_2',
 \end{equation}
have a limit at infinity, say $\displaystyle \ell:=\lim_{a\to+\infty}I_a$. In such case, also:
\begin{equation}
\lim_{x_3\to-\infty} u(x_2,x_3) = \ell.
\end{equation}
Thus, if downward limits exist, they are independent of $x_2$.

Furthermore, note that the limiting condition in \eqref{interval_averages} is in particular satisfied if sideways limits are known to exist, where we can set
\begin{equation}
 c_{\pm} := \lim_{x_2\pm\infty} u(x_2,0),
\end{equation}
to conclude that all downward limits exist and are equal to the average:
\begin{equation}\label{gennemsnit}
 \forall x_2\in\mathbb{R}:\: \lim_{x_3\to-\infty} u(x_2,x_3) = \frac{c_+ + c_-}{2}.
\end{equation}
\end{cor}

\begin{proof}[Proof of Proposition \ref{homog_down}]
The proof now follows immediately from Corollary \ref{down_limit}.
\end{proof}

The last corollary also gives rigorous counterexamples, showing that $L^\infty$ smooth solutions to $Lu = f$, even when $f=0$ and with, say, gradient and higher derivative decay estimates on the boundary data, is not sufficient to ensure that limits exist along downward lines.
For a counterexample in the full quasilinear case, which we relegate to the below Theorem \ref{WobbleEksempel}.

\begin{proof}[Proof of \ref{GlobalStab}]
To obtain a representation formula for the $L$-harmonic functions on a halfspace,
one can work with the Green's function directly, or refer to the known formula for the Yukawa equation (also known as the screened Poisson equation) derived by Duffin \cite{Duffin}, using that $L$ is conjugate to $\hat{L} = \Delta - \frac{1}{4}$, see \eqref{KonjugTrick}. I.e. it is the $\mu = \frac{1}{2}$ case of $(\Delta - \mu^2)v = 0$, meaning that
\begin{equation}
 v(x_2,x_3) = -\frac{1}{\pi}\int_{-\infty}^\infty x_3\: p_\mu\hspace{-3pt}\left[(x_2 - x_2')^2 + x_3^2\right] v(x_2',b) dx_2',
\end{equation}
where the Duffin kernel is defined as:
\begin{equation}
 p_\mu[\xi] := \int_ 0^\infty \cos(\mu z)(z^2 + \xi)^{-3/2}dz.
\end{equation}
Note that in fact, by standard formulae, the integration in Duffin's kernel can be performed in terms of Bessel function, as:
\begin{equation}
 p_\mu[\xi] = \frac{\mu K_1\left(\mu\sqrt{\xi}\right)}{\sqrt{\xi}}.
\end{equation}

Thus, we here derive the following general explicit Poisson-Duffin representation formula for bounded $L$-harmonic functions on the lower half-plane $\Omega_b$:
\begin{equation}\label{L_Poisson-Duffin}
 u(x_2,x_3) = -\frac{e^{\frac{b-x_3}{2}}}{2\pi}\int_{-\infty}^\infty \frac{x_3 K_1\left(\frac{\sqrt{(x_2 - x_2')^2 + x_3^2}}{2}\right)}{\sqrt{(x_2 - x_2')^2 + x_3^2}} u(x_2',b) dx_2',
\end{equation}
Uniqueness of this representation follows either from an argument of Duffin \cite{Duffin} or from an application of the maximum principle, which does hold for $L$ on lower half-planes in the class of bounded Dirichlet solutions, namely Lemma \ref{Phragmen_Lindelof_L_infty}. Existence also follows from Duffin, with the caveat that, after conjugating back to get an $L$-harmonic solution $u$, the factor $e^{-x_3}$ prevents us from knowing that $u$ is in $L^\infty$, before an additional estimation is performed. Alternatively, using constant solutions as barriers with a standard (boundary) Schauder estimates and an exhaustion argument, existence of a bounded Dirichlet solution is easily seen, which is then unique by Lemma \ref{Phragmen_Lindelof_L_infty}.

Let us use the suggestive notation $t:= -x_3$, which will be useful below when connecting the $L$-harmonic functions to heat solutions in one dimension lower, and also WLOG take $b=0$. I.e. we are considering the $t\to+\infty$ limits of the following:
\begin{equation}
 u(x_2,-t) = \frac{te^{t/2}}{2\pi}\int_{-\infty}^\infty \frac{K_1\left(\frac{\sqrt{(x_2 - x_2')^2 + t^2}}{2}\right)}{\sqrt{(x_2 - x_2')^2 + t^2}} u(x_2',0) dx_2',
\end{equation}

Using now the standard asymptotics for the Bessel function $K_1$ (see e.g. NIST's Digital Library of Mathematical Functions, §10.25.3), that
$$K_1(z) = \sqrt\frac{\pi}{2z}e^{-z} + \ldots,
$$
we see that:
\begin{equation}\label{Psi}
F:=\frac{te^{t/2}}{2\pi}\frac{K_1\left(\frac{\sqrt{(x_2 - x_2')^2 + t^2}}{2}\right)}{\sqrt{(x_2 - x_2')^2 + t^2}} = J + \Psi,
\end{equation}
where $\Psi$ is lower order as $t\to \infty$ (made precise below), and
\begin{equation}
J := \frac{t}{2\sqrt{\pi}}\left[(x_2 - x_2')^2 + t^2\right]^{-3/4}\exp\left(\frac{t-\sqrt{(x_2 - x_2')^2 + t^2}}{2}\right).
\end{equation}

To estimate the error in \eqref{Psi},
recall the classical bound $$\left|K_1(z) - \sqrt\frac{\pi}{2z}e^{-z}\right|\leq \frac{3}{8}\frac{1}{z}\sqrt\frac{\pi}{2z}e^{-z}, \quad \mbox{for $z>0$.}$$  Note also the elementary integral bounds, for $\beta > 1$:
\begin{equation}\label{poly_integrals}
\int_{-\infty}^{\infty} \left(s^2 + t^2\right)^{-\beta/2} ds = \sqrt{\pi}\frac{\Gamma(\frac{\beta - 1}{2})}{\Gamma(\frac{\beta}{2})}t^{1-\beta} =: c_\beta t^{1-\beta},
\end{equation}

Using \eqref{poly_integrals} find the estimates:
\begin{equation}
\begin{split}
 \int_{-\infty}^{\infty} |\Psi||u(x_2',0)|dx_2' &\leq  \frac{3}{8\sqrt{\pi}} \|u\|_{\infty} \int_{-\infty}^{\infty} \frac{t e^{t/2}e^{-\frac{1}{2}\sqrt{(x_2 - x_2')^2 + t^2}}}{\left[(x_2 - x_2')^2 + t^2\right]^{5/4}}dx_ 2'\\
 &\leq \frac{3c_{5/2}}{8\sqrt{\pi}} \|u\|_{\infty}\frac{1}{\sqrt{t}} = \frac{3\Gamma(\tfrac{3}{4})}{8\Gamma(\tfrac{5}{4})} \|u\|_{\infty}\frac{1}{\sqrt{t}}.
\end{split}
\end{equation}
We now cut up the integration based on the value of $t$, e.g. at $x_2' = x_2 \pm t^{\alpha}$, for an arbitrary $\alpha\in (\frac{1}{2}, \frac{5}{8}]$, where $N(\cdot)$ is as in \eqref{Psi}:

\begin{equation}\label{cut_integralet}
 \int_{-\infty}^\infty N\:u(x_2',0) dx_2' =: \int_{x_2-t^{\alpha}}^{x_2+t^{\alpha}} N\: u(x_2',0)dx_2' + \mathcal{R}.
\end{equation}

In the first term we may further approximate the function $J$, using that $(x_2 - x_2')^2$ is of order at most $t^{2\alpha}$ and hence small compared to $t^2$, and to leading order we now see that the 1-dimensional heat kernel appears, in the variable $t = -x_3$, i.e.:
\begin{equation}
 J = \frac{1}{\sqrt{4\pi t}}e^{-\frac{(x_2 - x_2')^2}{4t}} + \Upsilon.
\end{equation}
Here we of course expect, and will soon show, that the integral of $\Upsilon$ decays suitably.

Note that, by simply bounding by the first omitted term in the binomial expansion, we have the following bound, for $t>0$:
\begin{equation}\label{heat_indmad}
 \left|\frac{t-\sqrt{(x_2 - x_2')^2 + t^2}}{2} - \left(-\frac{(x_2-x_2')^2}{4t}\right)\right|\leq \frac{(x_2-x_2')^4}{16t^3},
\end{equation}
as well as as the elementary inequality, for arbitrary $x,y\in\mathbb{R}$ s.t. $|x-y|<1$
\begin{equation}\label{Kalkulus}
\left|e^x - e^y\right|\leq 2|x-y|e^y.
\end{equation}
Letting now
\begin{equation}\label{halvvejs}
J_1 := \frac{t}{2\sqrt{\pi}}\left[(x_2 - x_2')^2 + t^2\right]^{-3/4}e^{-\frac{(x_2 - x_2')^2}{4t}},
\end{equation}
we estimate, using \eqref{heat_indmad} and \eqref{Kalkulus}, that for $t\geq1$, and using $\alpha \leq \tfrac{5}{8}$:
\begin{equation}
\int_{x_2 - t^\alpha}^{x_2 + t^\alpha} |J - J_1||u(x_2',0)|dx_2' \leq t^{4\alpha - 3}\|u\|_\infty \int_{-t^\alpha}^{t^\alpha} \frac{1}{\sqrt{4\pi t}}e^{-\frac{s^2}{4t}} ds \leq \frac{\|u\|_\infty}{\sqrt{t}}.
\end{equation}

Using now the elementary bounds, for $t>0$:
\begin{equation}
 \left|\frac{t}{2\sqrt{\pi}}\left[(x_2 - x_2')^2 + t^2\right]^{-3/4} - \frac{1}{\sqrt{4\pi t}}\right|\leq \frac{3}{4}\left(\frac{x_2-x_2'}{t}\right)^2\frac{1}{\sqrt{4\pi t}},
\end{equation}
we may estimate the difference from \eqref{halvvejs} to the heat solution as follows, for $t\geq1$:
\begin{equation}
 \int_{x_2 - t^\alpha}^{x_2 + t^\alpha} \left|J_1 - \frac{1}{\sqrt{4\pi t}}e^{-\frac{(x_2 - x_2')^2}{4t}}\right||u(x_2',0)|dx_2' \leq \tfrac{3}{4} t^{2\alpha - 2} \|u\|_{\infty}\int_{- t^\alpha}^{t^\alpha}\frac{1}{\sqrt{4\pi t}}e^{-\frac{s^2}{4t}}ds\leq \tfrac{3}{4} \|u\|_{\infty}t^{-\tfrac{3}{4}}.
\end{equation}

Now for the final error, from cutting off in \eqref{cut_integralet}. Seeing as the best possible bound in terms of decay rates will come from the largest permissible value of $\alpha$, we might as well now fix $\alpha = \frac{5}{8}$. Then we estimate the size of $\mathcal{R}$ \eqref{cut_integralet} as follows (bounding $K_1$ by its leading order term):
\begin{equation}
\begin{split}
 |\mathcal{R}| \leq & \frac{te^{t/2}}{2\pi}\|u\|_{\infty}\int_{[-t^{\alpha}, t^{\alpha}]^c} \frac{K_1\left(\frac{\sqrt{(x_2 - x_2')^2 + t^2}}{2}\right)}{\sqrt{(x_2 - x_2')^2 + t^2}}dx_2'\\
 \leq & 2\|u\|_{\infty}\int_{t^{5/8}}^{\infty}J dx_2' \leq \frac{2}{\sqrt{\pi}}\|u\|_{\infty}\int_{t^{5/8}}^{\infty} t(s^2 + t^2)^{3/2} ds\\
 & =    \frac{2\|u\|_{\infty}}{\sqrt{\pi}}\frac{1}{t}\left(1 - \frac{t^{5/8}}{\sqrt{t^2 + t^{5/4}}}\right)\\
 & \leq \frac{2\|u\|_{\infty}}{\sqrt{\pi}}\frac{1}{t}.
\end{split} 
\end{equation}

\end{proof}

\subsubsection{Limits for $Lu=P$ on lower half-planes}
\begin{prop}[Quasilinear case of downwards limits]\label{inhomog_down}
Downward limits exist, assuming sideways limits exist. More precisely, assume that the side limits are $\displaystyle c_{\pm} = \lim_{x_2\to \pm\infty} u(x_2,x_3)$.
Then:
\begin{equation}
\forall x_2\in\R: \quad \lim_{x_3\to -\infty} u(x_2,x_3) = \frac{c_{-} + c_{+}}{2}.
\end{equation}
\end{prop}
\begin{proof}
Assume the sideways limits $\displaystyle \lim_{x_2 \to \pm\infty}u(x_2,x_3)$ are known to exist,
and assume the bounds $|P(\nabla u, \Hess u)|\leq c_P \,|x_3|^{-2}$ for a constant $c_P$ (as in Theorem \ref{P_decay_estimates}). To treat the translating horizontal graphs equation $Lu = P(\nabla u, \Hess u):=P$, we will now combine the limit behavior for $L$-harmonic functions together with a barrier argument for the inhomogeneous term.

We therefore first solve the homogeneous Dirichlet problem for $L$ on $\Omega_b := \{x_3<b\}$, where for now we fix $b<0$: Via exhaustion by compact domains obtained by intersection with balls, $\Omega_b = \cup_{R>0} \Omega_b\cap B_R(0)$, and using constant solutions as barriers, then together with standard local (boundary) Schauder estimates and extracting a solution by a standard compactness argument, we get a smooth bounded function $u_0$ such that $Lu_0 = 0$ in $\Omega_b$ and $u_0=u$ on the boundary.
In particular:
\begin{equation}\label{L_harmonic_limits}
 \lim_{x_2\to\pm\infty} u_0(x_2,b) = c_{\pm}.
\end{equation}

Thus we have decomposed $u = u_0 + \hat{u}$, where $L\hat{u} = P(\nabla u, \Hess u)$ and $\hat{u}$ is zero on the boundary $\partial\Omega_b$.

We now consider the following barrier functions, for constants $b<0$:
\begin{equation}\label{Ei_barrier}
\begin{split}
&w_b: \Omega_b \rightarrow \R,\\
&w_b = c_P\left(e^{-b}\mathrm{Ei}(b) - e^{-x_ 3}\mathrm{Ei}(x_3)\right),
\end{split}
\end{equation}
where $\mathrm{Ei}(x_3) := \int_{-\infty}^{x_3}t^{-1}e^tdt$, for $x_3 < b < 0$, denotes the exponential integral.

Note that the barrier functions $w_b$ solve the following PDE:
\begin{equation}
 Lw_b = c_P  \,|x_3|^{-2}, \quad x_3 < b.
\end{equation}
Given the solution $\hat{u}$ to $L\hat{u} = P(\nabla u, \Hess u)$, we thus see that:
\[
L(\hat{u} + w_b) = P(\nabla u, \Hess u) + c_P \; |x_3|^{-2} \geq P + |P| \geq 0,
\]
so that $\hat{u}+w_b$ is $L$-subharmonic on $\Omega_b$.

Since $\hat{u}=0$ on $\partial \Omega_b$, we have that $\hat u+w_b=0$ along $\partial \Omega_b.$

By the maximum principle for the drift Laplacian $L$ on lower half-planes in Lemma \ref{Phragmen_Lindelof_L_infty}, noting that $\hat{u} + w_b\in L^\infty$, we conclude that the $L$-subharmonic function constructed satisfies $\hat{u} + w_b \leq 0$ in $\Omega_b$.
Thus, by properties of the $w_b$ function, which as $x_3\searrow-\infty$ decreases to the constant limit $e^{-b}\mathrm{Ei(b)} < 0$ (and by considering also $-u$):
\begin{equation}
\forall b < 0:\quad \limsup_{x_3\to -\infty} |\hat{u}(x_2,x_3)| \leq c_P e^{-b}|\mathrm{Ei}(b)|.
\end{equation}
This means also that, since $u = u_0 + \hat{u}$, via our knowledge from Corollary \ref{down_limit} and \eqref{L_harmonic_limits}, that for the $L$-harmonic function $u_0$ the limit $\lim_{x_3\to\infty} u_0(x_2,x_3) = \tfrac{c_- + c_+}{2}$ holds:
\begin{equation}\label{squoze}
\forall b < 0:\quad \limsup_{x_3\to -\infty} \left|u(x_2,x_3) - \tfrac{c_- + c_+}{2}\right| \leq c_P e^{-b}|\mathrm{Ei}(b)|.
\end{equation}

Using now the elementary limit $\lim_{b\to-\infty} \frac{b\mathrm{Ei}(b)}{e^b} = 1$, i.e. that $c_Pe^{-b}\mathrm{Ei}(b) = O\left(|b|^{-1}\right)$ holds, we finally conclude by letting $b\searrow -\infty$ in \eqref{squoze} that:
\[
\lim_{x_3\to -\infty} u(x_2,x_3) = \frac{c_{-} + c_{+}}{2}.
\]
\end{proof}

\subsubsection{Persistent oscillation for $L-$harmonic functions on lower half-planes}

One might naively have thought that the affirmative limit results in the preceding sections, e.g. Proposition \ref{inhomog_down}, entailed a more global limiting behavior, in the down direction. Locally, on any fixed compact interval $\bar{I}\subseteq\R$ one does indeed have, by simply integrating the gradient estimate \eqref{grad_down}, the following oscillation decay, for $u$ any bounded solution to the soliton equation \eqref{TSE}:
\begin{equation}\label{compact_osc_decay}
    \rho > 1:\quad \underset{x_2\in I}{\osc} u(x_2, b-\rho) \leq C\: \frac{\underset{x_2\in I}{\osc} u(x_2,b)}{\sqrt{\rho}},
\end{equation}
where the constant depends on $I$ and $\|u\|_\infty$. Note however that \eqref{compact_osc_decay} itself cannot be used to prove limits going downwards: See namely the counterexamples to the existence of limits below in Theorem \ref{WobbleEksempel}, which we work out in the full quasilinear case.
\begin{figure}
    \centering
    \includegraphics[width=.7\linewidth]{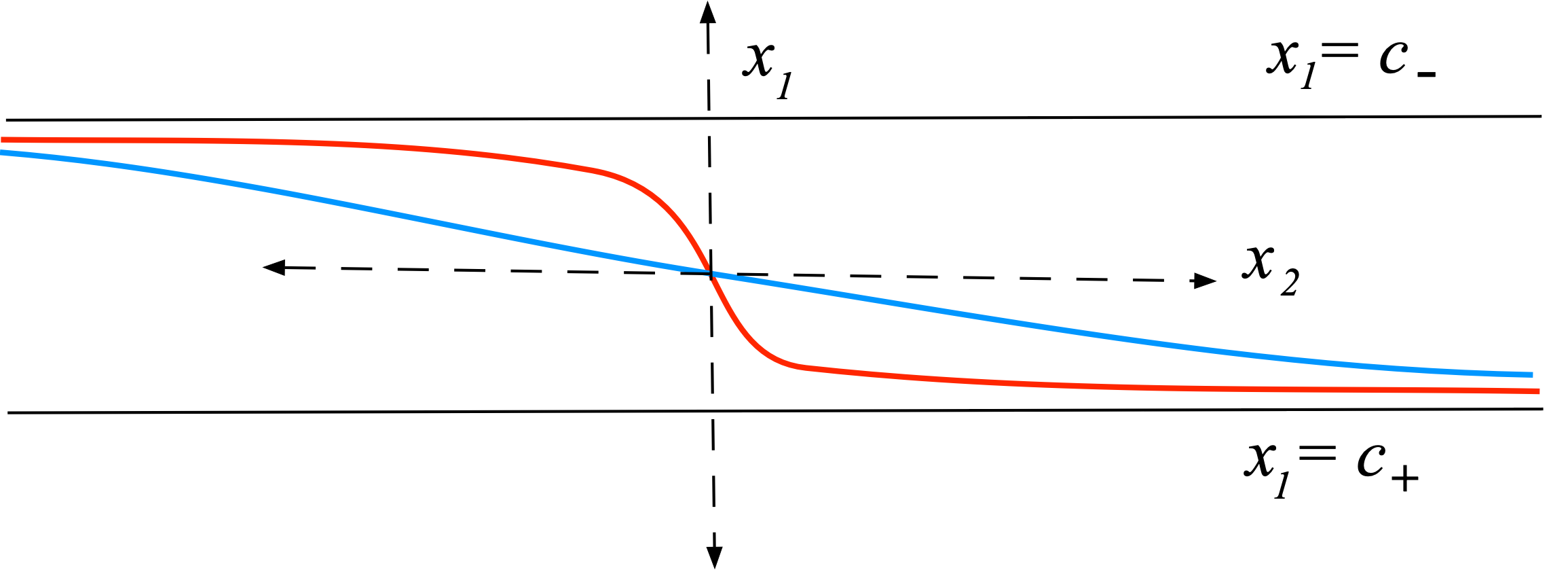}
    \caption{Two different sections of the function $u(x_2,x_3)$ in Proposition  \ref{prop:osc_persists} at different $x_3$-heights. The red curve represents $x_1=u(x_2,b)$ and the blue one represents $x_1=u(x_2,b')$, for $b'<<b$. Observe how the sideways limits remain unchanged, despite the profile flattening out and converging in $C^\infty$ to a constant on all compact scales.}
    \label{fig:osc_persists}
\end{figure}

The convergence is indeed of a more delicate nature, as oscillations do not overall attenuate at any global scale (see Figure \ref{fig:osc_persists}), in the down direction, which we emphasize with the following proposition for the $L$-harmonic case:

\begin{prop}[Failure of decay of global oscillation as $x_3\searrow -\infty$]\label{prop:osc_persists}
Let $\Omega_b: = \{x_3<b\}$ and suppose that $u:\Omega_b\to\R$ is in $L^\infty$ and satisfies $Lu = 0$, as well as
\begin{equation}\label{initial_limits}
 \lim_{x_2\to\pm \infty}u(x_2,b) = c_{\pm}.
\end{equation}
Then all other sideways limits agree, pairwise, with \eqref{initial_limits}, i.e.:
\begin{equation}\label{sideways_L_harmonic}
\forall x_3 < b:\quad \lim_{x_2\to\pm \infty}u(x_2,x_3) = c_{\pm}.
\end{equation}

In particular, if $c_-\leq u(x_2,b)\leq c^+$ holds, and if we consider the sections of $u$ at distances $\rho > 0$ below the boundary of $\{x_3 \leq b\}$:
\begin{equation}
\forall \rho >0: \quad \underset{x_2\in\R}{\osc} u(x_2,b-\rho) = \underset{x_2\in\R}{\osc} u(x_2,b)\quad \left(\: = |c_+ - c_-| \right)
\end{equation}
\end{prop}

\begin{rem}
 With a bit more knowledge, i.e. for larger downward pointing domains such as the ones that we consider in Theorem \ref{structure-thm}, a similar result actually has already been shown (and with a different proof in that case, using only the elliptic-parabolic estimates), in the quasilinear case, in Proposition \ref{sideways_limits}.
\end{rem}

\begin{proof}
Letting $x_3 = -\rho$ in Equation \eqref{L_Poisson-Duffin}, and changing variables to $\xi := x_2' - x_2$, we have:
\begin{equation}
 u(x_2,-\rho) = \frac{1}{2\pi}\int_{-\infty}^\infty \frac{\rho e^{\frac{\rho}{2}} K_1\left(\frac{\sqrt{\xi^2 + \rho^2}}{2}\right)}{\sqrt{\xi^2 + \rho^2}} u(x_2 + \xi,0) d\xi,
\end{equation}
Noting that the kernel in front of $u$ is uniformly in $L^1$ (in fact, has integral equal to unity), then \eqref{sideways_L_harmonic} follows simply from \eqref{initial_limits} via Lebesgue dominated convergence:
\begin{equation}\label{osc_konstant}
\lim_{x_2\to\pm\infty} u(x_2,-\rho) = \frac{c_{\pm}}{2\pi}\int_{-\infty}^\infty \frac{\rho e^{\frac{\rho}{2}} K_1\left(\frac{\sqrt{\xi^2 + \rho^2}}{2}\right)}{\sqrt{\xi^2 + \rho^2}} d\xi = c_{\pm}.
\end{equation}

To verify the second claim, note that $c_-\leq u(x_2,b)\leq c_+$ implies,
that $c_-\leq u(\cdot,\cdot) \leq c_+$ everywhere, e.g. by Lemma \ref{Phragmen_Lindelof_L_infty}, so that the conclusion \eqref{osc_konstant} now follows from \eqref{sideways_L_harmonic}.
\end{proof}

The basic example to have in mind in Proposition \ref{prop:osc_persists} is $u(x_2,0) = -\mathbf{1}_{\{x_2 < 0\}} + \mathbf{1}_{\{x_2 > 0\}}$, as boundary data for a lower half-plane solution, with $c_{\pm} = \pm 1$.

The downwards direction behavior discussed here is of course in stark contrast to the situation when looking upwards (in the direction of motion of the translating soliton), where the above analysis in Proposition \ref{thm:exponential} showed that an $L^\infty$ bound on a wedge is enough to ensure exponentially fast convergence (we did not write this fact out for the upper half-plane, but the same result there works as well, not only for narrower wedges).

\subsection{Periodic and finite total curvature translators, exterior problems} \label{subsec:periodic}

The special cases of periodic solutions and exterior domains (complements in $\mathbb{R}^n$ of bounded domains) are easier to treat as such, the latter ``infinite annulus'' case of these naturally showing up in the finite total curvature situation. Both cases are of course included in our general analysis in the above, and they are in general more well-behaved. We therefore record in this section some easy-to-state corollaries in these special cases.

\begin{cor}\label{cor:periodic}
Suppose $u$ is bounded and solves the translating soliton equation \eqref{TSE}, defined over the entire lower half-plane $\{(0,x_2,x_3):\: x_2\in\mathbb{R}, \: x_3\leq 0\}$ and is $x_2$-periodic, i.e. for some period $T>0$:
\[
\forall x_2\in\mathbb{R},\: x_3\leq 0:\quad u(x_2 + T, x_3) = u(x_2, x_3).
\]
Then all downward limits exist and agree:
\begin{equation}
\exists K_\infty\forall x_2:\quad \lim_{x_3\to-\infty} u(x_2,x_3) = K_\infty.
\end{equation}

The analogous statement holds for upward limits for bounded solutions which are periodic and entire over an upper half-plane.
\end{cor}

\begin{proof}
The same barrier as in the proof of Proposition \ref{inhomog_down} is again used at heights $-b < 0$.

As for upwards limits of solutions sideways periodic on an upper half-plane, it's a simple corollary of the same fact proven about upward wedges.
\end{proof}

As for finite total curvature, this condition does thus not hold for the tilted grim reaper cylinders themselves, finite total curvature is indeed a quite restrictive class. Namely, in our case, tilted reapers (see the many examples in  \cite{HIMW-2, HMW-2, HMW-1, Ng}, e.g. in Figure \ref{fig:wings}) play the basic role as solution ``atoms'', on which more general infinite ends/wings are asymptotically modeled, an idea which large part of the present paper will be devoted to making precise. The general class of complete surfaces $\Sigma$ with $\int_\Sigma |A|^2dA < \infty$, which by Huber's theorem have finite genus and finitely many ends, has been studied extensively in the case of classical minimal surfaces ($H = 0$) in $\R^3$, where they form a very rich and classically important class. For translating solitons the class was studied in \cite{Ilyas}, where Khan showed that such an embedded surface, if it has only one end, must be a flat vertical plane, and in generality is asymptotic to a number of parallel planes, which in turn implies it is a (not necessarily convex) collapsed soliton under our definitions. Note furthermore that finite total curvature in general implies finite entropy, as it firstly implies quadratic area growth by \cite[Corollary 2.11]{Ilyas} (proved using \cite[Proposition  1.3]{Li-Tam}), which by e.g. \cite[Theorem 9.1]{BrianBoundary} is equivalent to finite entropy. Thus, it is all in all clear that, by working in the class of finite entropy, finite genus and collapsed, we are in the present paper considering a strictly richer class which contains within it the whole class of finite total curvature self-translating solitons.

Essential to the finite curvature situation is the study of the exterior problems, where most of the difficulties faced in the general setting are absent. For completeness, we record a corollary for that special case (which recovers a result contained in \cite{Ilyas} on ``infinite annuli''):

\begin{cor}\label{cor:exterior}
Suppose $u$ is bounded and solves the translating soliton equation \eqref{TSE}, defined over an exterior domain $\mathbb{R}^2\setminus \Omega$, for $\Omega$ bounded. Then the limits along all rays exist and agree:
\begin{equation}
\exists K_\infty\in\mathbb{R} \;\forall \oomega\in\mathbb{S}^1: \quad \lim_{s\to \infty} u(s\;\oomega) = K_\infty,
\end{equation}
where we recall that $\mathbb{S}^1=\{(0,x_2,x_3) \in \R^3  \: : \: x_2^2+x_3^2=1\}$.

Furthermore, except when $\oomega = -\ee_3$, this convergence is exponentially fast, with rates depending on $\oomega$.
\end{cor}

\begin{proof}
This follows from Proposition \ref{cor:remov_sing}, or directly from Theorem \ref{thm:exponential}, for all rays except $\oomega = -\ee_3$, by picking appropriately slanted wedge subdomains of $\mathbb{R}^2\setminus\Omega$, which ensures that all limits agree, including sideways. We label this limit $K_\infty$.

For the exceptional direction $\oomega = -\ee_3$, the convergence then follows from Corollary \ref{down_limit}, seeing as the limit value by Equation \eqref{gennemsnit} simply gives $\frac{c_- + c_+}{2}=\frac{K_\infty + K_\infty}{2} = K_\infty$.
\end{proof}

\begin{remark}\label{rem:sq_root_barrier}
In \cite{Ilyas}, the special barrier that was used to show a convergence similar to this special case, Corollary \ref{cor:exterior}, was in reality the square root of the Green's function of the linearized operator $L$, which we exhibited in Equation \eqref{GreensL}--\eqref{u_K}, i.e. $\sqrt{G_L}$.

We note that the natural rate of decay in this nonlinear exterior problem is as the decay of the linear Green's function $G_L$ itself (exponential along all rays except straight down, where the rate is $|x|^{-1/2}$), rather than the slower rates of $\sqrt{G_L}$. This improvement of convergence rates e.g. from $|x|^{-1/4}$ to $|x|^{-1/2}$ (straight down) can indeed also be shown to hold true, with a bit more work.

\end{remark}


\section{Counterexamples to uniqueness of tangent planes at infinity}\label{sec:counterexamples}
In this section, we prove the necessity of the global nature of the proofs of our positive results on uniqueness of tangent planes at infinite times, Theorem \refThmAnospace, where we made essential use of the specific shape of the domains (see Figure \ref{fig:Xi1}). We will do this by constructing counterexamples: There exist bounded translating solitons which are $C^\infty$ graphs over the lower half-space, solving \eqref{TSE}, but for which the subsequential limit planes of $\Sigma + t\ee_3$, as $t\to + \infty$, fail to be independent of the chosen subsequences $\{t_k\}$ tending to infinity.

Such a counterexample thus cannot be the restriction of any translating soliton which is complete, collapsed, embedded, of finite genus and finite entropy, as this would violate Theorem \refThmAnospace. It also cannot be periodic, as seen from Corollary \ref{cor:periodic}, and cannot come as the restriction of any exterior problem solution, as Corollary \ref{cor:exterior} shows. It will of course also have to get asymptotically flat with respect to the distance to the boundary, obeying the gradient and curvature decay rates imposed on it by Propositions \ref{grad_estimate_quasilinear_all} and \ref{ImprovedHessian}. 

Despite these obstructions, using the tools we have developed in the preceding sections, we can now construct a counterexample:

\begin{thm}\label{WobbleEksempel}
There exist smooth bounded solutions $u$ to the translating soliton equation \eqref{TSE}, defined over the lower half-plane $\Omega=\{(0,x_2,x_3):\: x_2\in\mathbb{R}, \: x_3\leq 0\}$, for which

\begin{equation}
\forall x_2\in\mathbb{R}:\quad \limsup_{x_3\searrow\: -\:\infty}u(x_2,x_3) - \liminf_{x_3\searrow -\infty}u(x_2,x_3) \: >\: 0.
\end{equation}

That is, fixing $x_2$, the function $x_3 \mapsto u(x_2,x_3)$ exhibits persistent oscillation which does not attenuate, and hence $u$ fails to have a limit as $x_3\searrow -\infty$.

Furthermore, given any interval $J:=[\alpha,\beta]\subseteq \R$, there exists a bounded solution $u_J$ to \eqref{TSE} defined over the lower half-plane, for which, for every $\gamma\in (\alpha,\beta)$, there exists a corresponding sequence of times $\{t^\gamma_k\}$ such that $t^\gamma_k\nearrow +\infty$ and there's smooth convergence on compact subsets of $\R^3$ to the plane at height $\gamma$:
\begin{equation}
\Graph[u_J] + t_k^\gamma \ee_3\:\longrightarrow\: \{ x_1 = \gamma \}.
\end{equation}
\end{thm}

\begin{proof}
We first address the existence of a solution of \eqref{TSE}, defined over the lower half-space $\Omega$. For that, fix any continuous boundary data $c:\partial \Omega\to\R$ with $\|c\|_\infty< \infty$. 

Consider the three planes $x_3=0, x_1=\|c\|_\infty$ and $x_1=-\|c\|_\infty.$  Denote by $\mathcal{B}_{s}$ the bowl soliton with vertex at $(-\|c\|_\infty,0,-s).$ For all sufficiently large $s$, the planes $\{x_3=0\}, \{x_1=\|c\|_\infty\}, \{x_1=-\|c\|_\infty\},$ and $\mathcal{B}_{s}$ bound a piecewise smooth compact domain in $\R^3,$ which we call $\mathcal{R}(s).$ 

Assuming $s$ is a large constant, we consider in $\partial \mathcal{R}(s)$ the Jordan curve given by 
\[
\Gamma(s)=[\mathcal{B}_s\cap\{x_1=-\|c\|_\infty,\; x_3\leq 0\}]\cup \gamma_s^1\cup\gamma_s^2\cup \text{Graph}(c|_{[y_s^1,y_s^2]}),
\]
where $(0,y_s^1,c(y_s^1))$ and $(0,y_s^2,c(y_s^2))$ are the first point of contact between $\text{Graph}(c)=\{(0,x_2,c(x_2))\;:\;x_2\in\R\}$ and $\mathcal{B}_s\cap\{x_1\geq -\|c\|_\infty,\; x_3=0\},$ that is, they satisfy 
\[
\text{Graph}(c|_{[a,b]})\cap [\mathcal{B}_s\cap\{x_1\geq -\|c\|_\infty,\; x_3=0\}]=\varnothing,
\]
for every $y_s^1<a<b<y_s^2.$ Here $\gamma_s^1$ and $\gamma_s^2$ are the two arcs of $\mathcal{B}_s\cap\{x_1\geq -\|c\|_\infty,\; x_3=0\}$ connecting the endpoints of $\mathcal{B}_s\cap\{x_1=-\|c\|_\infty,\; x_3\leq 0\}$ to $(0,y_s^1,c(y_s^1))$ and $(0,y_s^2,c(y_s^2)),$ respectively (see Figure \ref{bowl-planes}). Note that $\Gamma(s)$ a $\ee_1$-graph. 

\begin{figure}
    \centering
    \includegraphics[width=.6\linewidth]{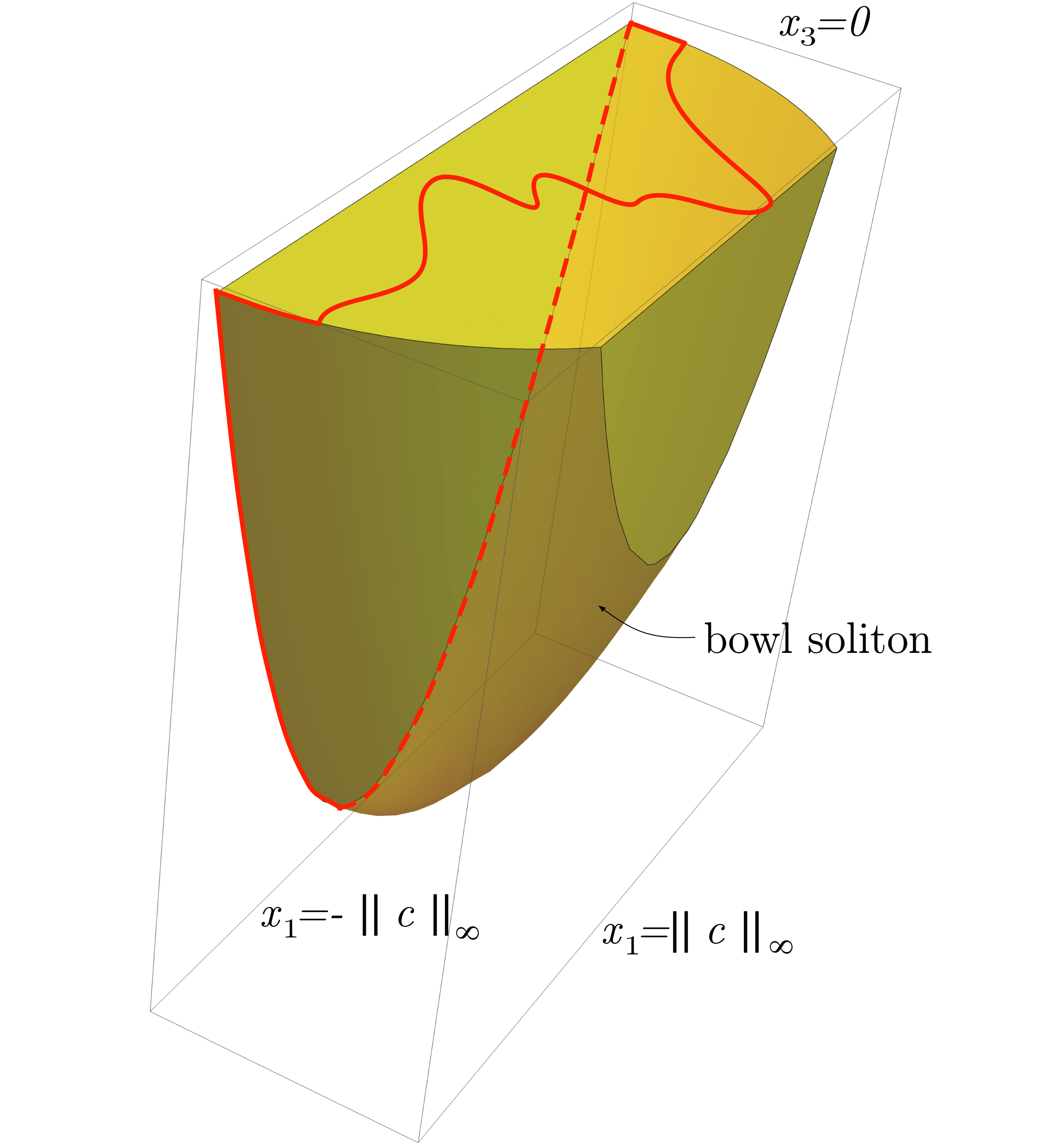}
    \caption{Representation of the domain $\mathcal{R}(s)$ (in yellow), for sufficiently large $s$. The curve $\Gamma(s)$ is represented in red.}
    \label{bowl-planes}
\end{figure}

Noting (see Figure \ref{bowl-planes}) that $H_M > 0$ with respect to the outward normal on the top boundary, while $H_M \equiv 0$ on all other boundaries ($H_M$ denoting the mean curvature with respect to Ilmanen's metric, see Proposition \ref{prop:Nicos} or \eqref{IlmanenMetric}). Hence, the domain $\mathcal{R}(s)$ lies on the mean convex side of $\partial\mathcal{R}(s)$. Furthermore, the curve $\Gamma(s)$ is nullhomotopic in $\mathcal{R}(s)$. By Meeks-Yau \cite{Meeks-Yau} (see also \cite[Theorem 1.1]{spadaro})  we can find an embedded topological disk solution of the translator Plateau problem $\Sigma(s)$ in $\mathcal{R}(s)$ with \begin{equation} \label{eq:Job}\Gamma(s)=\partial\Sigma(s)=\Sigma(s) \cap \partial \mathcal{R}(s).\end{equation} 
We deduce from \eqref{eq:Job} that the $\ee_1$-projection of the interior of $\Sigma(s)$ and the $\ee_1$-projection of $\Gamma(s)$ are disjoint. This fact and the assumption of $\ee_1-$graphicality of $\Gamma(s)$ further imply, via a version of Rad\'o's classical argument, that: 
\begin{Claim}
    $\Sigma(s)$ is an $\ee_1-$graph over the domain $\Omega(s)=(\mathcal{B}(s)+\|c\|_\infty\ee_1)\cap\{x_1=0\}.$ Moreover, if we write $\Sigma(s)=\mathrm{Graph}[u_s,\Omega(s)],$ then the function $u_s$ is smooth at the interior of $\Omega(s)$.\end{Claim}
\begin{proof}The argument is straightforward, but we include it here for completeness. 
We argue by contradiction. Suppose that \( \Sigma(s) \) is not a graph in the direction of $\ee_1$. Using compactness of \( \Sigma(s) \), translate it in the \( \mathbf{e}_1 \)-direction until the last time it makes contact with its $\ee_1$ translates. 
This last contact cannot occur along the boundary, since the boundary is a graph. 
Moreover, as discussed earlier, it cannot occur between an interior point and a boundary point, 
because their \( \mathbf{e}_1 \)-projections are disjoint. 
Hence, the last point of contact must occur between interior points. 
By the maximum principle, this implies that the two surfaces must coincide locally. 
Consequently, \( \Sigma(s) \) would coincide with a nontrivial \( \mathbf{e}_1 \)-translation of itself, 
which is impossible, since \( \Gamma(s) \) is an \( \mathbf{e}_1 \)-graph. 
This contradiction shows that \( \Sigma(s) \) must be the graph  $u_s$ as in the statement of the Claim. 

Finally, we want to prove that $u_s$ is smooth at any interior point of $\Omega(s).$ Notice that this is equivalent to assert that $\langle\nu(p),\ee_1 \rangle \neq 0$, for any $p$ in the interior of $\Sigma(s).$ Again we proceed via proof by contradiction, assume that there were $p_0 \in \Sigma(s)\setminus \Gamma(s)$ such that $\langle\nu(p),\ee_1 \rangle = 0.$ Among other things, this implies that $T_{p_0} \Sigma(s)= \{x_1=k_0\},$ for some $k_0\in \R.$ Furthermore, $\Sigma(s)\cap \{x_1=k_0\}$ is an analytic curve, which is regular in a neighborhood of $p_0.$ 

If we label $\mathcal{S}(x_1,x_2,x_3)=(2 k_0-x_1,x_2,x_3)$ and denote:
\begin{eqnarray*}
    \Sigma^+(s) & = &\Sigma(s) \cap \{x_1 >k_0\}, \\
    \Sigma^-(s) & = &\Sigma(s) \cap \{x_1 <k_0\}, \\
    \Sigma^*(s) & = & \mathcal{S}(\Sigma^-(s)),
\end{eqnarray*}
then it is clear that (as $\Sigma(s)$ is a $\ee_1$-graph) $\Sigma^+(s)$ is at one side of $\Sigma^*(s)$ around the point $p_0$ (which is at the common boundary of both surfaces) and that they are tangent at this point. This is contrary to the maximum principle at the boundary. This contradiction proves the claim.
\end{proof}

Note that $\displaystyle\bigcup_{s}\Omega(s)=\Omega$, the whole lower half-plane.

Now, take any sequence $s_n\searrow -\infty$ and set $u_n=u_{s_n}.$ By compactness of the set of graphical solutions with common supremum bound defined over fixed disks (which is standard, seeing as we have in this case derivative estimates \eqref{sec-decay} which depend only on the supremum norm of the graph), we may assume $u_n\to u$,
where $u:\Omega\to\R$ is a solution of \eqref{TSE} with boundary data
\begin{equation} \label{eq:1234}
u|_{\partial \Omega} = c \quad \text{and} \quad \|u\|_\infty = \|c\|_\infty.
\end{equation}

Now that existence of the nonlinear Plateau solution has been established, we can address the oscillation problem. The basic proof idea will be to let $u(x_2,0)$ be given by longer and longer bumps which e.g. stay identically $\alpha$ and $\beta$, alternating on ever longer intervals, and then solve the nonlinear Plateau problem for Equation \eqref{TSE}, maintaining good enough control of all error terms. In the proof, we tackle the remaining technical difficulties to make this idea work for the translator equation \eqref{TSE} itself, using the detailed understanding of drift $L$-harmonic functions and the linear-with-drift and nonlinear term estimates proven elsewhere in this paper.

For ease of exposition, we will in the proof type out the counterexamples for $\alpha = 0$ and $\beta = 1$, with $\gamma\in (\tfrac{1}{2},\tfrac{3}{4})$, the general case being a straightforward modification.

We will make use of the uniformity of the constants in our improved gradient and Hessian estimates \eqref{grad_down}-\eqref{up_RHS}, which only depend on the supremum norm of the translator graphs, which again by use of constant barriers is controlled by the supremum norm of the boundary data. In particular, by Proposition \ref{P_decay_estimates}, we have for a fixed constant C that for all translating graphs $u$ over the lower half-plane $\Omega=\{(0,x_2,x_3):\: x_2\in\mathbb{R}, \: x_3\leq 0\}$ with $\|u\|_{L^\infty({\Omega})}\leq 1$, the uniform bounds, for all such $u$:
\begin{equation}\label{recap_RHS}
|P(\nabla u,\Hess u)|(p)\leq C \left[\mathrm{dist}(p,\partial \Omega)\right]^{-2}, \quad\textrm{when}\quad \mathrm{dist}(p,\partial \Omega)\geq 1.
\end{equation}

To exploit the nonlinear barrier argument in Proposition \ref{inhomog_down}, we first pick $b = b(\varepsilon_1) > 0$ large enough so that
\begin{equation}\label{ikke_lin_tolerance}
C' \, b^{-\frac{1}{2}} < \varepsilon_1,
\end{equation}
with $\varepsilon_1>0$ to be chosen small later, where $C'$ is $C$ from \eqref{recap_RHS} multiplied by the constant from the $\mathrm{Ei}$ barrier in the earlier proof, \eqref{Ei_barrier}. We then get, as before, that for any bounded translator Plateau solution with boundary values such that $\|u_{\partial \Omega}\|_\infty\leq 1$, where by using use of constant barriers, we ensure $0\leq u\leq 1$, we see that any corresponding solution $\bar{u}$ to $L\bar{u} = P(\nabla u,\Hess u)$ with boundary condition $\bar{u}(\cdot, -b) = 0$ satisfies $\|\bar{u}\|_{L^\infty(\{x_3 \leq -b\})}\leq \varepsilon_1$. We then consider $b = b(\varepsilon_1)$ fixed for the rest of the argument, which is the distance from the boundary at which we will ``start'' the $L$-harmonic solutions in the construction that follows.

Thus, we now need to arrange the boundary data for the Plateau solution $u$ in such a way that the $L$-harmonic solution on $$\Omega_b := \{(x_2,x_3) : x_2 \in \R,  x_3\leq -b\}$$ with boundary data $u(\cdot, -b)$ will oscillate between (nearly) 0 and 1 in value as $x_3\searrow -\infty$. This will be arranged by placing tilted grim reapers as nonlinear barriers, with larger and large tilts, recalling that we work with a fixed value of $b$.

\underline{Barrier details:} We also refer to Figure \ref{fig:GreatBarrierReef}. The tilted grim reaper cylinder of angle $\zeta\in(0,\pi/2)$ has as profile curve in the intersection with $\{x_3 = z_0\}$, the planar curve of the form:
\begin{equation}\label{ReaperProfil}
x_1 = -\frac{\log\cos\left(x_2\cos\zeta\right)}{\cos\zeta\sin\zeta} - z_0\cos\zeta, \quad x_2\in (-\tfrac{\pi}{2\cos\zeta}, \tfrac{\pi}{2\cos\zeta})
\end{equation}
The factor $(\sin\zeta)^{-1} > 1$ means that this is not exactly a homethetic scaling of the standard reaper curve, while of course the additional stretch factor $(\sin\zeta)^{-1} \searrow 1$, 
in the limit as $\zeta \nearrow \pi/2$. When $z_0 = 0$, this curve has its apex at $(x_1,x_2) = (0,0)$, while for $b>0$, the $\{x_3 = -b\}$ section curve is shifted up to have an apex at $(b \, \cos\zeta,0)$. To ensure that this is close enough to the axis, we will fix an $\varepsilon>0$ (to later be chosen small) and choose $\zeta=\zeta(\varepsilon_1,\varepsilon)$ close enough to $\pi/2$ so that $b\cos\zeta \leq \varepsilon/4$, so for definiteness we will take 
\begin{equation}\label{zeta}
\zeta=\zeta(\varepsilon_1,\varepsilon):= \arccos\left(\frac{\varepsilon}{4b(\varepsilon_1)}\right).
\end{equation}
Working only with one tilted grim reaper as a barrier would complicate the integral estimates needed below, due to the large scaling involved, and for this reason, we will work with long strings of $x_2$-translated copies of a fixed scale/angle, which we will refer to as ``barrier reefs'' (see Figure \ref{fig:GreatBarrierReef}). For this, we will take the periodicity constant $\tau := \tau(\varepsilon_1,\varepsilon)$ to be the $x_2$-width at $x_1$-height $\varepsilon/4$ of the convex curve in \eqref{ReaperProfil}, i.e. explicitly, with $\zeta(\varepsilon_1,\varepsilon)$ from \eqref{zeta}:
\begin{equation}\label{periode}
\tau(\varepsilon_1,\varepsilon) := \tfrac{2}{\cos\zeta} \arccos\left(e^{-\varepsilon\cos\zeta\sin\zeta}\right) >0.
\end{equation}

Then we may for $N \in \{1,2,\ldots\}$, a large number to be chosen later, define the $N$-fold barrier reef $w_{N} := w_{N,\varepsilon_1,\varepsilon}$, as the (merely continuous, piece-wise $C^1$) function:
\begin{equation}\label{GreatBarrierReef}
w_{N}(x_2,x_3) := \min_{k = 0,\ldots, N-1} w_\zeta(x_2 - k\tau,x_3),\: x_2\in \left(-\tfrac{\pi}{2\cos\zeta}, \tfrac{\pi}{2\cos\zeta}+(N-1)\tau\right),\: x_3\in\mathbb{R},
\end{equation}
where $w_\zeta$ is the tilted grim reaper cylinder graph defined over a strip, with the value of $w_\zeta$ defined as $+\infty$ whenever evaluated outside its domain of definition.

The nonlinear solutions which we will be constructing by the limits of exhaustions by finite Plateau problems will respect each one of the $N$ complete smooth barriers in \eqref{GreatBarrierReef} (see Figure \ref{fig:GreatBarrierReef}). Hence the solutions $u$ will also lie below the minimum of this whole collection, which is the function $w_{N}$. Notice from the definition of the barrier reef $w_N$, via and its arranged flatness $\tfrac{\varepsilon}{4}$ and its apex height in a slice also of size $\tfrac{\varepsilon}{4}$, that $w_N\leq \tfrac{\varepsilon}{4} + \tfrac{\varepsilon}{4} = \frac{\varepsilon}{2}$ in the interval between its left- and rightmost apex. In particular, this therefore means that $0\leq u(x_2, -b)\leq \varepsilon/2$ will hold for all $x_2$.

To start the construction of our boundary data $u(x_2, 0)$, we first claim that we may localize and glue data, in the sense of the following claim, where all data $g$ is tacitly assumed to lie between 0 and 1. Recall as in the proof of Proposition \ref{GlobalStab}, that the time parameter is here in reality to be understood as $t = -x_3$.\\

\begin{figure}[h!]\label{fig:GreatBarrierReef}
\centering
\begin{tikzpicture}[scale=1.1]
    \begin{axis}[
        width=\textwidth,
        height=5cm,
        axis y line=none,
        axis x line=middle,
        axis line style={black, -},
        xmin=-15, xmax=15,
        ymin=0, ymax=5,
        clip=false,
        xtick=\empty
    ]

    \draw[blue, ultra thick, cap=round]
        (axis cs:-15, 0) -- (axis cs:-8, 0)
        plot[smooth, tension=0.7] coordinates {
            (axis cs:-8, 0) (axis cs:-6, 0.1) (axis cs:-4, 0.5) (axis cs:-1.5, 0.9)
            (axis cs:0, 0.5)
            (axis cs:1.5, 0.6) (axis cs:4, 0.3) (axis cs:6, 0.05) (axis cs:8, 0)
        } -- (axis cs:15, 0);

    \pgfplotsinvokeforeach{-13, -11, -9, -7, 7, 9, 11, 13}
    {
        \addplot[
            purple, very thick, cap=round, loosely dotted,
            domain=-1.3:1.3, samples=41,
        ]
        ({1.2*x + #1}, {1.2*(-ln(cos(deg(x)))) + 1.2});
    }

    \pgfplotsinvokeforeach{-13, -11, -9, -7, 7, 9, 11, 13}
    {
        \addplot[
            red, ultra thick, cap=round, solid,
            domain=-1.3:1.3, samples=41,
        ]
        ({1.2*x + #1}, {1.2*(-ln(cos(deg(x)))) + 0.4});
    }

    \draw[black, very thick, loosely dotted] (14, 0) -- (14, 3.35);
    \draw[black, very thick] (14, -0.1) -- (14, 0.1);
    \draw[black, very thick, loosely dotted] (-14, 0) -- (-14, 3.35);
    \draw[black, very thick] (-14, -0.1) -- (-14, 0.1);
    \draw[black, very thick, loosely dotted] (2, 0) -- (2, 3.35);
    \draw[black, very thick] (2, -0.1) -- (2, 0.1);
    \draw[black, very thick, loosely dotted] (-2, 0) -- (-2, 3.35);
    \draw[black, very thick] (-2, -0.1) -- (-2, 0.1);

    \node[anchor=north, font=\large, yshift=-3pt] at (14, 0) {$\phantom{-}a_1$};
    \node[anchor=north, font=\large, yshift=-3pt, xshift=-3pt] at (-14, 0) {$-a_1$};
    \node[anchor=north, font=\large, yshift=-3pt] at (2, 0) {$\phantom{-}a_0$};
    \node[anchor=north, font=\large, yshift=-3pt, xshift=-4.5pt] at (-2, 0) {$-a_0$};
    \node[anchor=north, font=\large, yshift=-3pt] at (0, 0) {$0$};

    \end{axis}
\end{tikzpicture}
\caption{Illustration for the proof of the inequalities \eqref{claim:heat_approx}, showing initial data and sections of the tilted grim reaper cylinder barriers. Red indicates a translate of a ``barrier reef'' $w_N$ from \eqref{GreatBarrierReef} evaluated at height approximately $x_3 = 0$, while purple indicates $w_N$ at level $-b$, whose apex has there lifted by at most $\varepsilon/2$. Blue indicates the arbitrary initial data on $[-a_0,a_0]$ and its arbitrary $C^\infty$ extension to  $[-a_1,a_1]$ (with finite $C^2(\R)$-norm) chosen below the sufficiently long barrier reef $w_N$.}
\end{figure}
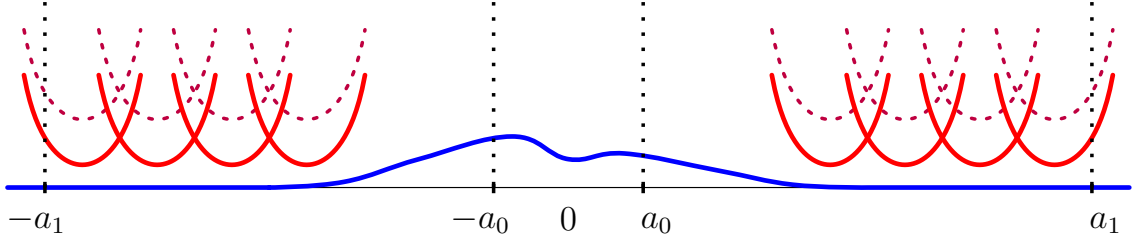

\begin{Claim}\label{Claim1}
For any $a_0 >0$ and $t_0>0$, regardless of the smooth data $g_0(\cdot, 0)|_{[-a_0,a_0]}$ there exists an $a_1 > a_0$ a time $t_1\geq t_0$ and smooth extensions $g_1^\pm(\cdot, 0)|_{[-a_1,a_1]}$ 
such that for all smooth extensions of these to $g^\pm_1(\cdot, 0)|_{[-\infty,\infty]}$ (as a technicality assumed to have finite $C^2(\R)$-norm, in order for Plateau solutions to exist), it holds that the heat equation solutions $h_1^\pm$, starting from data given by the sections $u^\pm(x_2,-b)$ of the nonlinear Plateau solution $u$ of the translator equation \eqref{TSE} with boundary data $u^\pm(\cdot, -b) = g_1^\pm(\cdot, -b)|_{[-\infty,\infty]}$ satisfies the estimates:

\begin{equation}\label{claim:heat_approx}
\begin{split}
h_1^-(0,t_1) &\leq 2\varepsilon,\\
h^+_1(0,t_1) &\geq 1 - 2\varepsilon.
\end{split}
\end{equation}
\end{Claim}
\vskip 3mm

\begin{proof}

Recall first that generally, a heat solution $h$ starting from a general $g$ has the form:
\begin{equation}
h(x,t) = \frac{1}{\sqrt{4\pi t}}\int_{\mathbb{R}} e^{-\frac{(x-x')^2}{4t}}g(x')dx'
\end{equation}

Letting $x = 0$ and after the change of variables $\xi=x'/\sqrt{4t}$, we recover the convolution with the standard heat kernel $\Phi$,
\begin{equation}\label{heat_convolv}
h(0,t) = (\Phi * g)(0,t) = \frac{1}{\sqrt{\pi}}\int_{\mathbb{R}} e^{-\xi^2}\, g(\sqrt{4t}\,\xi)\, d\xi,
\end{equation}
where $0\leq g\leq 1$ and below we will arrange that $g$ is below the barrier reef on the interval $[-a_1,a_1]\setminus [-a_0,a_0]$ (with $a_1$ to be chosen), meaning in particular of height less than $\varepsilon/2$.

The guiding idea is now simple: the heat kernel convolution at the origin is an average of the initial data $g$ weighted by the Gaussian $e^{-\xi^2}$. To show that this average can be made arbitrarily small (or symmetrically, for $h^+$, close to unity), we split the Gaussian integral into three regions. The \emph{central region} around the origin contributes very little because its length shrinks when $t$ is large. The \emph{tail region} contributes very little because the Gaussian decays. What remains is a bounded \emph{intermediate region}, which we can control by our choice of auxiliary cutoffs and the ``barrier reef construction'' \eqref{GreatBarrierReef}. With this strategy in mind, we now make these estimates precise while simultaneously completing the needed choices of parameters.

\smallskip
\noindent
\emph{Step 1. Tail estimate parameters.}

Since the Gaussian is $L^1$, given our chosen $\varepsilon>0$ we can also find a cutoff parameter $\chi_\varepsilon>0$ such that
\begin{equation}\label{xi_epsilon}
\frac{2}{\sqrt{\pi}}\int_{\chi_\varepsilon}^\infty e^{-\xi^2}\,d\xi \;\leq\; \frac{\varepsilon}{2}.
\end{equation}
This will in the below ensure that the far-field contribution to the convolution integral becomes negligible.

\smallskip
\noindent
\emph{Step 2. Control of the central region.}  
Next, we want the contribution from $[-\tfrac{a_0}{\sqrt{4t_1}},\, \tfrac{a_0}{\sqrt{4t_1}}]$ to be small, by arranging that this interval has short length. To arrange this, we pick $t_1\geq t_0$ so large that
\begin{equation}\label{a_0_frac}
\frac{a_0}{\sqrt{4t_1}} \;\leq\; \frac{\varepsilon}{2},
\end{equation}
where we will for definiteness fix:
\begin{equation}
t_1 := \max\!\left(t_0,\frac{a_0^2}{\varepsilon^2}\right).
\end{equation}

\smallskip
\noindent
\emph{Step 3. Intermediate cut-off.}

Next we introduce the size parameter of the domain of the extension of the data:
\begin{equation}
    a_1 := \chi_\varepsilon\sqrt{4t_1} = \max\left(\sqrt{4t_0},\tfrac{a_0}{\varepsilon}\right)\!\chi_\varepsilon,
\end{equation}
which defines the size of the intermediate interval where the Gaussian weight is not too small but the support is nevertheless bounded. 

\smallskip
\noindent
\emph{Step 4. Choice of the barrier reef function.}  
We now apply the construction in \eqref{GreatBarrierReef}, picking, with the period $\tau$ from \eqref{periode}, the number of repetition periods needed being
\begin{equation}\label{PickN}
N:=\Big\lceil\frac{a_1}{\tau}\Big\rceil + 1 =
\Big\lceil\frac{\max\left(\sqrt{4t_0},\tfrac{a_0}{\varepsilon}\right)\!\chi_\varepsilon}{\tau(\varepsilon_1,\varepsilon)}\Big\rceil + 1,
\end{equation}
which thus via \eqref{GreatBarrierReef} defines the ``barrier reef function'' $w_{N}$ adapted to these parameters.

We then add two suitably translated and reflected (under $x_2\mapsto -x_2$) copies of $w_N$, where by abuse of notation we still label the resulting barrier function by $w_N$, in such a way that $w_N(\pm a_0) = +\infty$. We then simply extend the given data 
$g_0(\cdot, 0)|_{[-a_0,a_0]}$ to $g_1^\pm(\cdot, 0)|_{[-a_1,a_1]}$ defined as any smooth extension that lies below $w_N$, e.g. by cutting off smoothly to zero, (see Figure \ref{fig:GreatBarrierReef}), and then allow for an arbitrary further smooth extension beyond $x_2\leq|a_1|$ to any $g^\pm_1(\cdot, 0)|_{[-\infty,\infty]}$ with finite $C^2(\R)$-norm. By construction, the maximum principle ensures that the Plateau solution $u$ stays below $w_N$, i.e. at level $x_3 = -b$ it has height at most $\varepsilon/2$.

\smallskip
\noindent
\emph{Step 5. Estimate of the heat convolution.}
With these choices in place, 
recall \eqref{heat_convolv}, evaluated at time $t_1$. We now split the integral representation of $h^-(0,t_1)$ into three regions:
\[
\int_{|\xi|\leq \frac{a_0}{\sqrt{4t_1}}} e^{-\xi^2}\, g(\sqrt{4t_1}\,\xi)\, d\xi +
\int_{\frac{a_0}{\sqrt{4t_1}}\leq |\xi|\leq \chi_\varepsilon} e^{-\xi^2}\, g(\sqrt{4t_1}\,\xi)\, d\xi + 
\int_{|\xi|\geq \chi_\varepsilon} e^{-\xi^2}\, g(\sqrt{4t_1}\,\xi)\, d\xi.
\]

The first region, by \eqref{a_0_frac}, is an interval of length at most $\varepsilon$, hence using $0\leq g\leq 1$, this contribution to \eqref{heat_convolv} is bounded by
$\frac{\varepsilon}{\sqrt{\pi}}$. On the second interval, with $g$ chosen below the barrier reef $w_N$ in \eqref{GreatBarrierReef} with $N$ as in \eqref{PickN}, recall that we have $|g(x)|\leq \varepsilon/2$ for $|x|\leq a_1$, which by using the definition of $a_1$ means the region $|\xi|\leq \chi_\varepsilon$, and thus this integral contributes at most $(\varepsilon/2)\int_{\mathbb{R}} e^{-\xi^2} d\xi$.  
On the third region, our choice of $\chi_\varepsilon$ in \eqref{xi_epsilon} guarantees the contribution is at most $\varepsilon/2$.

Putting these together, we finally obtain
\begin{equation}\label{h_plus_bound}
h^-(0,t_1) = 
(\Phi * g)(0,t_1) \;\leq\; \frac{1}{\sqrt{\pi}}\left[\varepsilon\int_{\mathbb{R}} e^{-\xi^2}\,d\xi + \frac{\varepsilon}{2}\int_{\mathbb{R}} e^{-\xi^2}\,d\xi\right] + \frac{\varepsilon}{2}
\;\leq\; 2\varepsilon.
\end{equation}
This proves the desired heat estimate.

\smallskip
\noindent
\emph{Step 6. Lower bound.}  
Exactly the same argument, just placing the ``barrier reef'' from below, i.e. the functions $1 - w_ N$, shows that one can likewise construct initial data leading to the flipped inequality $h^+(0,t_1) > 1-2\varepsilon$. This finishes the proof of Claim \ref{Claim1}.\end{proof}

Iterating now this construction from the claim alternately with the ``$\pm$'' choices, one may obtain a smooth initial profile $g(\cdot,0)|_{[-\infty,\infty]}$ with bounded first and second derivatives (so that the usual exhaustion using boundary Schauder estimates apply), and with the property that along one sequence of times $\{t^-_k\}\to\infty$, the heat evolution from the Plateau slice profiles at $x_3 = -b$
\begin{equation}\label{h_low}
\lim_{k\to\infty} h(0,t^-_k) \leq 2\varepsilon,
\end{equation}
while along another sequence $\{t^+_k\}\to\infty$,
\begin{equation}\label{h_high}
\lim_{k\to\infty} h(0,t^+_k) \geq 1-2\varepsilon.
\end{equation}

Since the error in Proposition \ref{GlobalStab}, which approximates the $L$-harmonic solution, which we label by $u_0$, by the heat solution $h$, tends to zero as $t\to\infty$, the same dichotomy as in \eqref{h_low}-\eqref{h_high} persists for the corresponding $L$-harmonic function. Thus, recalling that time $t = - x_3$:
\[
\limsup_{x_3\searrow -\infty} u_0(0,x_3) \;\geq\; 1-2\varepsilon 
\;>\; 2\varepsilon \;\geq\; \liminf_{x_3\searrow -\infty} u_0(0,x_3).
\]
This establishes a failure of convergence with a gap of at least $1-4\varepsilon$, for the $L$-harmonic function. Finally, tuning $\varepsilon$ sufficiently small relative to the nonlinear error tolerance $\varepsilon_1>0$ in \eqref{ikke_lin_tolerance}, we conclude as claimed that the same oscillatory behavior persists for the translating-graph quasilinear Plateau solution $u$ to \eqref{eq:1234} defined over the lower half-plane:
\[
\limsup_{x_3\searrow -\infty} u(0,x_3) \;>\; \liminf_{x_3\searrow -\infty} u(0,x_3).
\]

Tuning now the tolerances $\varepsilon_1 > 0$ and $\varepsilon > 0$ additionally shows that this gap can be made of size large enough to cover the interval $(\tfrac{1}{2},\tfrac{3}{4})$, and hence any value $\gamma$ in this interval occurs as the height for a limiting tangent plane of $u$, for some subsequence of $x_3$-values with $x_3\searrow -\infty$. The more general version of the theorem then also follows, for arbitrary intervals $[\alpha,\beta]$, by a straightforward modification of the argument.
\end{proof}

\section{Uniqueness of tangent planes at infinite times}\label{sec:Uniqueness}

We are now ready to prove the main theorem of this paper: Theorem \refThmAnospace. As we mentioned before, Gama, Mart\'{\i}n, and M{\o}ller  \cite{entropy} demonstrated that if our complete, embedded translator \( \Sigma \) possesses finite width, finite genus, and finite entropy, then the family of surfaces \( \Sigma \pm t \, \ee_3 \) converges, as \( t \) approaches infinity, to a finite set of (parallel) vertical planes. The convergence was there only known to happen subsequentially, meaning that along certain sequences \( t_k \to \infty \), the surfaces \( \Sigma \pm t_k \ee_3 \) converge to  planes, potentially with some multiplicity.

Building upon this foundation, the analysis carried out in the previous sections of the present work, concerning removable singularities theorems for the drift Laplacian $L$ as well as for the quasi-linear translating soliton equation \eqref{TSE}, can be employed to strengthen this result. Specifically, it can now be shown that the aforementioned limit is independent of the chosen sequence \( t_k \) that tends to infinity.  This independence of the limit on the sequences \( t_k \) is a strong indication of the robust geometric and analytic structure underlying the translator surface \( \Sigma \) and it has important applications, as we will see in the next sections.

\begin{thm}\label{thm-A}
Let $\Sigma \subset \R^3$ be a complete embedded translator with finite genus, finite entropy, and of width $2w$ such that
$\Sigma \subseteq\mathcal{S}_w$. Then, one has:
\begin{enumerate}
    \item There exists a partition $\{-w=u_1<\cdots<u_{k-1}<u_k=w\}$ of $[-w,w]$ and positive integer numbers $\sigma_1,\ldots,\sigma_k$ so that \[\Sigma+t \, \ee_3\to\sigma_1\{x_1=u_1\} \sqcup \ldots \sqcup \sigma_k\{x_1=u_k\}\ {\rm as}\  t\to-\infty.\]Furthermore, one has \[\lambda(\Sigma)=\sum_{j=1}^k\sigma_j.\] 
    \item There exists a finite set $\{d_1<\cdots<d_{l}\}$ of $[-w,w]$ and positive integer numbers $\delta_1,\ldots,\delta_l$ so that    \[\Sigma+t \, \ee_3\to \delta_1\{x_1=d_1\} \sqcup \ldots \sqcup \delta_l\{x_1=d_l\}\ {\rm as}\ t\to+\infty .\]  Furthermore, $\displaystyle2 \, \sum_{j=1}^l\delta_j$ coincides with the number of planar wings of $\Sigma.$ 
 \end{enumerate}
In both cases, the convergence is smooth on compact subsets of $\R^3$.

In particular, uniqueness holds in the subsequential convergence to tangent planes at infinite times $t\to\pm\infty$.

\end{thm}
\begin{proof}
   Let $\{\xi_i:\Wedge(p_i,\theta)\to\R\}_{i=1}^{\lambda(\Sigma)}$ be the functions in Theorem \ref{structure-thm} so that each upward limit is given by the limit of these functions, where $p_i$ is the sequence of vertex associate graph, and $\{\kappa_j:H\to\R\}_{j=1}^{k}$, where $H$ is the domain given by Theorem \ref{structure-thm}, be the functions on Theorem \ref{structure-thm} so that each downward limit is given by the limit of these functions. By Theorem \ref{thm:exponential}, and Proposition \ref{inhomog_down}, we obtain the existence of the limits:
   \[
   \lim_{x_3\to+\infty} \xi_i(x_2,x_3)=s_i\ {\rm and}\ \lim_{x_3\to-\infty} \kappa_j(x_2,x_3)=b_j.
   \]
   By \cite[Theorem 10]{Chini}, one has that $-w,w\in\{s_1, \ldots s_{\lambda(\Sigma)} \}$, so we can organize the set $\{s_i\}_{i=1}^{\lambda(\Sigma)}$ in such a way that it forms a partition $\{-w=u_1<\ldots u_{k-1}<u_k=w\}$. Set $\sigma_i$ to be the number of functions $\xi_j$ so that $\displaystyle\lim_{x_3\to+\infty} \xi_j(x_2,x_3)=u_i$. Notice that $\sigma_i$ counts the multiplicity in which the plane $\{x_1=u_i\}$ appears in the upward limit $\displaystyle \lim_{t\to+\infty}(\Sigma-t\ee_3).$ In particular, it must hold $\sum_{i=1}^k\sigma_i=\lambda(\Sigma)$ and
   \[
   \Sigma+t \, \ee_3\to\sigma_1\{x_1=u_1\} \sqcup \ldots \sqcup \sigma_k\{x_1=u_k\}\ {\rm as}\  t\to-\infty.
   \]
About the downward case, we only know that $\{b_1, \ldots, b_k\}\subseteq[-w,w]$ (the set could be empty if our translator only has grim reaper type wings). As before we order the set $\{b_j\}$ in increasing order to get a finite set $\{d_1<\ldots<d_l\}$ and setting $\delta_j$ to be the number of functions $\kappa_i$ so that $\displaystyle\lim_{x_3\to-\infty} \kappa_i(x_2,x_3)=b_j$ to finally conclude 
   \[\Sigma+t \, \ee_3\to \delta_1\{x_1=d_1\} \sqcup \ldots \sqcup \delta_l\{x_1=d_l\}\ {\rm as}\ t\to+\infty \]  
   and $\displaystyle2 \, \sum_{j=1}^l\delta_j=2k$ coincides with the numbers of planar wings of $\Sigma.$   
\end{proof}

\section{Asymptotic behaviors along wings of translators}\label{sec:asym-wings}
In this section, we will study the asymptotic behavior of the translator along a wing.

\subsection{Asymptotic behaviors along wings of planar type} \label{subsec:planar}

We consider $\Sigma$ a collapsed translator with finite topology, finite entropy and finite width and $W$ a planar wing of  $\Sigma.$ Up to a rigid motion, we assume that 
\begin{equation*}
    W= \{ (u(x_2,x_3), x_2,x_3) \; : \; x_2\geq t\} ,
\end{equation*}
where $u$ is a solution of \eqref{TSE}. From Theorem \ref{thm-A}, we have that there exist $a_{\pm}\in\R$ such that
\begin{eqnarray*}
    \lim_{x_3 \to +\infty} u(x_2,x_3) &= & a_+, \\
    \lim_{x_3 \to -\infty} u(x_2,x_3) & = & a_-.
\end{eqnarray*}
Again, after possibly rotating, we can assume that $$ a_-\leq a_+.$$
Under these circumstances, we have proven in Proposition \ref{sideways_limits} (more generally, Proposition \ref{cor:remov_sing}) that for any $x_3$, the following holds:
    \begin{equation}\label{eq:planar_sideways}
    \lim_{x_2\to+\infty}u(x_2,x_3)=a_+.
    \end{equation}

As mentioned in Remark \ref{rem:down_discontinuity}, the straight down direction may exhibit a jump discontinuity in the values of limits along rays, as elucidated more precisely by the following proposition, which shows that a continuum of tangent planes may in that situation occur as subsequential limits along sequences escaping to infinity:

\begin{prop} \label{prop:planar}
    Under the above hypotheses, for any $\kappa \in [a_-,a_+],$ there exists a sequence $p_n=(x_1^n,x_2^n,x_3^n) \in W$ such that:
    \begin{enumerate}
        \item $x_2^n \nearrow +\infty,$
        \item $x^n_3\to-\infty,$
        \item $W-(0,x_2^n,x_3^n)$ converges smoothly on compact subsets of $\RR^3$ to the vertical plane $\{x_1= \kappa\}.$
    \end{enumerate}
\end{prop}
\begin{proof}
From our hypotheses, we have that \begin{equation}\label{eq:chr}
  \overline{P_{x_1, x_2} (W)} \supseteq [a_-,a_+]\times \RR,  
\end{equation} 
where $P_{x_1,x_2}$ denotes the orthogonal projection over $(x_1,x_2)-$plane. 

    Take $\kappa\in [a_-,a_+].$ Then, \eqref{eq:chr} implies that it is possible to find sequences $x_1^n \in(a_-,a_+) $, $x_1^n \to \kappa$, $x_2^n \nearrow +\infty$ and such that there exists $x_3^n$ satisfying $$x_1^n=u(x_2^n,x_3^n).$$
So, $p_n=(x_1^n,x_2^n,x_3^n) \in W$. Besides, Proposition \ref{cor:remov_sing} implies $x_3^n\to-\infty$. Moreover,  $W-(0,x_2^n,x_3^n)$ is a sequence of side-wise (bounded) graphs. In particular, there is a subsequence that converges to a complete translator $W_\infty$ without boundary and entropy $1$. Consequently, $W_\infty$ must be a vertical plane. Since $W_\infty$ is contained in the slab $a_-\leq x_1\leq a_+$, and $x_1^n \to \kappa$, we conclude that $W_\infty=\{x_1=\kappa\}.$
\end{proof}

\begin{rem}
    Notice that, by Proposition \ref{cor:remov_sing}, we must have $x_3^n/x^n_2\to -\infty,$ as $n\to +\infty.$ In particular, $x_3^n$ must go to $-\infty$ faster than any linear function of $x_2^n,$ as $n\to+\infty.$
\end{rem}

\subsection{Asymptotic behavior along wings of grim reaper type}

Consider, as before, $\Sigma$ a complete, collapsed translating soliton of finite entropy and finite genus. 
In \cite[Proposition 7.2]{entropy}, the authors proved that for a grim reaper type wing, $W \subseteq \Sigma$, we have that for any sequence $$p_i\in \mathcal{M}_W \quad \mbox{so that} \quad |x_2(p_i)| \to \infty,$$ then $W-p_i$ smoothly converges, subsequentially, to a complete translating graph, i.e. to either a tilted grim reaper or a $\Delta$-wing. We introduced the term ``grim reaper type'' wing, because we suspected that the above asymptotic limit was always a grim reaper. 

The main goal of this section is to prove that we were using the correct term. In other words, we are going to prove that the above limit does not depend on the subsequence that we consider and that the limit is a tilted grim reaper determined by the limit planes of $W$, as $x_3\to \infty$. 

Up to a rigid motion in $\R^3$, we can assume that
 \begin{eqnarray}\label{normalization-wing}
 &W\subseteq\{x_2\geq 0\},\ \partial W\subseteq\{x_2= 0\},\ W-t\ee_3\to\{x_1=\pm b,x_2\geq 0\}&\\ 
 \nonumber&{\rm and}\quad \inf\{x_3\;:\;(x_1,0,x_3)\in\partial W\}=0.&
 \end{eqnarray}
   Under these assumptions, the curve $\mathcal{M}_W$ (see Section \ref{sec:reaperwings}) will be parametrized as 
   \begin{equation}\label{minimal-para}
       \gamma(s):=(x_1(s), s, f_W(s))
   \end{equation}  where the function $s\in[0,+\infty)\mapsto f_W(s)$ satisfies $f_W(0)=0.$ 

The main theorem of this part can be stated as follows:

\begin{thm}
    \label{prop:Asymp-GR}
Let $s(b)\geq 0$ be the slope of the tilted grim reaper cylinder of width $2b.$ Then for $\Sigma$ as above, either \[\lim_{s\to+\infty}f'_W(s)=s(b)\ {\rm or}\ \lim_{s\to+\infty}f'_W(s)=-s(b).\]    
\end{thm}
\begin{proof}
From \cite[Theorem A. 4]{HMW-A} there exists $s_b\geq0$ so that
\begin{equation}\label{control-wing}
    f_W(s)\geq f_W(s_0)-s_b |s-s_0|,
\end{equation}
for all $s,s_0\in[0,+\infty).$ In particular, for $s_0=0$, using \eqref{normalization-wing}, one has
\begin{equation*}\label{control-wing-0}
    f_W(s)\geq -s_b s.
\end{equation*}
Furthermore, notice that \eqref{control-wing} also implies that for all $R>0$
there exists $s_0$ so that 
\begin{equation}\label{initial-control}
(s,f_W(s))\in \{(x_2,x_3)\;:\;x_3\geq R-s_bx_2\},\ \forall s\geq s_0.    
\end{equation}

By Theorem \ref{structure-thm}, we can fix two points $p_1$ and $p_2$ in $\partial W,$ with $x_3(p_1)=x_3(p_2)=R>0$ and choose $\alpha$ sufficiently small so that if $\Wedge(p_i,\alpha)=\{(0,x_2,x_3+R)\;:\;x_3\geq\tan\alpha |x_2|\}$ there exist two smooth functions $\xi_i:\Wedge(p_i,\alpha)\to\R$ so that
\[p_i\in \Graph(\xi_i),\] where $\Wedge(p,\theta)$ denotes the wedge in Definition \ref{eq:def-wedge}. We say that $p_i$ is the vertex of $\Graph(\xi_i).$ Observe $\Wedge(p_1,\alpha)=\Wedge(p_2,\alpha)$ 

Now, from \cite[Theorem B.7]{HMW-A} we know that \[-s(b)\leq\liminf_{s\to+\infty}f_W'(s)\leq\limsup_{s\to+\infty}f_W'(s)\leq s(b).\] 
Using this, we see that for all  $\theta>0$ with $\cot\theta>\max\{s_b, s(b),\cot\alpha\}$, we can find $s(\theta)\geq s_0$ so that $| f_W'(s)|<\cot\theta,$ when $s\geq s(\theta).$ Fix one of those $\theta>0$ from now on. From this choice, we can infer from Section \ref{sector} that for all $q\in W_1=\{q\in W\;:\; {\rm dist}(q,\mathcal{M}_W)\geq 1,\ x(p)\geq s(\theta)\},$ it holds true
 \[\Wedge(p,\theta)\subseteq \Wedge(p_1,\alpha)\cup \Omega_1,\] where $\Omega_{1}=P_{x_2,x_3}(W_1),$ and $\Omega=P_{x_2,x_3}(W),$  where $P_{x_2,x_3}$ denotes the orthogonal projection over $(x_2,x_3)-$plane.

At this moment, we use  \eqref{control-wing} and \eqref{initial-control} to infer that for all $p\in W_1$ with $x_2(p)\geq s(\theta),$ we can find a unique point $q(p)\in\{(x_1,x_2,x_3)\in\partial W_1\cap\{x_2\geq s(\theta)\}\}$ with $x_2(q(p))=x_2(p)$ so that $p$ belongs to the graph of $\xi_{q(p)}:\Wedge(q(p),\theta)\to\R.$ Regarding the function $\xi_{q(p)}$, we further know from Theorem \ref{P_decay_estimates} and Proposition \ref{thm:exponential} that there exists a universal constant $C>0$ depending on $\theta$ and $b$ such that 

\begin{equation}\label{osc-inequ}
|\xi_{q(p)}(x_2(p),x_3)-\xi_{q(p)}(x_2(p),\hat x_3)|\leq C (e^{-\beta x_3}+e^{-\beta\hat{x}_3}),\ \forall x_3\geq\hat{x}_3\geq x_3(q(p)),   
\end{equation}
for all $p\in W_1.$ 

Assume that there was a sequence $s_n\nearrow+\infty$ so that $f_W'(s_n)\to s$ with $|s|<s(b).$ By passing to a subsequence, we can assume W.L.O.G that $W_\infty=\displaystyle\lim_{n} (W-s_n\ee_2-f_W(s_n)\ee_3).$ Since $W-t\ee_3\to\{x_1=\pm b,x_2\geq0\}$ as $t\to+\infty,$ we see that we could find a positive constant $K$ and sequences $x_3^n> \hat{x}_3^n\to+\infty$ as $n\to+\infty$ so that
\begin{equation}\label{estimate-111}
|\xi_{n}(s_n,x_3^n)-\xi_n(s_n,\hat{x}_3^n)|\geq K,\ \forall n\ {\rm large\ enough}
\end{equation}
where $\xi_n$ is $\xi_{q_n},$ where $q_n$ is the unique point in $\partial W_1\cap\{x_2=s_n\}$.

To see this, we note that there are two possibilities for $\xi_n.$ We select the one for which the plane $\{x_1=\displaystyle\lim_{x_3\to+\infty} \xi_n(s_n,x_3)\}$ does not lie in the boundary of the minimal slab of $W_\infty$, i.e. the smallest vertical slab that contains $W_\infty$. With this choice, the constant $K$ can be taken as half of the distance from $\{x_1=\displaystyle\lim_{x_3\to+\infty} \xi_n(s_n,x_3)\}$ to the boundary of the minimal slab of $W_\infty$.

Finally, \eqref{osc-inequ} and \eqref{estimate-111} give us the desired contradiction, since \[\lim_{n}|\xi_{n}(s_n,x_3^n)-\xi_n(s_n,\hat{x}_3^n)|=0.\] 

Therefore, we have proven that for any sequence $s_n\to+\infty$ so that  $\displaystyle\lim_n  f_W'(s)$ exists, one must hold that either \[\displaystyle\lim_n f_W'(s)=s(b)\ {\rm or}\ \displaystyle\lim_n  f_W'(s)=-s(b).\] From this, we can easily prove that it must hold true that either \[\displaystyle\lim_{s\to+\infty} f_W'(s)=s(b)\ {\rm or} \displaystyle\lim_{s\to+\infty}  f_W'(s)=-s(b).\] This completes the proof.
\end{proof}
The next corollary shows that controlling the slope of $f_W$ at $\infty$ is equivalent to control the asymptotic limit of our wing. Recall that $\gamma$ in \eqref{minimal-para} parametrizes the curve of minima.

\begin{cor}        \label{eq:corB}
Let $W$ be a grim reaper type wing satisfying the normalization described in \eqref{normalization-wing}. Then $W-\gamma(s)$ converges smoothly on compact sets, as $s \to \infty$, to a unique tilted grim reaper of width $2 b$.
\end{cor}
\begin{proof}
    We know that $W-\gamma(s)$ converges, subsequentially (by \cite[Proposition 7.2]{entropy}) to a graph contained in a slab of width at most $2b$ (vertical planes parallel to the slab being ruled out by the direction of the normals along $\gamma(s)$). 
    Let's call $u(x_1,x_2)$ the function defining this graph.
 
    Moreover, Theorem \ref{prop:Asymp-GR}, implies that
    \begin{equation}
        \label{eq:partialu}
        \partial_{2} u(0,0)=  s(b) \: \:{\rm or } \: -s(b).
    \end{equation}
    The only graph over a strip of width $\leq 2 b$ satisfying \eqref{eq:partialu} is the tilted grim reaper of slope $s(b)$ or $-s(b)$. In particular, the slab coincides with $\{|x_1|\leq b\}.$
\end{proof}

The next two technical consequences will be needed in the next section.
\begin{prop}\label{conseq.}
 Let $W$ be a grim reaper type wing satisfying the normalization described in \eqref{normalization-wing}. Consider $\{q_n=(x_1^n,x_2^n,x_3^n)\}$ any divergent sequence on $W$ so that $x_2^n\to+\infty$ and $\displaystyle\lim_{n\to+\infty}{\rm dist}\{q_n,\mathcal{M}_W\}=+\infty,$ then \[\Sigma-x_2^n\ee_2-x_3^n\ee_3 \to \{x_1=-b\}\cup\{x_1=b\}.\]
\end{prop}
\begin{proof}
   Let $\Sigma_\infty$ be the limit of any convergent subsequence of $\{\Sigma-x_2^n\ee_2-x_3^n\ee_3\}.$ Our hypothesis yields that it only has planar type wings. Furthermore, using \eqref{normalization-wing} and Theorem \ref{thm-A} (see also \cite[Theorem 8.2]{entropy}), we conclude that $\lambda(\Sigma_\infty)=2.$ Hence, $\Sigma_\infty=\{x_1=t_1\}\cup\{x_1=t_2\},$ where $t_1<t_2.$ 
   
   To prove that we do in fact have $t_1=-b$ and $t_2=b,$ we argue by contradiction. Let us assume $t_1\neq -b$. To achieve the desired contradiction, note that \eqref{osc-inequ} applied to $(x_2^n,x_3^n)$ and any point $(x_2^n,\tilde{x}^n_3)$, so that
   \[
   x_3^n<<\tilde{x}_3^n\quad {\rm and}\quad |\xi_n(x_2^n,\tilde{x}^n_3)-\lim_{y\to+\infty}\xi_n(x_2^n,x_3^n)|\leq \frac{1}{n}
   \]
   gives $\displaystyle\lim_{n\to+\infty}|\xi_n(x_2^n,x_3^n) -\xi_n(x_2^n,\tilde{x}^n_3)|=0.$ 
   
   On the other hand, taking $2K=\min\{|b-t_1|,|b+t_1|\}>0$ one would see that $\displaystyle\liminf_{n\to+\infty}|\xi_n(x_2^n,x_3^n) -\xi_n(x_2^n,\tilde{x}^n_3)|\geq K$, as in \eqref{estimate-111}. This contradiction proves $t_1=-b$. Similarly, we can prove $t_2=b,$ and thus $\Sigma_\infty=\{x_1=-b\}\cup\{x_1=b\}.$ 
 \end{proof}

The next proposition quantifies the eventual almost symmetry of the curve of minima along a wing of grim reaper type.

\begin{prop}\label{conseq.1}
    Let $W$ be a grim reaper type wing satisfying the normalization described in \eqref{normalization-wing}. Then for all $\epsilon>0$, there exists $s_\epsilon>0$ so that 
    \[\min \left\{{\rm dist}\left(\{x_1=-b\},\mathcal{M}_W\cap\{x_2\geq s_\epsilon\}\right),{\rm dist}\left(\{x_1=b\},\mathcal{M}_W\cap\{x_2\geq s_\epsilon\}\right)\right\}\geq b-\epsilon.\] 
    \end{prop}
    \begin{proof}
        Indeed, given $\epsilon>0.$ By Theorem \ref{prop:Asymp-GR}, we know
        \[
\lim_{s\to+\infty}(W-s \,\ee_2-f_W(s)\ee_3)=\mathcal{G}(b),
        \]
        where $\mathcal{G}(b)$ is the tilted grim reaper of slope $\displaystyle\lim_{s\to+\infty} f'_W(s).$ Using that the curve of minima of $\mathcal{G}(b)$ is on the plane $\{x_1=0\},$ we see that there exist $s_\epsilon>0$ so that 
        \[\min\{{\rm dist}\{\{x_1=-b\},\mathcal{M}_W\cap\{x_2\geq s_\epsilon\}\},{\rm dist}\{\{x_1=b\},\mathcal{M}_W\cap\{x_2\geq s_\epsilon\}\}\}\geq b-\epsilon.\] 
        This concludes the proof of the proposition.
    \end{proof}

We would like to end this section with the following control of the growth of the function $f_W.$

\begin{prop}\label{growth-function}  Let $W$ be a grim reaper type wing satisfying the normalization described in \eqref{normalization-wing}. Then, one has:
    \begin{itemize}
        \item[] If $\displaystyle c>\lim_{s\to+\infty}f_W'(s),$ then for all $d>0$ there exists $s(d)$ such that
    \[
    cs+d>f_W(s), \forall s\geq s(d).
    \]
    \item[] If $\displaystyle c<\lim_{s\to+\infty}f_W'(s),$ then for all $d>0$ there exists $s(d)$ such that
    \[
    cs+d<f_W(s), \forall s\geq s(d).
    \]
    \end{itemize}
\end{prop}
\begin{proof}
    Let us first assume $\displaystyle c>\lim_{s\to+\infty}f_W'(s).$ In a similar way we can prove the second case. Take $c'>0$ so that $\displaystyle c>c'+\lim_{s\to+\infty}f_W'(s).$ Define $g(s)=cs+d-f_W(s).$ Our hypothesis ensures that we can find $g'(s)=c-f'_W(s)=c-c'-f'_W(s)+c'>c',$ for all $s\geq s_0,$ for some $s_0$ large enough. Hence, $g(s)>c' s+g(s_0)$ for all $s\geq s_0$ which implies that for some $s(d)$ we must have $g(s)>0$, whenever $s\geq s(d).$ This completes the proof.
\end{proof}

\subsection{Proofs of Theorems B, C and D}

At this stage, we have developed all the necessary tools needed in the proofs of the (non-)removable singularities and structure theorems.

\begin{proof}[Proof of Theorem B]
Recall the normalization $\mathbb{S}^1=\{(0,x_2,x_3) \in \R^3 \; : \;  x_2^2+x_3^2=1\}$, and denote $\mathbb{S}^1_+=\{(0,x_2,x_3) \in \mathbb{S}^1 \; : \;  x_2>0\}$ and $\mathbb{S}^1_-=\{(0,x_2,x_3) \in \mathbb{S}^1 \; : \;  x_2<0\}.$ For given $\oomega \in \mathbb{S}^1_{\pm}$ denote by $\mathcal{G}(\oomega,t)$ the unique (possibly tilted) grim reaper passing through $(t,0,0)$ and so that $\oomega$ is tangent to $\mathcal{G}(\oomega,t)$ and $\mathcal{G}(\oomega,t)$ is symmetric across $\{x_1=t\}.$ Notice that $\mathcal{G}(\oomega,t)=\mathcal{G}(-\oomega,t).$

Let $\Sigma \subset \R^3$ be a properly embedded translator with finite entropy, finite genus and contained in the slab
$$ \mathcal{S}_w= \{(x_1,x_2,x_3) \in \R^3 \: \colon \: |x_1| < w\}.$$

We define the set $\mathscr{R}(\Sigma)$ as follows: firstly define
\[
 \mathscr{R}_+(\Sigma):=\begin{Bmatrix}
\oomega\in\mathbb{S}^1_+ & :\;\exists\ \text{right\ wing}\ W\subseteq \Sigma\;:\;\displaystyle\lim_{y\to+\infty}(W-y\ee_2-f_W(y)\ee_3)=\mathcal{G}(\oomega,t_W)
\end{Bmatrix}
\]
and
\[
 \mathscr{R}_-(\Sigma):=\begin{Bmatrix}
\oomega\in\mathbb{S}^1_- & :\;\exists\ \text{left\ wing}\ W\subseteq \Sigma\;:\;\displaystyle\lim_{y\to-\infty}(W-y\ee_2-f_W(y)\ee_3)=\mathcal{G}(\oomega,t_W)
\end{Bmatrix}
\]
then we set $\mathscr{R}(\Sigma)=\mathscr{R}_+(\Sigma)\cup \mathscr{R}_-(\Sigma)$, here $t_W\in(-w,w)$ depends only on $W$ by Corollary \ref{eq:corB}.

Regarding the set $\mathscr{R}(\Sigma)$, since $\Sigma$ has finite wings by \cite[Theorem 1.4]{entropy}, we conclude that $\Bbbk:=\sharp \mathscr{R}(\Sigma)<\infty.$ Indeed, one has that $$\Bbbk\leq \omega_G(\Sigma),$$ where $\omega_G(\Sigma)$ denotes the number of grim reaper type wings of $\Sigma.$ 

If we parametrize $$ \mathbb{S}^1= \{(0,\cos \varrho, \sin \varrho) \; : \; \; \varrho \in [-\tfrac \pi2, \tfrac{3 \pi}2 ) \}. $$
Then, we can write 
\[\mathscr{R}(\Sigma)= 
\left\{\begin{array}{c|c}
    \oomega_j=(0,\cos \varrho_j, \sin \varrho_j)       & \begin{matrix}
j=1, \ldots , \Bbbk\\
-\tfrac \pi2 < \varrho_1<\varrho_2< \cdots < \varrho_i<\tfrac \pi2<\varrho_{i+1} < \cdots <\varrho_\Bbbk < \tfrac{3\pi}2
\end{matrix}    
      \end{array}\right\}
\]

    Notice that with this notation, one has: $$\mathscr{R}_+(\Sigma)=\{ \oomega_1, \ldots , \oomega_i\} \quad \mbox{and} \quad  \mathscr{R}_-(\Sigma)=\{ \oomega_{i+1}, \ldots , \oomega_\Bbbk\}.$$
If $\oomega\in\mathbb{S}^1\setminus \mathscr{R}(\Sigma)$ we set $\digamma_{\oomega}\equiv 0.$ 
\vskip 3mm

\noindent If $\oomega_j \in \mathscr{R}(\Sigma)$, we pick one wing $W$ such that $$\displaystyle\lim_{x_2\to+\infty}(W-x_2 \, \ee_2-f_W(x_2)\ee_3)=\mathcal{G}(\oomega_j,t_W).$$ Let $s_j$ denote the slope of $\mathcal{G}(\oomega_j,t_W).$ Then one has
\begin{eqnarray*}
    x_2 \, \ee_2+f_W(x_2)\ee_3&=& x_2 \, (\ee_2+s_j\ee_3)+(f_W(x_2)-s_j\, x_2)\ee_3 \\ &=& x_2 \cos(\varrho_j)\oomega_j+(f_W(x_2)-s_j \,x_2)\ee_3.
    \end{eqnarray*} 
    Hence, if we set $r=x_2 \, \sec\varrho_j$ and $\digamma_{\oomega_j}(r)=-(f_W(x_2)-s_j x_2)$, one gets (using Theorem \ref{prop:Asymp-GR}) that: 
\[\lim_{r\to+\infty}\digamma_{\oomega_j}'(r)=0\ \text{and}\ \mathcal{G}(\oomega_j,t_W)\subseteq\lim_{r\mapsto+\infty}(\Sigma-r\oomega+\digamma_{\oomega_j}(r)\ee_3)=:\Sigma_\infty(\oomega),\]

\noindent Our choice of the function $\digamma_{\oomega_j}$, together with Theorem \ref{thm-A}, Corollary \ref{eq:corB}, and Propositions \ref{sideways_limits}, \ref{cor:remov_sing}, and \ref{conseq.} imply that \(\Sigma_\infty(\oomega_j)\) is a (possibly empty) finite multiset of planes and (possibly tilted) grim reapers tangent to $\oomega_j,$ for all $j\in \{1,\ldots,\Bbbk\}.$ Notice that if $W'$ is another wing of grim reaper type, tangent to $\oomega_j$ at infinity, then the contribution of $W'$ to $\Sigma_\infty(\oomega_j)$ is:
\begin{itemize}
    \item A (possibly tilted) grim reaper, tangent to $\oomega_j$, if $|f_W -f_{W'}$ is bounded (Corollary \ref{eq:corB}),
    \item A couple of vertical planes, if $|f_W -f_{W'}$ is unbounded (Proposition \ref{conseq.}.)
\end{itemize}

Finally, given any $\oomega\in\mathbb{S}^1\setminus(\mathscr{R}(\Sigma)\:\cup\:\{-\ee_3\}),$ we can write it as $\oomega=(0,\cos \varrho, \sin \varrho),$ where $\varrho \in [-\pi/2, 3 \pi/2) \setminus \{\varrho_1, \ldots , \varrho_\Bbbk\}. $ The behavior of the set $\Sigma_\infty(\oomega)$ is given by:
\begin{itemize}
    \item For $\varrho \in (-\pi/2, \varrho_1)$, then $\Sigma_\infty(\omega)$ consists of  the set of planes $\mathscr{P}_+$ which contains $\omega_P(\Sigma)/2$ planes with multiplicity, corresponding to the upper limit of all the right planar wings of $\Sigma.$ This is a direct consequence of Propositions \ref{sideways_limits} and \ref{cor:remov_sing}. Notice that if $\Sigma$ does not have planar wings, then $\mathscr{P}_+$ is empty.

     \item By induction, we conclude that whenever $\varrho \in (\varrho_j,\varrho_{j+1})$, $j<i$, the limit set $\Sigma_\infty(\oomega)$ consists of the previous set of planes $\mathscr{P}_+$ plus a pair of planes for each wing of grim reaper type tangent at infinity to $\oomega_j$ (Proposition \ref{conseq.}.)

\item For $\varrho \in (\varrho_i,\varrho_{i+1})$ the limit set $\Sigma_\infty(\oomega)$ consists of the previous set of planes $\Sigma_\infty(\tilde{\oomega})$ associated to any $\tilde{\oomega}\in(\varrho_{i-1},\varrho_i)$ plus a pair of planes for each wing of grim reaper type tangent at infinity to $\oomega_i$ (Proposition \ref{conseq.}.) In this case, Theorem \ref{thm-A} says that this set contains exactly $\lambda(\Sigma)$ planes.

  \item By induction, we find that whenever $\varrho \in (\varrho_{i+j},\varrho_{i+j+1})$, $j<\Bbbk-i$, the limit set $\Sigma_\infty(\oomega)$ consists of the previous set of planes $\Sigma_\infty(\bar{\oomega})$ associate to any $\bar{\oomega}\in(\varrho_{i+j-1},\varrho_{i+j})$ minus a pair of planes for each wing of grim reaper type tangent at infinity to $\oomega_{i+j}$ (Proposition \ref{conseq.}.)

  \item Finally, for $\varrho \in (\varrho_\Bbbk,3\pi/2)$, consists of  the set of planes $\mathscr{P}_-$ which contains $\omega_P(\Sigma)/2$ planes with multiplicity,, corresponding to the upper limit of all the left planar wings of $\Sigma$ ( Propositions \ref{sideways_limits} and \ref{cor:remov_sing}.)
\end{itemize}

In summary, for $\oomega \in \mathbb{S}^1\setminus(\mathscr{R}\cup \{-\ee_3\})$, the behavior of the limit set $\Sigma_\infty(\oomega)$ follows a kind of monotonicity. We start with a collection of planes (possibly empty), $\mathscr{P}_+$ associated with the right planar-type wings. As we move counterclockwise along $\mathbb{S}^1$, the number of planes (which is locally constant) increases by an even amount each time we cross a direction asymptotic to a right grim–reaper-type wing. This number of planes reaches its maximum when we are near the vertical direction, $\ee_3$. This maximum coincides with the entropy of $\Sigma$. From that point on, the number of planes begins to decrease (again by an even amount) each time we cross a direction asymptotic to a left grim–reaper-type wing. Finally, when $\oomega$ is once again close to $-\ee_3$, we return to the family of planes $ \mathscr{P}_-$ associated with the left planar-type wings.
\end{proof}

\begin{proof}[Proof of Theorem C] As we have already explained in the proof of Theorem \refThmBnospace, if $\oomega\in\mathbb{S}^1\setminus\mathscr{R}(\Sigma)$, Theorem \ref{thm-A} and Propositions \ref{conseq.} and \ref{growth-function} ensure that each connected component of $\Sigma_\infty(\oomega)$  is a plane with certain multiplicity. 

On the other hand, when $\oomega\in\mathscr{R}(\Sigma)$, our choice of the function $\digamma_{\oomega}$ implies that $\Sigma_\infty(\oomega)$ contains at least one (possibly tilted) grim reaper tangent to $\oomega$ and (possibly empty) finite planes. Observe again that, in light of the results we are invoking, both the limit planes and the limit grim reapers may arise with a certain multiplicity, rather than appearing as single copies. This finishes the proof of Theorem C.
\end{proof}

Recall that Theorem D's second part has already been proven in Section \ref{sec:counterexamples}.
Before finalizing the proof Theorem~\refThmD (namely, the precise formulation of the first part of Theorem~\refThmDnospace), it is useful to introduce a geometric notion that will play an important role in our arguments. We shall refer to this notion as a flap. Informally, a flap can be thought of as something larger than a wing, in the sense of the collapsed translators described in~\cite{entropy}, but still smaller than an entire topological end. This makes it a natural bridging concept, positioned between the geometric detail of a wing and the global extent of an end, and well-suited for the structural analysis we require.

\begin{dfn}[Flaps]\label{flip-flop-flap}
Let us consider the sets $(-w,w) \times (-t,t) \times (-s, +\infty)$, where $t$ and $s$ are positive constants given by Proposition \ref{eq:decomposition-1}. We call a flap of $\Sigma$ any connected component of the difference $$\Sigma \setminus (-w,w) \times (-t,t) \times (-s, +\infty).$$
Notice that this is an intermediate notion, between that of a wing and a topological end of the surface $\Sigma$.

We say that a flap of $\Sigma$ is of pitchfork type, if it contains the graph of a smooth function $u:\Omega\to\R,$ where $\Omega=\left((-\infty,-t)\times \R \right) \cup  \left([-t,t] \times (-\infty,-s) \right)\cup  \left((t,\infty) \times \R\right)$ is an unbounded $U$-shaped domain in the $(x_2,x_3)$-plane, for some $s,t>0$, such that
\[
\forall\; x_3\in\R: \quad \displaystyle\lim_{x_2\to+\infty} u(x_2,x_3)\neq \displaystyle\lim_{x_2\to-\infty} u(x_2,x_3).
\]
\end{dfn}

We can therefore now conclude this section by stating and proving a more precise version of the first half of Theorem~\refThmDnospace, as follows.

\begin{thm}\label{Thm-D}
Let $\Sigma$ be a complete embedded translator with finite genus, finite entropy, and finite width. The multiset-valued map $\oomega\in\mathbb{S}^1\mapsto\Sigma_\infty(\oomega)$ is discontinuous at all points of $\mathscr{R}(\Sigma)$ with respect to Hausdorff distance.

Furthermore, there exists an open neighborhood $B=B(\Sigma)$ of $-\ee_3$ in $\mathbb{S}^1$ such that:
\begin{itemize}
    \item [(i)] either $\Sigma_\infty(-\ee_3)=\Sigma_\infty(\oomega)$ for all $\oomega\in B(\Sigma),$
    \item [(ii)] or $\Sigma_\infty(-\ee_3)\neq\Sigma_\infty(\oomega)$ for all $\oomega\in B(\Sigma)\setminus\{-\ee_3\}.$
\end{itemize}
In the first case, $\Sigma$ does not have flaps of pitchfork type. In the second case $\Sigma$ must have at least one flap of pitchfork type.
\end{thm}
\begin{proof}
We first observe that, as we have already explained in the proof of Theorem \refThmB above, the map $\oomega\mapsto\Sigma_{\infty}(\oomega)$ is discontinuous at all points of $\mathscr{R}(\Sigma).$ 

Regarding $\oomega_0=-\ee_3,$ we set $d=\text{dist}_{\mathbb{S}^1}(-\ee_3,\mathscr{R}(\Sigma))>0.$ If $\Sigma$ does not have any flap of pitchfork type, then Theorem \ref{thm-A}, Corollary \ref{cor:exterior} and Propositions \ref{sideways_limits}, \ref{inhomog_down}, \ref{prop:planar}, and \ref{growth-function} imply 
\[
\Sigma_\infty(\oomega)=\Sigma_\infty(-\ee_3), \; \; \forall \oomega\in B:=B_{\mathbb{S}^1}(-\ee_3,d).
\]

On the other hand, in light of the notation of the proof of Theorem \refThmBnospace, if $\Sigma$ has at least one flap of pitchfork type, then \cite[Corollaries 8.3 and 8.5]{entropy}, Proposition \ref{sideways_limits}, and Theorem \ref{thm-A} give
\[
\Sigma_\infty(\oomega)=\begin{cases}
\mathscr{P}_+, & \text{if}\ \oomega\in\mathbb{S}_+^1\cap B\\
\mathscr{P}_-, & \text{if}\ \oomega\in\mathbb{S}_-^1\cap B
\end{cases}\]
and
\[
\Sigma_\infty(-\ee_3)=\{x_1=d_1\} \cup \ldots \cup\{x_1=d_l\},
\]
for some $l\geq 1, $ where $d_1\leq d_2\leq \cdots \leq d_l$.  In turn, Proposition \ref{inhomog_down} implies $\Sigma_\infty(\oomega)\neq\Sigma_\infty(-\ee_3),\ \forall\;\oomega\in B=B_{\mathbb{S}^1}(-\ee_3,d)\setminus\{-\ee_3\}.$ Indeed, those results and the embeddedness of $\Sigma$ guarantee that there must exist two planes $\{x_1=d_+\}\in\mathscr{P}_+$ and $\{x_1=d_-\}\in\mathscr{P}_-$, with $d_+\neq d_-,$ with $\{x_1=d_\pm\}\notin\mathscr{P}_\mp,$ and $d_k$ such that $d_k=(d_++d_-)/2,$ and $\{x_1=d_k\}\notin \mathscr{P}_\pm,$ which implies the desired claim. This completes the proof of the theorem.
\end{proof}
\begin{rem}
    Note that when $\Sigma$ possesses more than one wing whose asymptotic (possibly tilted) grim reaper cylinder is tangent to the same $\oomega$, then the above does not a priori rule out that some other grim reaper type wings are lost in the limit as we follow asymptotically one chosen wing. This could occurs because we lack information about the asymptotic expansion of the function $f_W$ near infinity. Although Proposition \ref{growth-function} establishes the expansion is sub-quadratic, we cannot guarantee that the difference of $f_W$ of different wings behave at infinity in a similar way.
\end{rem}

\section{Collapsed translators of finite type}\label{sec:collapsed}

In \cite{HMW-A}, Hoffman, Martín and White introduced the concept of translators of finite type. In order to understand this concept, it is necessary to recall that Ilmanen showed \cite{ilmanen} $\Sigma$ is a translator in $\R^3$ if and only if it is a minimal surface with respect to the conformal metric 
\begin{equation}\label{IlmanenMetric}
g= {\rm e}^{x_3} \left( dx_1^2+d x_2^2+dx_3^2\right)
\end{equation}

Suppose that $\Sigma$ is a proper minimal surface  in $M = (\R^3,g)$.
Suppose that $F:\R^3\to\R$ is a continuous function  such that
 the level sets of $F$ are minimal surfaces, and 
 for each $t$, $\{F=t\}$ is in the closure of $\{F>t\}$ and of $\{F<t\}$.
Suppose also that $F$ is nonconstant on each connected component of $\Sigma$. 
Then $\mathsf{N}(F|_\Sigma)$ denotes the number interior critical points of $F|_\Sigma$ (counting multiplicity).

\noindent{}There are two such types of functions on $(\R^3,g)$ which are particularly interesting:

Firstly, if $\vv$ is a horizontal unit vector in $\RR^3$, then
 the function 
\begin{equation}\label{definition-F_v}
\begin{aligned}
&F_\vv: \RR^3\rightarrow \RR, \\
&F_\vv(p) = \langle \vv, p \rangle
\end{aligned}
\end{equation}
is a $g$-minimal foliation function.

Secondly, suppose that $U$ is $\RR^2$ or an open strip in $\RR^2$ and that $h:U\to\RR$ is a function
whose graph is a complete translator.  Then
\begin{equation}\label{general-H}
\begin{aligned}
   &\mathcal{H}: U\times \RR \rightarrow \RR, \\
   &\mathcal{H}(x_1,x_2,x_3) = x_3 - h(x_1,x_2)
\end{aligned}
\end{equation}

\begin{dfn}[Translators of finite type]
    \label{finite-type-definition}
We say that a translator $\Sigma$ in $\R^3$ is of finite type
provided there are finite numbers $c$, $K$, and $k$ such that
\begin{enumerate}
\item\label{area-item-def} $\area(\Sigma\cap\BB)\le cr^2$ for every ball $\BB$ of radius $r$.
\item\label{curvature-item-def} For every $p\in \Sigma$, the norm of the second fundamental form satisfies:
\[
    |A(p)| \, \min\{1, \dist(p,\partial \Sigma)\} \le K.
\] 
\item\label{Fv-item} $\mathsf{N}(F_\vv|_\Sigma)\le k$ for each horizontal unit vector $\vv$.
\item\label{H-item} $\mathsf{N}(\mathcal{H}|_\Sigma)\le k$ for every function $\mathcal{H}(x_1,x_2,x_3)=x_3-h(x_1,x_2)$ whose level sets
are grim reaper surfaces (tilted or untilted).
\end{enumerate}
\end{dfn}

\begin{remark}\label{entropy-remark}
Condition~\eqref{area-item-def} in Definition~\ref{finite-type-definition} is equivalent to the condition that $\lambda(\Sigma)<\infty$ (see ~\cite[Theorem~9.1]{BrianBoundary})
.
\end{remark}

\begin{remark}\label{general-position-remark}
Let $\Hh$ be the collection of all functions $\mathcal{H}$ as in condition~\eqref{H-item} of Definition~\ref{finite-type-definition}.  
By lower semicontinuity (\cite[Theorem~4.2]{HMW-A}), to prove
$\mathsf{N}(\mathcal{H}|_\Sigma)\le k$ for all $\mathcal{H}\in \Hh$, it suffices to prove it for a dense set of $\mathcal{H}\in \Hh$.  Likewise,
to prove $\mathsf{N}(F_\vv|_\Sigma)\le k$ for all horizontal unit vectors $\vv$, it suffices to prove it for a dense
set of such $\vv$.
\end{remark}

In the present paper, we are mainly interested in translators without boundary. 
For a surface without boundary,
$\dist(p,\partial \Sigma)=\infty$, so that the curvature bound in~\eqref{curvature-item-def} of Definition~\ref{finite-type-definition} simplifies to
\[
  |A(p)| \le K, \quad \mbox{for all $p\in \Sigma$.}
\]

\begin{thm} \label{th:finite-type}
    Let $\Sigma$ be a properly embedded, collapsed translator with $\lambda(\Sigma)<\infty$ and $\genus(\Sigma) <\infty.$ Then $\Sigma$ is of finite type.

\end{thm}
\begin{proof}
As we mentioned in Remark \ref{entropy-remark}, Statement \eqref{area-item-def} in Definition \ref{finite-type-definition} is a direct consequence of the fact $\lambda(\Sigma)<\infty.$

In order to prove Statement \eqref{curvature-item-def} we proceed by contradiction. As $\partial \Sigma=\varnothing$, then we are assuming that there exists a sequence $p_i \in \Sigma$ so that \[\lambda_i:=|A(p_i)| \to +\infty.\] It is clear that $|p_i|\to +\infty.$

Then, using Theorem 1 in \cite{BrianCompactness}, we have that $\lambda_i(\Sigma-p_i)$ converges to a complete, non-flat, embedded minimal surface in $\RR^3.$ However, as $\Sigma$ has finite topology, then the limit of $\lambda_i(\Sigma-p_i)$ must be simply connected. But this is contradicts the fact that the limit is non-planar, because the only complete embedded, simply connected  minimal surfaces with finite total curvature are planes. This contradiction proves statement \eqref{curvature-item-def}.

In order to prove Statement \eqref{Fv-item}, we are going to consider $\vv$ a horizontal unitary vector $\vv \neq \pm \ee_1.$ Theorem A implies that the map \[ F_\vv: \Sigma \longrightarrow \R,\] is proper.
Then, we can make use \cite[Theorem 4]{morserado} to deduce that \begin{equation} \label{eq:prey}\mathsf{N}(F_\vv|_\Sigma ) \leq \lambda(\Sigma)+\frac 12 \,\omega_P(\Sigma)-\chi(\Sigma),\end{equation}
where we recall that $\omega_P(\Sigma)$ represents the number of planar wings of $\Sigma$ (which is always even). Taking into account Remark \ref{general-position-remark}, we conclude that $\Sigma$ item \eqref{Fv-item} in Definition \ref{finite-type-definition}.

Finally, let us consider $x_3=h(x_1,x_2)$ a grim reaper (tilted or untilted) whose plane of symmetry is not parallel to the plane $\{x_1=0\},$ and define $\mathcal{H}(x_1,x_2,x_3)=x_3-h(x_1,x_2).$ 
As in the previous paragraph, we can use Theorem A to deduce that $\mathcal{H}: \Sigma \to\R$ is a proper map.  Hence, using again \cite[Theorem 4]{morserado}, we see that \[\mathsf{N}(\mathcal{H}|_\Sigma) \leq \lambda(\Sigma)-\chi(\Sigma).\]
Taking into account Remark \ref{general-position-remark}, we conclude that $\Sigma$ also satisfies \eqref{H-item} in Definition \ref{finite-type-definition}.
\end{proof}

\section{A classification result for translators in half-slabs}\label{sec:half-slab}

The main goal of this section is to establish a version of the half-space theorem for translators. To be more precise, we are going to study {\em simply connected}, collapsed translators with finite entropy whose limit planes, as $x_3 \to +\infty$, have multiplicity one. We will see that if we impose the additional hypothesis that the translator is contained in a half-slab, then we will show that it must be either a tilted grim reaper or a $\Delta$-wing.

\begin{dfn} \label{def:half-slab}
We say that a surface $\Sigma$ lies in a half-slab, if there are $\theta\in[0,\pi), c\in\R$ and $w>0$ so that 
\[
\Sigma\subseteq \mathcal{S}(w,\theta,c),
\]
where $\mathcal{S}(w,\theta,c):=\{(x_1,x_2,x_3)\in\R^3\;:\;|x_1|\leq w,\ \cos\theta x_2+\sin\theta x_3\geq c\}.$ 
\end{dfn}

Throughout this section, we will assume that $\Sigma$ is a complete, connected translating soliton of finite genus and finite width with only grim reaper-type wings. Furthermore, we assume that
\[
\Sigma \subseteq \{-w < x_1 < w\}, \quad \text{where } 2w = \width(\Sigma).
\]
We further assume that $\Sigma$ satisfies the multiplicity-one property at infinity. Specifically, this means that
\begin{equation}\label{limit-1}
\Sigma - t \ee_3 \to \{x_1 = w_1\} \cup \ldots \cup \{x_1 = w_{\lambda(\Sigma)}\} \quad \text{as } t \to +\infty
\end{equation}
with multiplicity one, where $\{-w = w_1 < \cdots < w_{\lambda(\Sigma)} = w\}$ is the partition from Theorem \ref{thm-A}.

Under these assumptions, \cite[Theorem 4.3 and Corollary 8.3]{entropy} say we can decompose $\Sigma$ as follows:
\begin{equation*}\label{division}
\Sigma = \bigcup_{i=1}^{\frac{\omega_G(\Sigma)}{2}} W_i^- \cup \Sigma(-s, s) \cup \bigcup_{i=1}^{\frac{\omega_G(\Sigma)}{2}} W_i^+,
\end{equation*}
where $W_i^+$ (respectively, $W_i^-$) denotes the right (respectively, left) wings of $\Sigma$, $\Sigma(-s, s) = \Sigma \cap \{|x_2| < s\}$, with $s$ a fixed large constant. We order the index $i$ so that if we set
\[
W_i^+ - t \ee_3 \to \{x_1 = \kappa^+_i, x_2 \geq s\} \cup \{x_1 = \rho^+_i, x_2 \geq s\} \quad \text{as } t \to +\infty,
\]
and
\[
W_i^- - t \ee_3 \to \{x_1 = \kappa^-_i, x_2 \leq -s\} \cup \{x_1 = \rho^-_i, x_2 \leq -s\} \quad \text{as } t \to +\infty,
\]
then, we have the following inequalities:
\begin{equation}\label{order}
\kappa^\pm_1 < \kappa^\pm_2 < \cdots < \kappa^\pm_{\omega_G(\Sigma)/2}, \quad \text{and} \quad \rho_i^\pm > \kappa_i^\pm, \quad \forall i \in \{1, \ldots, \omega_G(\Sigma)/2\}.
\end{equation}

The next lemma summarizes the properties we will use later. For simplicity, we will state it only for the right wings.
\begin{lem}\label{1-property}
The following properties hold:
\begin{itemize}
    \item[(i)] We have $\rho_i^+ \neq \rho_j^+$ and $\kappa_j^+ \neq \rho_i^+$ for all $i\neq j$;
    \item[(ii)] $\kappa^+_1 = w_1$;
    \item[(iii)] Either $\kappa^+_{2} = w_{2}$ or $\kappa^+_{2} = w_{3}$;
    \item[(iv)] $\kappa^+_{2} = w_{3}$ if and only if $\rho^+_1 = w_2$;
    \item[(v)] If $\kappa^+_i = w_j$, then $\rho^+_i = w_{j + 2L(i) + 1}$ for some $L(i) \geq 0$.
\end{itemize}
\end{lem}
\begin{proof}
The statements (i) and (ii) follows directly from \eqref{limit-1} and \eqref{order}. To derive (iii), we observe the following: By the multiplicity-one property at infinity and (ii), we know that $\kappa^+_{2} = w_{1+l}$ for some $l \geq 1$. If $l \geq 3$, then there must exist a wing $W^+_j$ for some $j \geq 3$ such that either $\kappa^+_j = w_3$ or $\rho^+_j = w_3$. In the latter case, we would have $\kappa^+_j = w_2$. Hence, in both cases, we would obtain 
\[
w_{1} = \kappa^+_1 < \kappa^+_j < \kappa^+_{2},
\] 
which contradicts \eqref{order}. This implies that $\kappa^+_{2}$ must be either $w_2$ or $w_3$.

Regarding (iv), if it was true $\kappa^+_2 = w_3$ and $\rho^+_1 \neq w_2$, then there would exist a wing $W^+_j,$ for some $j \geq 3,$ such that either $\kappa^+_j = w_2$ or $\rho^+_j = w_2$, leading to a contradiction with \eqref{order}.

The proof (v) can be obtained as follows: Let $W_i^+$ and $W_j^+$ be two right wings of $\Sigma.$ Suppose that either $\kappa_i^+>\kappa_j^+$ or $\rho_j^+>\rho_i^+$ one sees from the embeddedness of the surface that 
\[
\kappa_i^+>\kappa_j^+>\rho_j^+>\rho_i^+.
\]
Once we get this, $L(i)$ will be the number of wings $W_J^+$ so that $\kappa_i^+>\kappa_j^+>\rho_j^+>\rho_i^+.$ This completes the proof. 
\end{proof}

The following lemma provides a relationship between the left and right ordering. Here, we are going to assume henceforth that $\Sigma$ is simply connected.

\begin{lem}\label{2-property}
Assume that $\Sigma$ is also simply connected. Then, it holds $\kappa^+_{2}=w_{2},$ iff $\kappa^-_2=w_3.$ Similarly, $\kappa^-_{2}=w_{2},$ iff $\kappa^+_2=w_3.$
\end{lem}
\begin{proof}
Both cases will follow by verifying that we cannot have $\kappa^+_{2}=\kappa^-_{2}=w_{3}$. Indeed, if this were true, then by Lemma \ref{1-property}, we would have $\rho_1^+=\rho_1^-=w_2$. Consequently, the multiplicity-one property at infinity and \eqref{order} imply the nonexistence of wings of $\Sigma$ such that its upward tangent planes are half-planes of $\{x_1=w_2\}\cup\{x_1=w_3\}$ and so that the intersection  $\Sigma \cap \left\{x_1 = (w_2 + w_3)/2 \right\}$ is not compact. In particular, this implies that $\Sigma \cap \left\{x_1 = (w_2 + w_3)/2\right\}$ would be a compact set, and thus, it is a union of topological circles. Since $\Sigma$ is simply connected, each circle in $\Sigma \cap \left\{x_1 = (w_2 + w_3)/2\right\}$ is null-homotopic, and so, by the maximum principle, it would bound a planar region in $\Sigma \cap \left\{x_1 = (w_2 + w_3)/2\right\},$ which is absurd. This contradiction completes the proof of the lemma.
\end{proof}

The previous lemma provides an intriguing rigidity result for strictly mean convex translators in light of \cite{IRM}.
\begin{prop}
    Let $\Sigma$ be a complete embedded translating soliton of finite entropy and genus satisfying the multiplicity-one property at infinity so that ${\rm width}(\Sigma)\in[\pi, 2\pi]$ and without planar type wings. Then, $\Sigma$ is either a tilted grim reaper or $\Delta-$wing. 
\end{prop}
\begin{proof}
We only need to prove that $\lambda(\Sigma)=2,$ by Theorem \ref{Class-2} below. Assume this were not the case. The proof of Lemma \ref{2-property} implies that either $\kappa^+_{2}=w_{2},$ and $\kappa^-_2=w_3$, or $\kappa^+_{2}=\kappa^-_{2}=w_{3}.$ In the first case, Lemma \ref{gap-structure} below will implies that $\kappa_1^+=\kappa_1^-=w_1<\rho^-_1=\kappa^+_2=w_2<w_3\leq\rho_2^+<\rho_1^+\leq w_{\lambda(\Sigma)},$ which leads ${\rm Width}(\Sigma)> 2\pi,$ contradicting our hypothesis.

In the second case, one has
    \[
    \kappa_1^+=\kappa_1^-=w_1<\rho^+_1 =\rho^-_1=w_2<\kappa^+_2=\kappa^-_2=w_3\ {\rm and}\ \rho_2^+,\rho_2^-\leq w,
    \]
    which again implies ${\rm Width}(\Sigma)> 2\pi,$ absurd. Hence, we have proven that $\lambda(\Sigma)=2.$ This concludes the proof.
\end{proof}

The next lemma captures a strong geometric property enjoyed by the ``first'' wings of  $\Sigma$. More precisely, we have:
\begin{lem}\label{gap-structure}
If $\lambda(\Sigma)\geq4$, then $\Sigma$ possesses, up to reflection across $(x_1,x_3)-$plane,  
two right wings $W_1$, $W_2$, and a left wing $W_3$ so that
\begin{eqnarray*}
    \nonumber&W_{1}-t\ee_3\to\{x_1=\kappa_1,x_2\geq s\}\cup\{x_1=\rho_1,x_2\geq s\}&\\ 
    \nonumber &W_{2}-t\ee_3\to\{x_1=\kappa_2,x_2\geq s\}\cup\{x_1=\rho_2,x_2\geq s\}&\\
    &{\it and}&\\ 
    \nonumber&W_{3}-t\ee_3\to\{x_1=\kappa_3,x_2\leq -s\}\cup\{x_1=\rho_3,x_2\leq -s\}\ {\rm as}\ t\to+\infty,&
\end{eqnarray*}
where
\[\kappa_3=\kappa_1<\kappa_2=\rho_3<\rho_2<\rho_1.\]
\end{lem}
\begin{proof}
We begin by noting that $\rho_1^+$ is either equal to $w_2$ or greater than $w_2$. Assume first that $\rho^+_1 = w_2$. Then, by Lemma \ref{1-property} (iii), we have $\kappa^+_2 = w_3$. Consequently, Lemma \ref{2-property} and the embeddedness of $\Sigma$ imply that $\kappa^-_2 = w_2$. However, Lemma \ref{1-property} (iv) and (v) then ensure that $\rho^-_1 = w_{2 + 2L} > \rho^-_2$, where $L \geq 1$. In other words, we have verified that the wings $W_1^-$, $W_2^-$, and $W_1^+$ satisfy the desired properties.

Similarly, one can deduce that $W_1^+$, $W_2^+$, and $W^-_1$ also satisfy the desired properties. This concludes the proof of the lemma.
\end{proof}

\begin{rem}
    Notice that the wings $W_2$ lies geometrically above $W_1$ with respect to $x_3-$axis. To be more precise, if $L=\{(x_1,x_2,s)\;:\;s\in\R\}$ is any vertical line that intersects $W_1$ and $W_2,$ then
    \[
\sup\{s\;:\;(x_1,x_2,s)\in L\cap W_1\}\leq\inf\{s\;:\;(x_1,x_2,s)\in L\cap W_2\}.
    \] We will indicate this condition by writing $W_1\leq W_2.$

Comparing with prongs constructed by \cite{HMW-A}, this lemma accurately captures the behavior of the wings that appear in any prong. 
\end{rem}

Now we are ready to show the main result of this section.
\begin{thmE}\label{Classification-half-space}
The unique complete, simply connected translating solitons of finite width and entropy in $\R^3$ satisfying the multiplicity-one property at infinity and lying in a half-slab are tilted grim reaper surfaces, and $\Delta$-wings.
\end{thmE}
\begin{proof}
First, observe that we cannot have $\theta = \pi/2$ by \cite[Theorem 1]{Chini-Moller}. Moreover, up to reflection across the $(x_1,x_3)$-plane, we can assume that $\Sigma \subset S(\theta,0,w)$, where $w > 0,$ $\theta \in [0, \pi/2)$ and $S(\theta,0,w)=\{(x_1,x_2,x_3)\in\R^3\;:\;|x_1|\leq w,\ \cos\theta x_2+\sin\theta x_3\geq 0\}.$  Under this normalization, Theorem \refThmF below (see also \cite[Theorem 1]{Chini} or \cite[Theorem 10.1]{entropy}) shows that we only need to verify that $\lambda(\Sigma) = 2$. We proceed by contradiction and assume that $\lambda(\Sigma) \geq 4$. To reach a contradiction, we divide the proof into two cases: $\displaystyle\inf_\Sigma z>-\infty$ and $\displaystyle\inf_\Sigma z=-\infty.$

Assume first that $\displaystyle\inf_\Sigma z > -\infty$. Without loss of generality, we may assume $\displaystyle\inf_\Sigma z = 0$. Then Lemma \ref{gap-structure} implies that there exist two right wings $W_1$ and $W_2$ such that $W_2 \geq W_1$, up to reflection across the $(x_1,x_3)$-plane. Now, by Theorem \ref{prop:Asymp-GR}, we obtain \[\lim_{s\to+\infty} (W_1-s\ee_2-f_{W_1}(s)\ee_3)=\mathcal{G}\left(\eta,\frac{\kappa_1+\rho_1}{2}\right)\] and \[\lim_{s\to+\infty}(W_2-s\ee_2-f_{W_2}(s\ee_3))=\mathcal{G}\left(\varsigma,\frac{\kappa_2+\rho_2}{2}\right),\] where $\mathcal{G}(\beta,t)$ denotes the tilted grim reaper of width $\pi/\cos \beta$ with plane of symmetry $\{x_1=t\}$, $(\rho_1-\kappa_1)\cos\eta=\pi$ and $(\rho_2-\kappa_2)\cos\varsigma=\pi.$ By our assumptions on $\Sigma$, we have $\frac{\pi}{2} > \eta > \varsigma > 0$. Consequently, Proposition \ref{growth-function} implies that $W_1$ and $W_2$ must have an intersection, which is a contradiction. 

Now, assume instead that $\displaystyle\inf_\Sigma z = -\infty$. In this case, we must have $\theta \in (0, \pi/2)$. One can see that the two wings of Lemma \ref{gap-structure} cannot be left wings. Otherwise, using the notation of the previous paragraph, we could infer $\varsigma < \eta < \theta + \frac{\pi}{2}$, which would mean that the wings $W_1$ and $W_2$ intersect by Proposition \ref{growth-function}, yielding a contradiction. Consequently, we conclude that $W_1$ and $W_2$ are the right wings. Using the notation from the second paragraph, it can be shown that $\varsigma > \eta > \theta + \frac{\pi}{2}$. 

On the other hand, the third wing $W_3$ from Lemma \ref{gap-structure} would be a left wing, such that \[\lim_{s\to-\infty}(W_3-s\ee_2-f_{W_2}(s)\ee_3)=\mathcal{G}\left(\beta,\frac{\kappa_1+\kappa_2}{2}\right),\] where $\beta>\theta+\frac{\pi}{2}.$ Hence, Proposition \ref{growth-function} implies that $W_3$ cannot lie in $S(\theta,0,w)$, which again leads to a contradiction. This completes the proof of the theorem.
\end{proof}
\begin{cor}
The unique complete translating graph of finite width in $\R^3$ lying in a half-slab are tilted grim reaper surfaces, and $\Delta$-wings.
\end{cor}
\begin{proof}
    Indeed, by \cite[Propositions 3.6 and 4.1]{entropy}, any complete graph of finite width is simply connected,  has finite entropy and satisfies the multiplicity-one property at infinity. Hence, the result follows from Theorem \refThmE.
\end{proof}

\section{Translators of entropy two and grim reaper type wings}\label{sec:two-clas}
We conclude the paper with a section devoted to characterizing translators of entropy 2 (see Theorem \refThmF in the introduction and compare with the discussion following Theorem~\refThmFnospace. More precisely, our goal is to classify embedded, complete, collapsed translators \(\Sigma \subset \mathbb{R}^3\) that meet the following criteria:
\begin{enumerate}[(a)]
\item  \(\Sigma\) has finite genus.
\item The entropy \(\lambda(\Sigma)\) equals 2.
\item The limit \(\Sigma + t \, \ee_3 \to \varnothing\) as \(t \to +\infty\).
\end{enumerate}
Hypothesis (c) implies that \(\Sigma\) consists solely of grim reaper-type wings, as per our terminology. Meanwhile, hypothesis (b) specifies that \(\Sigma\) has exactly two wings.
We are going to show that a complete, embedded translator \(\Sigma \subset \mathbb{R}^3\), satisfying (a), (b) and (c) above, must fall into one of two distinct categories: either a tilted grim reaper surface or a \(\Delta\)-wing surface.

For the special case where the width of the strip is precisely \(\pi\), we will establish that the only possible configuration for \(\Sigma\) is the standard grim reaper cylinder. This is because a strip of width \(\pi\) uniquely restricts the shape of the translator, ensuring that \(\Sigma\) must take the form of the classical grim reaper soliton.\vskip3mm

Thus, henceforth, \(\Sigma\) will represent a complete translator with entropy 2, finite genus, and width \(2w\), having two wings of the grim reaper type. We can further normalize \(\Sigma\) (up to a rigid motion) such that:
\begin{equation} \label{norm}
    \Sigma \subseteq \mathcal{S}_w \quad \text{and} \quad \lim_{t \to +\infty} (\Sigma - t \, \ee_3) = \{x_1 = -w\} \cup \{x_1 = w\}.
\end{equation}

\begin{lem}\label{sym}
 $\Sigma$ is symmetric with respect the plane $\{x_1=0\}$. Moreover, it is bi-graph over this plane.   
\end{lem}
\begin{proof} The proof of this lemma relies on the classical Alexandrov moving planes method, a symmetry technique frequently used in geometric analysis. Although the argument is well known, we include it here for completeness.

Hence, let us define \(\Sigma_{-}(t) = \{(x_1, x_2, x_3) \in \Sigma \;:\; x_1 \leq t\}\), \(\Sigma_{+}(t) = \{(x_1, x_2, x_3) \in \Sigma \;:\; x_1 \geq t\}\), and \(\Sigma_{-}^*(t)\) to be the reflection of \(\Sigma_-(t)\) across the plane \(\{x_1 = t\}\).

We claim there exists \(\eta\) small enough such that
\begin{eqnarray}\label{moving-plane}
\Sigma_{-}(-w + \eta)\ \text{is an} \ \ee_1\text{-graph} \ {\rm and}\ \Sigma_-^*(-w + \eta) \leq \Sigma_+(-w + \eta),
\end{eqnarray}
where by \(S_1 \leq S_2\)  we mean that for every horizontal line $L$ which intersect $S_1$ and $S_2$ all intersections of $S_1$ with $L$ lie to the left of all intersections of $S_2$ with $L.$ In other words, in our particular situation we want to see that
\[
\sup\{s\;:\; (s,x_2,x_3)\in L\cap \Sigma_-(-w+\eta)^*\leq\inf\{s\;:\;(s,x_3,x_3)\in L\cap\Sigma_+(-w+\eta)\}\}.
\]

To prove this, we know \(\Sigma \setminus \{|x_2| \geq s\}\), for sufficiently large \(s\), has two connected components that are \(\ee_1\)-bi-graphs, by Proposition \ref{bigraph}. Thus, from Theorem \ref{structure-thm}, Theorem \ref{thm:exponential}, Proposition \ref{conseq.1}, and \eqref{norm}, we may conclude that there exists \(\epsilon > 0\) small enough such that \(\Sigma_-(-w + \eta)\) is an \(\ee_1\)-graph for all \(\eta < \epsilon \leq w\). This concludes the proof of the first part of the claim.

Regarding the second part, define the set
\[
\mathcal{A}:=\{s<0\;:\; (\Sigma_-^*(-w+\eta)+s\ee_1)\cap\Sigma_+(-w+\eta)=\varnothing\}.
\]
Let \(s_0=\sup\mathcal{A}\leq 0.\) We assert that \(s_0=0.\) In fact, if \(s_0<0\), we would have either 
\[
(\Sigma_-^*(-w+\eta)+s_0\ee_1)\cap\Sigma_+(-w+\eta)=\varnothing \ {\rm and}\ \mathrm{dist}((\Sigma_-^*(-w+\eta)+s_0\ee_1),\Sigma_+(-w+\eta))=0
\]
or 
\[
(\Sigma_-^*(-w+\eta)+s_0\ee_1)\cap\Sigma_+(-w+\eta)\neq \varnothing.
\]
However, the maximum principle implies that the second case cannot happen. Regarding the first case, it follows from Theorem \ref{thm:exponential} that there exist two sequences \(\{q_k\}\subseteq \Sigma_-^*(-w+\eta)+s_0\ee_1)\) and \(\{p_k\}\subseteq\Sigma_+(-w+\eta)\) which diverge along a wing \(W\) of \(\Sigma\) such that \(|p_k-q_k|\to 0\) and \( x_2(q_k)\to+\infty.\)

Assume first that \(\mathrm{dist}(p_k,\mathcal{M}_W)<+\infty.\) By Theorem \ref{prop:Asymp-GR}, we have \(\Sigma - x_2^k \ee_2 - f_W(x_2^k) \ee_3 \to \mathcal{G}\), where \(\mathcal{G}\) is the tilted grim reaper of width \(2w\) in \(\mathcal{S}_w\), and \(p_k=(x_1^k,x_2^k,x_3^k).\) Furthermore, we may assume that \(x_1^k\to x_\infty.\) Now, our assumption about \(s_0\) implies that \((x_\infty,0,0)\in\mathcal{G}_{-}^{*}(-w + \eta) \cap \mathcal{G}_{+}(-w + \eta) \neq \varnothing\), which is a contradiction since \(\eta<\epsilon\leq w.\)

On the other hand, if \(\mathrm{dist}(p_k,\mathcal{M}_W)=+\infty\) and writing \(p_k=(x_1^k,x_2^k,x_3^k)\) as before, after passing to a subsequence, we can assume \(\Sigma - x_2^k \ee_2 - x_3^k \ee_3 \to \{x_1 = -w\} \cup \{x_1 = w\}\), by Proposition \ref{conseq.}, and \(x_1^k \to x_\infty\), which leads to a contradiction again since \(\eta < w\) and \((x_\infty, 0, 0) \in \{x_1 = -w\} \cap \{x_1 = w\}.\) In both cases, we arrive at a contradiction, and thus it must hold that \(\sup\mathcal{A}=0.\)

Once we have proven \eqref{moving-plane}, define
\[
\mathcal{B}=\begin{Bmatrix}
 & \Sigma_{-}(-w + \eta)\ \text{is an} \ \ee_1\text{-graph} \\ 
\eta\geq0\;:\; & {\rm and} \\ 
& \Sigma_-^*(-w + \eta) \leq \Sigma_+(-w + \eta)
\end{Bmatrix}.
\]
We will normalize the Gauss map \(\nu\) of \(\Sigma\) so that \(\langle\nu,\ee_1\rangle >0\) when \(\Sigma_{-}(-w + \eta)\) is a smooth \(\ee_1\)-graph. Let $b=\sup \mathcal{B}.$ We claim that $b=w.$ We argue by contradiction. Let us assume $b<w.$ Our first step is to prove that $b\in\mathcal{B}.$ Indeed, $\Sigma_{-}(-w + b)$ is a continuous $\ee_1-$graph and $\Sigma_-^*(-w + b) \leq \Sigma_+(-w + b),$ by continuity. We claim $\langle \nu,\ee_1\rangle>0$ on $\Sigma_{-}(-w + b).$ Namely, this is the case for ${\rm int} \Sigma_{-}(-w + b).$ In turn, if $\langle \nu,\ee_1\rangle=0$ at $q\in\partial \Sigma_{-}(-w + b),$ then by the boundary maximum principle applying to $\Sigma_{-}^*(-w + b)$ and $\Sigma_{+}(-w + b)$ at $p$ we achieve $\Sigma_{-}^*(-w + b)=\Sigma_{+}(-w + b),$ which is impossible since $b<w.$ Hence, we must have $\langle \nu,\ee_1\rangle=0$ on $\Sigma_{-}(-w + b),$ which finally implies that $b\in\mathcal{B}.$

Now, we are going to prove the existence of $\epsilon>0$ so that $b+\epsilon\in\mathcal{B},$ which is impossible, and thus it shall hold true that $b=w.$ Take $\eta=\frac{w-b}{2}.$ From Proposition \ref{conseq.1}, we can find $s_\eta>0$ so that if we set $W_1$ and $W_2$ to be the right and left wings of $\Sigma\setminus \{|x_2|\geq s_\eta\},$ respectively, then for $i=1,2$, one has 
\begin{multline*}
\min \left[{\rm dist}\{\{x_1=-w\},\mathcal{M}_{W_i}\cap\{|x_2|\geq s_\eta\}\}, \right.\\ \left.{\rm dist}\{\{x_1=w\},\mathcal{M}_{W_i}\cap\{|x_2| \geq s_\eta\}\}\right]\geq w-\eta.\end{multline*}
From this and using that $W_1$ and $W_2$ are bi-graphs, by Proposition \ref{bigraph}, we conclude that $\Sigma_{-}(-w+b+\eta)\cap\{|x_2|\geq s_\eta\}$ is a smooth $\ee_1-$graph. In turn, since $\langle \nu,\ee_1\rangle >0$ on $\Sigma_{-}(-w+b),$ we can find $\epsilon<\eta$ small enough so that
\[
\langle \nu, \ee_1\rangle>0\ {\rm on}\  \Sigma_{-}(-w+b+\epsilon)\cap\{|x_2|\leq s_\eta\}.
\]
Pulling it all together, we have proven that $\langle \nu, \ee_1\rangle>0$ on $\Sigma_{-}(-w+b+\epsilon),$ which directly infers $\Sigma_-(-w+b+\epsilon)$ is an $\ee_1-$graph. Regarding the inequality $\Sigma_-^*(-w + b+\epsilon) \leq \Sigma_+(-w + b+\epsilon),$ we can apply the same proof used to conclude $\sup \mathcal{A}=0$ to conclude this. Therefore, it follows that $b+\epsilon\in\mathcal{B}.$ This concludes the proof of the lemma.
\end{proof}

From the previous lemma, we conclude that $\{p \in \Sigma \: : \: \langle \nu(p),\ee_1\rangle=0\}$ is a smooth $1$-manifold, where $\nu$ denotes the Gauss map of $\Sigma.$ Besides, this set lies on $\{x_1=0\}.$ 
Furthermore, we have that $\Sigma$ is a bi-graph over a domain $\Omega$ in the $(x_2,x_3)$-plane. The boundary $\partial \Omega$ consists of a non-compact arc $\Gamma$ and $g=\genus(\Sigma)$ Jordan curves, $\gamma_1, \ldots,\gamma_g,$ which are contained in one of the two connected components of the complement of $\Gamma$ in the $(x_2,x_3)$-plane (see Figure \ref{fig:holedomains}.) Indeed, it is not hard to see that $\Gamma=\{p\in\Sigma\;:\;\langle p,\ee_3\rangle=\inf\{\langle q,\ee_3\rangle\;:\;q\in\Sigma\cap\{x_2=F_{\ee_2}(p)\}\}\}\subseteq\{x_1=0\}$. Let us define  $$\mathcal{H}:=\{p \in \Sigma \: : \: H(p)=0\}=\{p \in \Sigma \: : \: \langle \nu(p) , \ee_3\rangle=0\}.$$ Note $C(F_{\ee_2})=\{\langle \nu,\ee_1\rangle=0\}\cap\mathcal{H},$ where $C(F_{\ee_2})$ denotes the set of critical points of $F_{\ee_2}$.

 \begin{figure}
     \centering
     \includegraphics[width=0.4\linewidth]{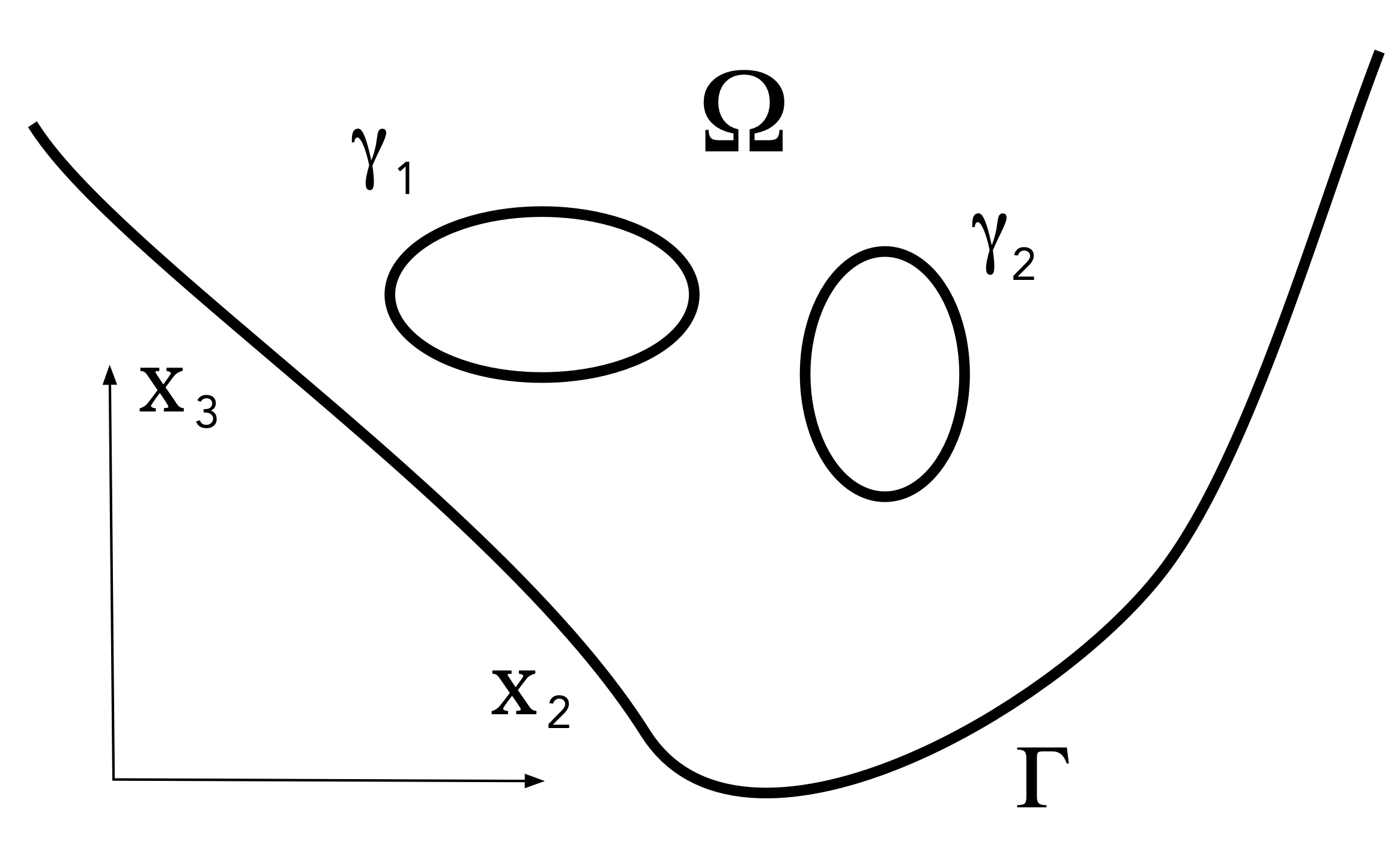}
     \caption{The  domain $\Omega$, for $g=2$. }
     \label{fig:holedomains}
 \end{figure}




\begin{lem} \label{lem:Gamma-graph}
   One has $\Gamma\cap\mathcal{H}=\varnothing$. In particular, the curve $\Gamma$ is a graph over the line $\{x_1=x_3=0\}$.
\end{lem}
\begin{proof}
    If there were a critical point $p \in C(F_{\ee_2}) \cap \Gamma$, then  the intersection $\Sigma\cap (p+[\ee_2]^\perp)$ would be a union of $n$ smooth curves passing through $p$ and meeting at angles $\frac{\pi}{n}$ for some $n\geq2$. But then $p$ could not be a minimum of
$\{\langle q,\ee_3\rangle \;:\; q\in\Sigma\cap\{x_2=\langle p,\ee_2\rangle\}\}$, contradicting our hypothesis on $p$.  In particular, there are no critical point on $\Gamma$ and $\Gamma\cap\mathcal{H}=\varnothing.$

    Regarding the second statement, we first observe that the vertical line $\{(0,x_2,x_3)\;:\;x_3\in\R\}$ cannot contain a segment of $\Gamma$, as established in the first part of the proof. Next, let us assume that $\{(0,x_2,x_3)\;:\;x_3\in\R\}$ intersects $\Gamma$ at more than one point. In this case, we could find an arc $\alpha$ of $\Gamma$ such that the endpoints of $\alpha$ lie on $\{(0,x_2,x_3)\;:\;x_3\in\R\}$, and $F_{\ee_2}|_{\alpha}$ would have either a maximum or minimum point. However, this point would belong to $\mathcal{H}$, leading to a contradiction. This concludes the proof. 
\end{proof}
We would like to note that there is a proof of the previous lemma using Morse-Rad\'o theory, but we have chosen this approach because it is shorter.

Notice that Lemma \ref{lem:Gamma-graph} implies that we can parametrize $\Gamma$ by a smooth map $f_\Sigma:\R \to\Sigma$ so that $f_\Sigma(t)=\Gamma\cap\{x_2=t\}.$ Hence, we define \[\hslash(\Sigma)=\sup\{|p-f_\Sigma(\langle\ee_2,p\rangle)|\;:\; p\in\Sigma\cap\{x_1=0\}\}\geq0.\] Note that $\hslash(\Sigma)<+\infty,$ since $\Sigma$ has finite topology.

\begin{lem}\label{topology}
One has that  $\hslash(\Sigma)=0.$ In particular, $\genus(\Sigma)=0.$
\end{lem}
\begin{proof}
    Since we know that $\Sigma$ is an $\ee_1-$bi-graph, let us decompose \[\Sigma=\Graph(u)\cup \Graph(u)^*,\] where $u:\Omega\to\R$ is a non-negative smooth function so that $u=0$ on $\partial \Omega$. Let us denote  by $\Graph(u)^*$ the reflection of $\Graph(u)$ across $\{x_1=0\}.$

    For any $\varepsilon>0$, define \[\mathcal{A}=\{s\in[0,+\infty)\;:\; \Graph(u)\cap (\Graph(u)-(\hslash(\Sigma)+\varepsilon)\; \ee_3+s\; \ee_1)=\varnothing\}.\] 
    Since ${\rm width}(\Sigma)<+\infty,$ we get $\mathcal{A}\neq\varnothing.$ Using that $\Gamma$ is a graph over the line $\{x_1=0\}$ over the $x_1x_2-$plane and $\Sigma$ is an $\ee_1-$bi-graph, we conclude that $\inf\mathcal{A}=0.$ Furthermore, the maximum principle implies that $\Graph(u)\cap (\Graph(u)-(\hslash(\Sigma)+\varepsilon)\ee_3)=\varnothing.$ Finally, letting  $\varepsilon\searrow 0$ we get
    \[
    \Graph(u)\geq(\Graph(u)-\hslash(\Sigma)\ee_3)\ {\rm and}\ \Graph(u)\cap (\Graph(u)-\hslash(\Sigma)\ee_3)\neq\varnothing.
     \]
     Thus, the boundary maximum principle guarantees \(\Graph(u)=\Graph(u)-\hslash(\Sigma)\ee_3,\) which means $\hslash(\Sigma)=0.$ This concludes the proof. 
\end{proof}

Finally, we can prove the main result of this section.

\begin{thmF}\label{Class-2}
Let $\Sigma$ be  a complete, embedded translator with finite width, finite genus and entropy $\lambda(\Sigma)=2$. Assume that  the limit, as $t \to +\infty$, of $\Sigma+t \ee_3$ is the empty set. Then $\Sigma$ is either a tilted grim reaper or a $\Delta$-wing. In particular, if the width of $\Sigma$ is $\pi$, then $\Sigma$ is the standard grim reaper cylinder.
\end{thmF}
\begin{proof}
    Lemma \ref{topology} says that $\Sigma$ is simply connected. Hence, by Chini's result \cite[Theorem 1]{Chini} (see also \cite[Theorem 10.1]{entropy}) we have that $\Sigma$ is either a (possibly tilted) grim reaper cylinder or a $\Delta-$wing.
    
    Regarding the second part, when the width is $\pi$, we conclude that $\Sigma$ coincides with the standard grim reaper cylinder, by the classification theorem for mean convex translators in \cite{HIMW}.
\end{proof}

\bibliographystyle{acm}
\bibliography{bibliography_v3.bib}

\begin{thebibliography}{10}

\bibitem{AW}
{\sc Altschuler, S., and Wu, L.}
\newblock Translating surfaces of the non-parametric mean curvature flow with
  prescribed contact angle.
\newblock {\em Calc. Var. Partial Differential Equations 2}, 1 (1994),
  101--111.

\bibitem{Bamler-Lai1}
{\sc Bamler, R.~H., and Lai, Y.}
\newblock Classification of ancient cylindrical mean curvature flows and the
  mean convex neighborhood conjecture, 2025.
\newblock arXiv:2512.24524.

\bibitem{Bamler-Lai2}
{\sc Bamler, R.~H., and Lai, Y.}
\newblock The {P}{D}{E}-{O}{D}{I} principle and cylindrical mean curvature
  flows, 2025.
\newblock arXiv:2512.25050.

\bibitem{Chini}
{\sc Chini, F.}
\newblock Simply connected translating solitons contained in slabs.
\newblock {\em Geom. Flows 5}, 1 (2020), 102--120.

\bibitem{Chini-Moller2}
{\sc Chini, F., and M\o{}ller, N.~M.}
\newblock Ancient mean curvature flows and their spacetime tracks, 2019.
\newblock Preprint, arXiv:1901.05481.

\bibitem{Chini-Moller}
{\sc Chini, F., and M\o{}ller, N.~M.}
\newblock Bi-halfspace and convex hull theorems for translating solitons.
\newblock {\em Int. Math. Res. Not. IMRN}, 17 (2021), 13011--13045.

\bibitem{Otis-Felix-Choi-Mantou}
{\sc Chodosh, O., Choi, K., Mantoulidis, C., and Schulze, F.}
\newblock Mean curvature flow with generic initial data.
\newblock {\em Invent. Math. 237}, 1 (2024), 121--220.

\bibitem{Otis-Felix}
{\sc Chodosh, O., and Schulze, F.}
\newblock Uniqueness of asymptotically conical tangent flows.
\newblock {\em Duke Math. J. 170}, 16 (2021), 3601--3657.

\bibitem{CDDHS}
{\sc Choi, B., Daskalopoulos, P., Du, W., Haslhofer, R., and Sesum, N.}
\newblock Classification of bubble-sheet ovals in {$\Bbb{R}^4$}.
\newblock {\em Geom. Topol. 29}, 2 (2025), 931--1016.

\bibitem{CHH2}
{\sc Choi, K., Haslhofer, R., and Hershkovits, O.}
\newblock A gradient estimate for the linearized translator equation, 2025.
\newblock arXiv 2509.07629.

\bibitem{CHH1}
{\sc Choi, K., Haslhofer, R., and Hershkovits, O.}
\newblock The linearized translator equation and applications, 2025.
\newblock arXiv 2509.06667.

\bibitem{Choi}
{\sc Choi, K., Seo, D.-H., Su, W.-B., and Zhao, K.-W.}
\newblock Uniqueness of tangent flows at infinity for finite-entropy shortening
  curves.
\newblock {\em Geom. Funct. Anal. 0\/} (2025), 1--55.

\bibitem{CM12}
{\sc Colding, T.~H., and Minicozzi, W.~P.}
\newblock Generic mean curvature flow {I}; generic singularities.
\newblock {\em Ann. of Math. (2) 175}, 2 (2012), 755--833.

\bibitem{Colding-Minicozzi}
{\sc Colding, T.~H., and Minicozzi, W.~P.}
\newblock Uniqueness of blowups and Łojasiewicz inequalities.
\newblock {\em Ann. of Math. (2) 182}, 1 (2015), 221--285.

\bibitem{CM-fre}
{\sc Colding, T.~H., and Minicozzi, W.~P.}
\newblock Complexity of parabolic systems.
\newblock {\em Publ. Math. Inst. Hautes Études Sci. 132\/} (2020), 83--135.

\bibitem{CM25}
{\sc Colding, T.~H., and Minicozzi, W.~P.}
\newblock Quantitative uniqueness for mean curvature flow.
\newblock {\em Journal of Mathematical Study (to appear)\/} (2025).
\newblock Preprint arXiv:2502.03634.

\bibitem{DZ}
{\sc Denisov, V.~N., and Zhikov, V.~V.}
\newblock Stabilization of the solution of the {C}auchy problem for parabolic
  equations.
\newblock {\em Mat. Zametki 37}, 6 (1985), 834--850.

\bibitem{Duffin}
{\sc Duffin, R.~J.}
\newblock Yukawan potential theory.
\newblock {\em J. Math. Anal. Appl. 35}, 1 (1971), 105--130.

\bibitem{Ecker}
{\sc Ecker, K.}
\newblock {\em Regularity theory for mean curvature flow}, vol.~57 of {\em
  Progress in Nonlinear Differential Equations and Their Applications}.
\newblock Birkhäuser Boston, Boston, MA, 2004.

\bibitem{Ecker-Huisken}
{\sc Ecker, K., and Huisken, G.}
\newblock Interior estimates for hypersurfaces moving by mean curvature.
\newblock {\em Invent. Math. 105\/} (1991), 547--569.

\bibitem{Evans}
{\sc Evans, L.~C.}
\newblock {\em Partial Differential Equations}, 2nd~ed., vol.~19.
\newblock American Mathematical Society, 2022.

\bibitem{Evans-Spruck}
{\sc Evans, L.~C., and Spruck, J.}
\newblock Motion of level sets by mean curvature iii.
\newblock {\em J. Geom. Anal. 2\/} (1992), 121--150.

\bibitem{GHLM}
{\sc Gama, E.~S., Heinonen, E., Lira, J.~H., and Mart\'{\i}n, F.}
\newblock The {J}enkins-{S}errin problem for translating horizontal graphs in
  ${M} \times \mathbb{R}$.
\newblock {\em Rev. Mat. Iberoam. 37}, 3 (2020), 1083--1114.

\bibitem{entropy}
{\sc Gama, E.~S., Mart\'{\i}n, F., and Møller, N.~M.}
\newblock Finite entropy translating solitons in slabs, 2022.
\newblock Amer. J. Math. (to appear). Preprint arXiv:2209.01640.

\bibitem{GMM-FollowUp}
{\sc Gama, E.~S., Mart\'{\i}n, F., and Møller, N.~M.}
\newblock Collapsed translating solitons with planar wings, 2025.
\newblock Work in progress.

\bibitem{HIMW}
{\sc Hoffman, D., Ilmanen, T., Mart\'{\i}n, F., and White, B.}
\newblock Graphical translators for the mean curvature flow.
\newblock {\em Calc. Var. Partial Differential Equations 58}, 117 (2019).

\bibitem{HIMW-2}
{\sc Hoffman, D., Ilmanen, T., Mart\'{\i}n, F., and White, B.}
\newblock Notes on translating solitons of the mean curvature flow.
\newblock In {\em Minimal surfaces: integrable systems and visualisation},
  vol.~349 of {\em Springer Proc. Math. Stat.} Springer, Cham, 2021,
  pp.~147--168.

\bibitem{HMW-2}
{\sc Hoffman, D., Mart\'{\i}n, F., and White, B.}
\newblock Nguyen's tridents and the classification of semigraphical translators
  for mean curvature flow.
\newblock {\em J. Reine Angew. Math. 786\/} (2022), 79--105.

\bibitem{HMW-1}
{\sc Hoffman, D., Mart\'{\i}n, F., and White, B.}
\newblock Scherk-like translators for mean curvature flow.
\newblock {\em J. Differential Geom. 122}, 3 (2022), 421--465.

\bibitem{morserado}
{\sc Hoffman, D., Mart\'{\i}n, F., and White, B.}
\newblock Morse-{R}adó theory for minimal surfaces.
\newblock {\em J. Lond. Math. Soc. 108}, 4 (2023), 1669--1700.

\bibitem{HMW-SpruckVol}
{\sc Hoffman, D., Mart\'in, F., and White, B.}
\newblock Annuloids and {$\Delta $}-wings.
\newblock {\em Adv. Nonlinear Stud. 24}, 1 (2024), 74--96.

\bibitem{HMW-A}
{\sc Hoffman, D., Mart\'{\i}n, F., and White, B.}
\newblock Translating annuli for mean curvature flow.
\newblock {\em Adv. Math. 455\/} (2024), 109875.

\bibitem{Hoffman-Meeks}
{\sc Hoffman, D., and Meeks, W. H.~I.}
\newblock The strong halfspace theorem for minimal surfaces.
\newblock {\em Invent. Math. 101\/} (1990), 373--377.

\bibitem{Hu93}
{\sc Huisken, G.}
\newblock Local and global behaviour of hypersurfaces moving by mean curvature.
\newblock In {\em Differential geometry: partial differential equations on
  manifolds (Los Angeles, CA, 1990)}, vol.~54 of {\em Proc. Sympos. Pure Math.}
  Amer. Math. Soc., Providence, RI, 1993, pp.~175--191.

\bibitem{ilmanen}
{\sc Ilmanen, T.}
\newblock {\em Elliptic regularization and partial regularity for motion by
  mean curvature}, vol.~108 of {\em Mem. Amer. Math. Soc.}
\newblock American Math. Soc., 1994.

\bibitem{IRM}
{\sc Impera, D., M\o{}ller, N.~M., and Rimoldi, M.}
\newblock Rigidity and non-existence results for collapsed translators.
\newblock {\em Int. Math. Res. Not. IMRN}, 7 (2025), Paper No. rnaf076, 15.

\bibitem{NicosSurvey}
{\sc Kapouleas, N.}
\newblock Doubling and desingularization constructions for minimal surfaces.
\newblock In {\em Surveys in geometric analysis and relativity}, vol.~20 of
  {\em Adv. Lect. Math. (ALM)}. Int. Press, Somerville, MA, 2011, pp.~281--325.

\bibitem{KKM}
{\sc Kapouleas, N., Kleene, S.~J., and M\o{}ller, N.~M.}
\newblock Mean curvature self-shrinkers of high genus: non-compact examples.
\newblock {\em J. Reine Angew. Math. 739\/} (2018), 1--39.

\bibitem{Ilyas}
{\sc Khan, I.}
\newblock The structure of translating surfaces with finite total curvature.
\newblock {\em Calc. Var. Partial Differential Equations 62}, 3 (2023), Paper
  No. 104, 28.

\bibitem{KM}
{\sc Kleene, S., and M\o{}ller, N.~M.}
\newblock Self-shrinkers with a rotational symmetry.
\newblock {\em Trans. Amer. Math. Soc. 366}, 8 (2014), 3943--3963.

\bibitem{Li-Tam}
{\sc Li, P., and Tam, L.-F.}
\newblock Complete surfaces with finite total curvature.
\newblock {\em J. Differential Geom. 33}, 1 (1991), 139--168.

\bibitem{Lotay}
{\sc Lotay, J.~D., Schulze, F., and Sz\'ekelyhidi, G.}
\newblock Ancient solutions and translators of {L}agrangian mean curvature
  flow.
\newblock {\em Publ. Math. Inst. Hautes \'Etudes Sci. 140\/} (2024), 1--35.

\bibitem{Lynch-Tinaglia}
{\sc Lynch, S., and Tinaglia, G.}
\newblock Translators {A}symptotic to {P}lanes.
\newblock {\em J. Geom. Anal. 36}, 1 (2026), Paper No. 27.

\bibitem{MM08}
{\sc Magni, A., and Mantegazza, C.}
\newblock Some remarks on {H}uisken's monotonicity formula for mean curvature
  flow.
\newblock In {\em Singularities in nonlinear evolution phenomena and
  applications}, vol.~9 of {\em CRM Series}. Ed. Norm., Pisa, 2009,
  pp.~157--169.

\bibitem{Meeks-Yau}
{\sc Meeks~III, W.~H., and Yau, S.~T.}
\newblock The existence of embedded minimal surfaces and the problem of
  uniqueness.
\newblock {\em Math Z 179}, 6 (1982), 151–--168.

\bibitem{Meyers-Serrin}
{\sc Meyers, N., and Serrin, J.}
\newblock The exterior {D}irichlet problem for second order elliptic partial
  differential equations.
\newblock {\em J. Math. Mech. 9\/} (1960), 513--538.

\bibitem{NielsThesis}
{\sc M\o{}ller, N.}
\newblock {\em Mean curvature flow self-shrinkers with genus and asymptotically
  conical ends}.
\newblock PhD thesis, MIT, 2012.

\bibitem{Ng}
{\sc Nguyen, X.~H.}
\newblock Translating tridents.
\newblock {\em Comm. Partial Differential Equations 34}, 1-3 (2009), 257--280.

\bibitem{Serrin70}
{\sc Serrin, J.}
\newblock On the strong maximum principle for quasilinear second order
  differential inequalities.
\newblock {\em J. Functional Analysis 5\/} (1970), 184--193.

\bibitem{Simon83}
{\sc Simon, L.}
\newblock Asymptotics for a class of nonlinear evolution equations, with
  applications to geometric problems.
\newblock {\em Ann. of Math. (2) 118}, 3 (1983), 525--571.

\bibitem{spadaro}
{\sc Spadaro, E.}
\newblock Mean-convex sets and minimal barriers.
\newblock {\em Matematiche (Catania) 75}, 1 (2020), 353--375.

\bibitem{Spruck-Xiao}
{\sc Spruck, J., and Xiao, L.}
\newblock Complete translating solitons to the mean curvature flow in
  {$\mathbb{R}^3$} with nonnegative mean curvature.
\newblock {\em Amer. J. Math. 142}, 3 (2020), 993--1015.

\bibitem{sun-wang}
{\sc Sun, A., and Wang, Z.}
\newblock On {M}ean {C}urvature {F}low {T}ranslators with {P}rescribed {E}nds.
\newblock {\em Arch. Ration. Mech. Anal. 249}, 5 (2025), Paper No. 51.

\bibitem{Wang}
{\sc Wang, X.-J.}
\newblock Convex solutions to the mean curvature flow.
\newblock {\em Ann. of Math. (2) 173}, 3 (2011), 1185--1239.

\bibitem{Whi03}
{\sc White, B.}
\newblock The nature of singularities in mean curvature flow of mean-convex
  sets.
\newblock {\em J. Amer. Math. Soc. 16}, 1 (2003), 123--138.

\bibitem{BrianCompactness}
{\sc White, B.}
\newblock On the compactness theorem for embedded minimal surfaces in
  3-manifolds with locally bounded area and genus.
\newblock {\em Comm. Anal. Geom. 26}, 3 (2018), 659--678.

\bibitem{BrianBoundary}
{\sc White, B.}
\newblock Mean curvature flow with boundary.
\newblock {\em Ars Inven. Anal.\/} (2021), Paper No. 4, 43.

\end{thebibliography}
\end{document}